\documentclass[12pt]{amsart}
\usepackage[utf8]{inputenc}
\usepackage[marginratio=1:1,height=9in,width=6.5in]{geometry}
\usepackage{amsfonts, amsmath}
\usepackage{xcolor}
\usepackage{hyperref}
\usepackage{graphicx}
\usepackage{amsthm}
\usepackage{amssymb}
\usepackage{bbm}
\usepackage{mathtools}
\usepackage{latexsym}
\usepackage{listings}
\usepackage{multirow}
\usepackage{enumitem}
\usepackage{siunitx}

\newtheorem{theorem}{Theorem}
\newtheorem{lem}{Lemma}
\newtheorem*{rmk}{Remark}
\newtheorem{proposition}{Proposition}
\newtheorem{assump}{Assumption}

\newcommand{\ve}{\varepsilon}

\newcommand{\EE}{\mathbb{E}}
\newcommand{\RR}{\mathbb{R}}
\newcommand{\PP}{\mathbb{P}}
\newcommand{\ZZ}{\mathbb{Z}}
\newcommand{\mF}{\mathcal{F}}
\newcommand{\mG}{\mathcal{G}}

\newcommand{\mO}{\mathcal{O}}

\newcommand{\mS}{\mathcal{S}}
\newcommand{\mU}{\mathcal{U}}
\newcommand{\mX}{\mathcal{X}}

\newcommand{\rd}{\mathrm{d}}

\newcommand{\bt}{\overline{t}}
\newcommand{\bx}{\overline{x}}
\newcommand{\btx}{(\overline{t}, \overline{x})}

\newcommand{\nbu}{\nabla_u}
\newcommand{\nx}{\nabla_x}
\newcommand{\ny}{\nabla_y}
\newcommand{\pt}{\partial_t}
\newcommand{\ps}{\partial_s}
\newcommand{\pve}{\partial_{\ve}}
\newcommand{\hp}{\widehat{p}}
\newcommand{\hq}{\widehat{q}}
\newcommand{\htk}{\widehat{t}_k}
\newcommand{\hsk}{\widehat{s}_k}
\newcommand{\hxk}{\widehat{x}_k}
\newcommand{\hyk}{\widehat{y}_k}
\newcommand{\xlet}{x^{0,\ve}_t}
\newcommand{\xyet}{x^{1,\ve}_t}
\newcommand{\xeet}{x^{2,\ve}_t}

\newcommand{\pyet}{\phi^{1,\ve}_t}
\newcommand{\peet}{\phi^{2,\ve}_t}
\newcommand{\hPhik}{\widehat{\Phi}_k}

\newcommand{\tp}{^{\top}}
\newcommand{\parentheses}[1]{\left(#1\right)}
\newcommand{\sqbra}[1]{\left[#1\right]}
\newcommand{\curlybra}[1]{\left\{#1\right\}}
\newcommand{\abs}[1]{\left|#1\right|}
\newcommand{\norm}[1]{\left\|#1\right\|}
\newcommand{\normHtwo}[1]{\left\|#1\right\|_{(T;H^2)}}
\newcommand{\inner}[2]{\left\langle#1,\,#2\right\rangle}
\newcommand{\fd}[2]{\dfrac{\delta #1}{\delta #2}}
\newcommand{\pd}[2]{\dfrac{\partial #1}{\partial #2}}
\newcommand{\rom}[1]{\uppercase\expandafter{\romannumeral #1\relax}}

\DeclareMathOperator{\Tr}{Tr}
\DeclareMathOperator{\sign}{sign}
\DeclareMathOperator*{\argmax}{arg\,max}

\title[Policy Gradient for Optimal Control]{A Policy Gradient Framework for Stochastic Optimal Control Problems with Global Convergence Guarantee}

\begin{document}
\maketitle
\begin{center}
  \normalsize
  Mo Zhou \footnote{mo.zhou366@duke.edu} \textsuperscript{$\dagger$}, Jianfeng Lu\footnote{jianfeng@math.duke.edu}\textsuperscript{$\dagger \ddagger$} \par \bigskip

  \textsuperscript{$\dagger$}Department of Mathematics, Duke University \par
  \textsuperscript{$\ddagger$} Department of Physics, and Department of Chemistry, Duke University \par \bigskip

  \date{}
\end{center}

\begin{abstract}%
We consider policy gradient methods for  stochastic optimal control problem in continuous time. In particular, we analyze the gradient flow for the control, viewed as a continuous time limit of the policy gradient method. We prove the global convergence of the gradient flow and establish a convergence rate under some regularity assumptions. The main novelty in the analysis is the notion of local optimal control function, which is introduced to characterize the local optimality of the iterate.

\end{abstract}

\section{Introduction}
The stochastic optimal control problem is an important field of study and has a wide range of applications, such as financial portfolio investment \cite{pham2009continuous}, manufacturing system management \cite{sethi2002optimal}, climate policy decisions \cite{bahn2008stochastic}, disease control and prevention \cite{lee2022mean}, and multiagent path finding \cite{onken2022neural}, just to name a few. It also has natural connections with machine learning, as it can be viewed as a reinforcement learning problem with continuous time \cite{wang2020reinforcement}. 

Given its importance, extensive research has been devoted to solving the optimal control problem. The majority of them fall into two categories: those based on dynamic programming \cite{bertsekas2012dynamic} or Hamilton--Jacobi--Bellman (HJB) equations \cite{bardi1997optimal}, and those utilizing Pontryagin’s maximum principle \cite{kopp1962pontryagin} or method of successive approximations (MSA) \cite{chernousko1982method}. HJB methods based on conventional numerical discretization, such as finite difference \cite{bonnans2003consistency} and finite elements \cite{kushner1990numerical} face difficulties when the dimensionality of the problem becomes high. For deterministic control problems, the method of characteristic lines could be applied to overcome the curse of dimensionality \cite{ruthotto2020machine}, while it becomes difficult for stochastic control, as the stochastic Pontryagin principle involves the backward SDE \cite{pardoux1992backward}, which is hard to solve numerically. 

In recent years, machine learning based methods have emerged as powerful tools to solve the stochastic optimal control problem in high dimension. A comprehensive overview of such methods can be found in \cite{hu2023recent}. In particular, Ji et al. \cite{ji2022solving} apply the stochastic maximum principle with neural network parametrization. Zhou et al. develop an actor-critic framework and score dynamic to solve the HJB equation through deep learning \cite{zhou2021actor, zhou2024deep}. Sirignano et al. \cite{sirignano2018dgm} use a deep Galerkin method to solve for PDEs such the HJB equation. Min and Hu \cite{min2021signatured} proposed a signatured deep fictitious play method for mean field games based on rough path theory. Zang et al. \cite{zang2022machine} propose a machine learning method to solve the optimal landing problem. Deep learning method can also solve optimal control problems with delay \cite{fouque2020deep, han2021recurrent}.

Despite the great empirical success in solving the optimal control problem, the theoretical studies of such algorithms are still lacking. In this work, we analyze the theoretical properties of a policy gradient method for the optimal control problem in a quite general setting. In particular, our analysis covers stochastic optimal control with controlled diffusion (the diffusion coefficient is part of the control), which leads to  a fully nonlinear HJB equation for the value function,  making the analysis significantly more difficult. Moreover, in order to obtain the exact optimal control, we set the control as a deterministic function of time and state, without entropy regularization.
%
We establish the global convergence of the gradient flow (which can be viewed as a continuous time limit of the policy gradient method) under relatively mild regularity assumptions. 
The proof for the convergence of the gradient flow is based on the construction of a barrier function, which is motivated by the uniqueness theory of the viscosity solution to the HJB equation. We also design a local optimal control function in order to distinguish a key criterion to establish a convergence rate. This local optimal function is crucial for the proof of global convergence as the cost functional is non-convex, so that one cannot apply standard tools in convex optimization.

\subsection{Related works}
Before we present the framework and our results, we summarize a few existing theoretical studies on the optimal control problem that are related to ours.

For the analysis of numerical methods for optimal control, many studies focus on specific settings, such as the linear quadratic regulator (LQR) problem, which is  thoroughly studied thanks to its simple structure. Its optimal control has an explicit expression w.r.t. the value function, which makes the HJB equation semi-linear. Several works have been proposed to analyze the global convergence of a policy gradient method for LQR \cite{wang2021global,giegrich2022convergence}. Gobet and Grangereau \cite{gobet2022newton} also develop a Newton's method for control problems with linear dynamics and show quadratic convergence.

Another well-studied scenario is the control problem with a soft policy to encourage exploration. Such global search makes the convergence analysis more feasible \cite{zhou2021curse}.
For instance, in the context of mean-field games, the convergence of policies to the Nash equilibrium is guaranteed under mild assumptions \cite{domingo2020mean, firoozi2022exploratory, guo2022entropy}. Tang et al. \cite{tang2022exploratory} analyze the property of a class of soft policy algorithms where the rate of exploration decreases to $0$.

In more general settings, a variety of recent works focus on studying the convergence of algorithms. Carmona and Lauri{\`e}re \cite{carmona2021convergence, carmona2022convergence} analyze the approximation errors with neural networks for linear quadratic (LQ) mean-field games and general mean-field games. Kerimkulov et al. \cite{kerimkulov2020exponential} study the convergence and stability of a Howard’s policy improvement algorithm. Ito et al. \cite{ito2021neural} investigate an iterative method with a superlinear convergence rate. However, these studies do not cover controlled diffusion. As Yong and Zhou \cite{yong1999stochastic} point out, the presence of control in diffusion will make the control problem significantly more difficult, even in the LQ scenario.

Regarding the optimal control problem with controlled diffusion, Kerimkulov et al. \cite{kerimkulov2021modified} study the convergence rate of a MSA algorithm in a controlled diffusion setting. Sethi et al. \cite{sethi2022modified} propose a modified MSA method, with a proof for convergence. Reisinger et al. \cite{reisinger2022linear} study the condition for linear convergence of policy gradient method, while our analysis works with much weaker assumptions in comparison. 






\subsection{Organization of the paper}
The rest of the paper is organized as follows. The stochastic control problem is introduced in Section \ref{sec:background}. A policy gradient method to solve the problem is proposed in Section \ref{sec:policy_gradient}. In Section \ref{sec:convergence}, we give a convergence analysis of the algorithm, with sketch of the proofs, and the technique details are deferred to the appendix. We conclude our work and mention some directions of future research in Section \ref{sec:future}. 

\section{Theoretical background: the stochastic optimal control problem}\label{sec:background}
We clarify some notations first. We use $|\cdot|$ to denote the absolute value of a scalar, the $l^2$ norm of a vector, the Frobenius norm of a matrix, or the square root of the square sum of a higher order tensor according to the context. $\norm{\cdot}_{L^1}$ and $\norm{\cdot}_{L^2}$ denote the $L^1$ and $L^2$ norms of a function. $\norm{\cdot}_2$ denotes the $l^2$ operator norm (i.e. the largest singular value) of a matrix. $\inner{\cdot}{\cdot}$ denotes the inner product between two vectors, and $\inner{\cdot}{\cdot}_{L^2}$ denotes the inner product between two $L^2$ functions. $\Tr(\cdot)$ denotes the trace of a squared matrix.

In this work, we consider the optimal control problem on state space $\mX$ during a time period $t \in [0,T]$. For simplicity we will assume $\mX$ is an $n$-dimensional unit flat torus, i.e., $\mX = [0,1]^n$ with periodic boundary condition. Without loss of generality, we assume the control variable lies in $\mathbb{R}^{n'}$. 
Let $\left(\Omega, \mF, \{\mF_t\}, \PP \right)$ be a filtered probability space. Let $x_t$ be the state trajectory in $\mX$ satisfying the stochastic differential equation (SDE)
\begin{equation}\label{eq:SDE_X}
\rd x_t = b(x_t, u_t) \rd t + \sigma(x_t, u_t) \rd W_t,
\end{equation}
where $b(x,u): \mX \times \RR^{n'} \to \RR^{n}$ and $\sigma(x,u): \mX \times \RR^{n'} \to \RR^{n \times m}$ are the drift and diffusion coefficients, $u_t \in \RR^{n'}$ is an $\mF_t$-adapted control process, and $W_t$ is an $m$-dimensional $\mF_t$-Brownian motion. With a slight abuse of notation, the letter $u$ may refer to the control process or a vector in $\RR^{n'}$. The initial point $x_0$ is uniformly distributed in $\mX$ unless otherwise specified. Throughout the paper, we assume that the matrix valued function
$$D(x,u) := \frac12 \sigma(x,u) \sigma(x,u)\tp \in \RR^{n \times n}$$
is uniformly elliptic with minimum eigenvalue $\lambda_{\min}(x,u) \ge \sigma_0 > 0$. The goal of the optimal control problem is to minimize the cost functional
\begin{equation}\label{eq:cost}
    J[u] = \EE\sqbra{\int_{0}^{T} r(x_t, u_t) \,\rd t + h(x_T)}
\end{equation}
over all adapted controls, where $r(x,u): \mX \times \RR^{n'} \to \RR$ is the running cost and $h(x): \mX \to \RR$ is the terminal cost. To study the optimal control problem, the value function is very important and useful, which is defined as the expected cost if we start at a certain time and location:
\begin{equation}\label{eq:value}
    V_u(t,x) = \EE\sqbra{\int_{t}^{T} r(x_s, u_s) \,\rd s + h(x_T) ~\Big|~ x_t=x}.
\end{equation}
Here, the subscript $u$ indicates $V_u$ is the value function w.r.t. the control process $u$. By the Markov property, we can verify that $V_u(t, x)$ satisfies the Bellman equation
\begin{equation*}
    V_u(t_1,x) = \EE \sqbra{\int_{t_1}^{t_2} r(x_t, u_t) \,\rd t + V_u(t_2, x_{t_2}) ~\Big|~ x_{t_1} = x}
\end{equation*}
for any $0 \le t_1 \le t_2 \le T$ and $x \in \mX$.

The existence of an optimal control that minimizes \eqref{eq:cost} is well-studied (see e.g., \cite{yong1999stochastic}). In this work, we assume the optimal control exists and we denote $u^*_t$ the optimal control, $x^*_t$ the optimal state process, and $V^*(t, x) = V_{u^*}(t,x)$ the optimal value function.
By the dynamic programming principle \cite{fleming2012deterministic},
\begin{equation*}
    V^*(t_1, x) = \inf_{u} \EE\sqbra{\int_{t_1}^{t_2} r(x_t, u_t) \,\rd t + V^*(t_2, x_{t_2}) ~\Big|~ x_{t_1}=x},
\end{equation*}
where the infimum is taken over all the adapted controls that coincide with $u^*$ in $[0,t_1] \cup [t_2,T]$.
This dynamic programming principle informs that the optimal solution from $t_1$ to $T$ can be obtained if we optimize the control from $t_1$ to $t_2$ w.r.t. the loss $\EE[\int_{t_1}^{t_2} r(x_t, u_t) \,\rd t + V^*(t_2, x_{t_2})]$ and apply $u^*$ after $t_2$.
Based on this principle, let $t_2 \to t_1$, we see that the optimal control $u^*$ at time $t_1$ is a deterministic function of the state $x_{t_1}$, at least heuristically. For a rigorous argument, we refer the reader to the verification theorem (see, for example, \cite[Section 5.3.5]{yong1999stochastic}). Therefore, we will only consider the control as a function of $t$ and $x$ and we will use the shorthand $u_t$ for $u(t, x_t)$ whenever there is no confusion. The objective function becomes
\begin{equation*}
    J[u] = \EE\sqbra{\int_{0}^{T} r(x_t, u(t, x_t)) \,\rd t + h(x_T)},
\end{equation*}
where $u: [0, T] \times \mathcal{X} \to \RR^{n'}$ belongs to a class of admissible control functions, to be specified later.

To simplify notation, we will denote $\rho^{u}(t,x)$ the density of $x_t$ under control $u$,  then $\rho^{u}$ satisfies the Fokker--Planck equation (see e.g., \cite{risken1996fokker})
\begin{equation}\label{eq:FokkerPlanck}
    \partial_t \rho^u(t,x) = -\nx \cdot \sqbra{b(x, u(t,x)) \rho^u(t,x)} + \sum_{i,j=1}^n \partial_i \partial_j \sqbra{D_{ij}(x, u(t,x)) \rho^u(t,x)},
\end{equation}
where we denote $\partial_i = \partial_{x_i}$ for simplicity (recall $D = \frac12 \sigma \sigma\tp$). The initial condition $\rho^u(0,\cdot)$ is the density of $x_0$. For example, $\rho^u(0,\cdot) \equiv 1$ if $x_0 \sim \text{Unif}(\mX)$ and $\rho^u(0,\cdot) = \delta_{x_0}$ if $x_0$ is deterministic. If we denote $\mG_{u}$ the infinitesimal generator of the SDE \eqref{eq:SDE_X} under control $u$, the Fokker--Planck equation can be written as $\partial_t \rho^u = \mG_u^{\dagger} \rho^u$, where $\mG_u^{\dagger}$ is the adjoint operator of $\mG_{u}$. 

One important tool to study the optimal control problem is the adjoint equation. We introduce the adjoint state $p_t = -\nx V_u(t,x_t)$ (also known as the shadow price \cite{aubin1979shadow}) and also $q_t = - \nx^2 V_u(t,x_t) \sigma(x_t, u_t)$. Then, $(p_t, q_t) \in \RR^n \times \RR^{n \times m}$ is the unique solution to the following backward stochastic differential equations (BSDE) \cite{pardoux1992backward}
\begin{equation}\label{eq:adjoint1}
\left\{ \begin{aligned}
    \rd p_t & = -\sqbra{\nx b(x_t,u_t)\tp p_t + \nx \Tr\parentheses{\sigma(x_t,u_t)\tp q_{t}} - \nx r(x_t,u_t)} \rd t + q_t ~ \rd W_t \\
    & = -\nx H\parentheses{t,x_t,u_t,p_t,q_t} \rd t + q_t ~ \rd W_t \\
    p_T & = -\nx h(x_T),
\end{aligned} \right.
\end{equation}
known as the first order adjoint equation. Here,
\begin{equation*}
H(t,x,u,p,q) := \Tr\parentheses{q\tp \sigma(x,u)} + \inner{p}{b(x,u)} - r(x,u)
\end{equation*}
is the Hamiltonian. We also define the generalized Hamiltonian as
\begin{equation}\label{eq:generalized_Hamiltonian}
G(t,x,u,p,P) := \frac12 \Tr\parentheses{P \sigma(x,u) \sigma(x,u)\tp} + \inner{p}{b(x,u)} - r(x,u),
\end{equation}
where $P \in \RR^{n\times n}$ is the argument for the negative Hessian of value functions.
Then, using It\^o's calculus, we can show that the value function $V_u$ in \eqref{eq:value} is the solution to the Hamilton--Jacobi (HJ) equation
\begin{equation}\label{eq:HJ}
\left\{ \begin{aligned}
    -\partial_t V_u(t,x) + G\parentheses{t, x, u(t,x), -\nx V_u(t,x), -\nx^2 V_u(t,x)} & = 0 \\
    V_u(T, x) & = h(x).
\end{aligned} \right.
\end{equation}
Moreover, the optimal value function $V^*$ satisfies the HJB equation
\begin{equation}\label{eq:HJB}
\left\{ \begin{aligned}
    -\partial_t V^*(t,x) + \sup_u G\parentheses{t, x, u, -\nx V^*(t,x), -\nx^2 V^*(t,x)} & = 0 \\
    V^*(T, x) & = h(x).
\end{aligned} \right.
\end{equation}
Therefore, one necessary condition for a control $u$ to be optimal is that 
\begin{equation}\label{eq:max1}
u(t,x) = \argmax_{u' \in \RR^{n'}} G\parentheses{t, x, u', -\nx V_u(t,x), -\nx^2 V_u(t,x)}
\end{equation}
for all $(t,x) \in [0,T] \times \mX$. Note that on the right-hand side of \eqref{eq:max1}, the optimization is only with respect to the third argument of the generalized Hamiltonian (and thus $V_u$ is fixed). In particular, this is a local condition imposed at each $(t, x)$. Since the diffusion function $\sigma$ involves control $u$, this HJB equation is fully nonlinear. The existence and uniqueness of the (viscosity) solution of HJB equation is well established (see \cite{bardi1997optimal} for example). To simplify notation, we will often use the shorthand  $G(t, x, u, -\nx V, -\nx^2 V)$ for $G(t, x, u, -\nx V(t,x), - \nx^2 V(t,x))$ in the rest of the paper.  We define the $H^2([0,T];\mX)$ norm of a value function $V(t,x)$ by
\begin{equation*}
\norm{V}_{(T,H^2)}^2 := \int_0^T \int_{\mX} \parentheses{\abs{V(t, x)}^2 + \abs{\nx V(t, x)}^2 + \abs{\nx^2 V(t, x)}^2} \rd x \, \rd t. 
\end{equation*}

\section{A policy gradient method for the control problem}\label{sec:policy_gradient}

Policy gradient is one of the most popular methods for  reinforcement learning problems \cite{sutton1999policy}. It updates parametrized policy using gradient based method such as gradient descent. This method also applies to the control problem with continuous time \cite{munos2006policy}, which can be viewed as a instantaneous but local dynamic programming method through gradient descent. In order to design a policy gradient method, we present the following proposition, which gives an explicit expression for the derivative of the cost functional \eqref{eq:cost} w.r.t. the control function.
\begin{proposition}\label{prop:cost_derivative}
Let $u$ be a control function and $V_u$ be the corresponding value function. Let the state process $x_t$ start with uniform distribution on $\mX$ and follow the SDE \eqref{eq:SDE_X} with control $u$. Then we have
\begin{equation}\label{eq:actor_derivative}
\fd{J}{u}(t, x)  = - \rho^u(t,x) ~\nbu G\parentheses{t, x, u(t, x), -\nx V_u, -\nx^2 V_u },
\end{equation}
where $\fd{J}{u}$ denotes the functional derivative of $J$ w.r.t.{} $u$, and $\nbu G$ denotes the gradient of $G$ w.r.t.{} its third argument (as a vector in $\RR^{n'}$); furthermore as a generalization,
\begin{equation}\label{eq:dVdu}
\fd{V_u(s,y)}{u}(t, x)  = - \mathbbm{1}_{\{t \ge s\}} \, p^u(t,x;s,y) ~\nbu G\parentheses{t, x, u(t, x), -\nx V_u, -\nx^2 V_u }
\end{equation}
for all $x,y \in \mX$, where $p^u(t,x;s,y)$ is the fundamental solution to the Fokker--Planck equation \eqref{eq:FokkerPlanck}. 
\end{proposition}

\begin{rmk}
We remark that a similar result is established in \cite[Proposition 2]{carmona2021convergence} under a different setting. That work considers infinite horizon mean-field games with constant diffusion coefficient, while our setting is finite horizon stochastic control problem with controlled diffusion.
\end{rmk}

Motivated by Proposition \ref{prop:cost_derivative},
we consider a continuous-time version of the policy gradient, i.e., the gradient flow of $J$ with respect to $u$. Let us recall the gradient descent method in discrete time first. Using $\tau_k$ as time steps for the policy gradient, for any $(t,x) \in [0,T] \times \mX$, the gradient descent method is 
\begin{equation}\label{eq:PG_discrete}
\begin{aligned}
u^{\tau_{k+1}}(t,x) & = u^{\tau_k}(t,x) - \Delta \tau_k \fd{J}{u^{\tau_k}}(t, x) \\
& = u^{\tau_k}(t,x) + \Delta \tau_k ~ \rho^{u^{\tau_k}}(t,x) ~\nbu G\parentheses{t, x, u^{\tau_k}(t, x), -\nx V_{u^{\tau_k}}, -\nx^2 V_{u^{\tau_k}}},
\end{aligned}
\end{equation}
where $\Delta \tau_k = \tau_{k+1} - \tau_k$ is the step size and $V_{u^{\tau_k}}$ is the value function w.r.t. the control $u^{\tau_k}$.
Directly implementing the gradient descent \eqref{eq:PG_discrete} is expensive, as it would involve solving the FP equation for $\rho^{u^{\tau_k}}$. In practice, therefore, we can instead consider a stochastic version of gradient descent that updates the control function through
\begin{equation}\label{eq:PG_discrete_sgd}
\begin{aligned}
u^{\tau_{k+1}}(t,X_t) 
& = u^{\tau_k}(t,X_t) + \Delta \tau_k \nbu G\parentheses{t, X_t, u^{\tau_k}(t, X_t), -\nx V_{u^{\tau_k}}, -\nx^2 V_{u^{\tau_k}}},
\end{aligned}
\end{equation}
where $X_t \sim \rho^{u^{\tau_k}}(t, \cdot)$. Note that on average the update of the stochastic gradient is the same with \eqref{eq:PG_discrete}.
Then, trajectories for $X_t \sim \rho^{u^{\tau_k}}(t,\cdot)$ could be sampled using Euler-Maruyama scheme
\begin{equation}\label{eq:EulerMaruyama}
x_{i+1} = x_i + b(x_i, u^{\tau_k}(t_i,x_i)) \Delta t + \sigma(x_i, u^{\tau_k}(t_i,x_i)) \sqrt{\Delta t} \, \xi
\end{equation}
with $x_0 \sim \text{Unif}(\mX)$ and $\xi \sim N(0,I_m)$ to approximate the SDE \eqref{eq:SDE_X} numerically. Here $\Delta t = T / N$ is the time step size and $t_i = i \Delta t$ ($i=0,1,\ldots,N$). 
Multiple numerical trajectories can be sampled from \eqref{eq:EulerMaruyama}, which provides a set
$$\mS_k := \{(t_i,x_i^{(j)}) ~|~ j \ge 1, ~ i=0,1,\ldots N-1 \},$$
where $j$ is the index of samples.
In practice, under some parametrization of the control function $u(t,x;\theta)$,\footnote{This parametrization could be grid, finite element, or neural network (see \cite{han2020solving} for an example of neural network with periodic boundary condition).} where $\theta$ denotes the collective parameters, the stochastic update \eqref{eq:PG_discrete_sgd} can be further approximated through a lease square fitting
$$\min_{\theta_{k+1}} \sum_{(t,x) \in \mS_k} \abs{u(t,x;\theta_{k+1}) - u(t,x;\theta_k) - \Delta \tau_k \, \nbu G(t,x,u_k(t,x),-\nx V_{u_k}, -\nx^2 V_{u_k})}^2,$$
where $u_k := u(\cdot,\cdot;\theta_k)$.


There are several interesting topics arising from this scheme, including
\begin{enumerate}
\item 
The consistency between the algorithm above and the continuous gradient flow to be introduced next.
\item The theoretical analysis for the neural network parametrization, i.e., the approximation error. 
\item Designing practical numerical experiments to train and optimize the neural network.
\item Conducting theoretical analysis for the gradient flow with continuous time.
\end{enumerate}
\smallskip 

In this work, we will focus on theoretical analysis for the gradient flow of functions in continuous time and leave other important topics for future research. Let us return to the continuous dynamic. Let $\Delta \tau = \max_k \Delta \tau_k \to 0$, the update scheme of the control \eqref{eq:PG_discrete} converges to the gradient flow in continuous time
\begin{equation}\label{eq:PG_ideal}
\dfrac{\rd}{\rd \tau} u^{\tau}(t, x) = - \fd{J}{u^{\tau}}(t, x) = \rho^{u^{\tau}}(t,x) ~\nbu G\parentheses{t, x, u^{\tau}(t, x), -\nx V_{u^{\tau}}, -\nx^2 V_{u^{\tau}}}.
\end{equation}
By the definition of value function \eqref{eq:value} and Proposition \ref{prop:cost_derivative}, we can show that $V_{u^{\tau}}(t,x)$ is decreasing in $\tau$ (for fixed $t$ and $x$).

\begin{proposition}[Monotonicity of value function in $\tau$]\label{prop:value_decreasing}
Under the policy gradient dynamic \eqref{eq:PG_ideal}, the value function $V_{u^{\tau}}(t,x)$ is decreasing in $\tau$ for all $(t,x) \in [0,T] \times \mX$.
\end{proposition}

While our main focus for the paper is the analysis of the policy gradient flow \eqref{eq:PG_ideal}, we remark that a practical algorithm also requires the derivatives of value function $\nx V_u$ and $\nx^2 V_u$ for a given control, which reduces to solving the linear parabolic HJ equation \eqref{eq:HJ}. Computing the value function with given policy is called policy evaluation  in reinforcement learning \cite{heckman1992policy}. Such structure naturally motivates us to use the actor-critic framework \cite{konda1999actor}, where the policy gradient and policy evaluation are operated jointly. In our setting, the gradient flow \eqref{eq:PG_ideal} can be viewed as the limit of a two time-scale actor-critic method \cite{wu2020finite}, where the speed of policy evaluation is infinitely faster than policy gradient. While it is possible to extend our analysis to the actor-critic method, we choose to focus on the dynamic \eqref{eq:PG_ideal} in this paper and leave that to future works.

\section{Convergence of the policy gradient}\label{sec:convergence}
In this section, we give the main results of this work on convergence of the policy gradient. First, we give some technical assumptions
\begin{assump}\label{assump:basic}
Assume the followings hold.
\begin{enumerate}
\item $r$, $b$, and $\sigma$ are smooth, with $C^{4}(\mX \times \RR^{n'})$ norm and fourth order H\"older norm bounded by some constant $K$. 
\item $h$ is smooth with $C^{4}(\mX)$ norm and fourth order H\"older norm bounded by $K$.
\item There exists a unique solution $V^* \in C^{1,2}([0,T]; \mX)$ to the HJB equation \eqref{eq:HJB} in the classical sense. The optimal control function $u^*$ is smooth and 
$$\norm{u^*}_{C^{2,4}([0,T];\mX)} \le K.$$
\end{enumerate}
\end{assump}

In order to avoid tedious technicality and focus on the main ideas of the analysis, we also make a regularity assumption on the control function through the gradient flow \eqref{eq:PG_ideal}.
\begin{assump}\label{assump:u_smooth}
The control function $u^{\tau}$ remains smooth through the gradient flow \eqref{eq:PG_ideal}, and there exists a constant $K$ such that $\norm{u^{\tau}}_{C^{2,4}([0,T]; \mX)} \le K$.
\end{assump}

While smoothness is inherently preserved by the policy gradient dynamic, the boundedness is a technical assumption that does not directly follow from standard energy estimates. In Theorem \ref{thm:actor_rate}, we will show a convergence rate for the policy gradient flow through establishing a Polyak-{\L}ojasiewicz (PL) condition \cite{karimi2016linear}, which means that the policy gradient dynamics are guided towards the optimal control, effectively avoiding regions that could lead to unbounded behavior.

Let us define a set 
$$\mU = \curlybra{u(t,x) ~\Big|~ u \text{ is smooth and} \norm{u}_{C^{2,4}([0,T]; \mX)} \le K} $$
to include all the regular control functions we consider. 
We make a few remarks about the assumptions.
\begin{enumerate}
\item[-] It follows from definition that $D$ is also smooth, and we also use $K$ to denote its $C^{4}(\mX \times \RR^{n'})$ bound. Since the control function $u(t,x) \in \mU$ is bounded, we just need $r,b,\sigma,D$ has bounded derivative when the input vector $u \in \RR^{n'}$ is within this bounded range. Similar boundedness assumptions are very common, such as in \cite{vsivska2020gradient}.
\item[-] When Assumption \ref{assump:basic} holds and $u \in \mU$, we have solution $V_u \in C^{1,2}([0,T]; \mX)$ to the HJ equation \eqref{eq:HJ} in classical sense.
\item[-] Regarding the third condition in Assumption 1, the existence and uniqueness of viscosity solution is well established under mild assumption \cite{bardi1997optimal}. For the boundedness of $\norm{V^*}_{C^{1,2}([0,T],\mX)}$, see e.g.{} \cite[Theorem 5.3]{mou2019remarks}. This smoothness further implies that $V^*$ is a classical solution to the HJB equation.

\item[-] When Assumption \ref{assump:basic} holds and $u \in \mU$, we know from Schauder estimate \cite{ladyvzenskaja1988linear} that $V_u$ has fourth order derivative in $x$. So, $V_u \in C^{1,4}([0,T]; \mX)$. Then, we observe that $G\parentheses{t, x, u(t,x), -\nx V_u(t,x), -\nx^2 V_u(t,x)}$ in \eqref{eq:HJ} is differentiable in $t$, which implies $V_u \in  C^{2,4}([0,T]; \mX)$. We will also use $K$ to denote the bound for $C^{2,4}([0,T]; \mX)$ norm of $V_u$. The same argument also holds for $V^*$, so $\norm{V^*}_{C^{2,4}} \le K$. 

\item[-] We assume boundedness of $r$, $b$, and $\sigma$ in Assumption \ref{assump:basic} as the state space $\mX$ is compact. In the setting with unbounded state space such as $\RR^n$, the common assumption is that the functions grows at most linearly in $\abs{x}$ \cite{yong1999stochastic}.

\item[-] In Assumption \ref{assump:basic}, the bounded derivative implies a Lipschitz condition. For example, if $\abs{\nx b} \le K$, then $\abs{b(x_1 ,u) - b(x_2 ,u)} \le K \abs{x_1 - x_2}$. In the proofs, we will use $L$ instead of $K$ whenever we use the Lipschitz condition, in order to be more reader friendly.

\item[-] The assumptions are weaker than those of \cite{reisinger2022linear}, where it is required that the running cost is sufficiently convex or the time span $[0,T]$ is sufficiently short.
\end{enumerate}

\smallskip 

Under these assumptions, we have a lower bound and an upper bound for the density function on the compact space $\mX$.
\begin{proposition}\label{prop:rho}
Under Assumption \ref{assump:basic}, let $u \in \mU$ and $\rho^u$ be the solution to the Fokker--Planck equation \eqref{eq:FokkerPlanck} with initial condition $\rho^u(0,x) \equiv 1$. Then $\rho^u(t,x)$ has a positive lower bound $\rho_0$ and an upper bound $\rho_1$ that only depend on $n$, $T$, and $K$.
\end{proposition}

The proof of this proposition can be found in Appendix \ref{sec:props}.

We now state convergence results of the gradient flow \eqref{eq:PG_ideal}. We give a warm-up theorem about its critical point.
\begin{theorem}[Critical point for policy gradient]
Under Assumption \ref{assump:basic}, assume further that $G$ is strongly concave in $u$. Then, any critical point of the gradient flow \eqref{eq:PG_ideal} is the optimal control.
\end{theorem}
Similar to the remark under Assumptions \ref{assump:basic} and \ref{assump:u_smooth} above, it would be sufficient to only require the concavity of $G$ when $p = -\nx V_u$ and $P = -\nx^2 V_u$ are within a bounded range given by the Schauder estimates.
\begin{proof}
Let $u$ be a critical point of \eqref{eq:PG_ideal}. Then since $\rho^u$ is not vanishing by Proposition \ref{prop:rho} (lower bound), we have
$$\nbu G\parentheses{t, x, u(t, x), -\nx V_u, -\nx^2 V_u} = 0.$$
Since $G$ is strongly concave in $u$, the control function $u(t,x)$ satisfies the maximum condition \eqref{eq:max1}. Therefore, $V_u$ is not only the solution to the HJ equation \eqref{eq:HJ}, but also the solution to the HJB equation \eqref{eq:HJB}. Then, by the uniqueness of HJB equation, $u(t,x)$ is the optimal control.
\end{proof}

In order to explain the necessity of the strong concavity assumption of $G$ in $u$, we give a counter-example in Appendix \ref{sec:counter_example}, showing that there are multiple critical points of the policy gradient dynamics \eqref{eq:PG_ideal} when the concavity assumption does not hold. It is also clear that the commonly studied LQR problem satisfies this strong concavity assumption \cite{wang2021global}. We shall emphasize that this concavity assumption does not imply that the optimal control problem \eqref{eq:cost} is convex in $u$. In fact the cost functional is in general non-convex.

Next, we state our main results, establishing the convergence guarantee of  gradient flow \eqref{eq:PG_ideal}.
\begin{theorem}[Convergence of the policy gradient]\label{thm:actor_converge}
Let Assumptions \ref{assump:basic} and \ref{assump:u_smooth} hold. Further assume that $G$ is uniformly strongly concave in $u$. Then, the gradient flow \eqref{eq:PG_ideal} of control satisfies
\begin{equation}\label{eq:actor_convergence}
\lim_{\tau \to \infty} J[u^{\tau}] = J[u^*].
\end{equation}
\end{theorem}

Here, by uniformly strongly concave, we mean that there exists an absolute constant $\mu_G > 0$ such that the family of functions $G(t,x,\cdot,p,P)$ is $\mu_G-$strongly concave for all $(t,x,p,P)$ within the range given by the Schauder estimates above.
Given Theorem \ref{thm:actor_converge}, one natural question is whether one can establish a convergence rate for the policy gradient. The answer is yes, with a mild extra assumption to avoid flatness.
\begin{assump}\label{assump:actor_rate}
There exists a modulus of continuity $\omega: [0,\infty) \to [0,\infty)$ such that
$$\norm{u-u^*}_{L^2} \le \omega(J[u] - J[u^*])$$
for any $u \in \mU$. Here $u^*$ is the optimal control.
\end{assump}

With this assumption, Theorem \ref{thm:actor_converge} guarantees us that $\norm{u^{\tau}-u^*}_{L^2} \to 0$ as $\tau \to \infty$. Therefore, we just have to analyze when $u^\tau$ is sufficient close to $u^*$ in order to get a global convergence rate. 
\begin{theorem}[Convergence rate of the policy gradient]\label{thm:actor_rate}
Let Assumptions \ref{assump:basic}, \ref{assump:u_smooth}, and \ref{assump:actor_rate} hold. Further assume that $G$ is uniformly strongly concave in $u$. Then, the gradient flow \eqref{eq:PG_ideal} of control satisfies
\begin{equation}\label{eq:actor_rate}
J[u^{\tau}] - J[u^*] \le e^{-c \tau} \parentheses{J[u^0] - J[u^*]}
\end{equation}
for some positive constant $c$. This $c$ depends on $n$, $m$, $n'$, $K$, and $\omega$. As a direct corollary, $\norm{u^{\tau}-u^*}_{L^2} \to 0$. 
\end{theorem}

The proofs for Theorems \ref{thm:actor_converge} and \ref{thm:actor_rate} overlap quite a bit, so we will prove them together. We present the key idea of the proof here and leave a detailed version to Appendix \ref{sec:thms}. Throughout all the proofs, we will use $C$ to denote some absolute constant that only depends on $n, K, T$, which may change depending on the context.
\begin{proof}[Key ideas of proofs to Theorems \ref{thm:actor_converge} and \ref{thm:actor_rate}]
Establishing the 
PL condition
is one typical way to show the convergence of gradient descent (or gradient flow). In order to distinguish this condition, we make a technical definition. 

For a control function $u(t,x)$, we define a corresponding ``local optimal'' control function as
\begin{equation}\label{eq:udiamond_sketch}
u^{\diamond}(t,x) := \argmax_{u' \in \RR^{n'}} G(t,x,u',-\nx V_u(t,x), -\nx^2 V_u(t,x)).
\end{equation}
We call $u^{\diamond}$ ``local optimal'' because 
it is similar to the (global) optimal condition \eqref{eq:max1}, while we have $V_u$ instead of $V_{u^{\diamond}}$ on the right-hand side.
So, $\abs{u(t,x) - u^{\diamond}(t,x)}$ can measure the distance between the current control $u$ and the local optimal one at $(t,x)$. For each control function $u^{\tau}$ along the gradient flow \eqref{eq:PG_ideal}, we thus define a corresponding ``local optimal'' control function $u^{\tau \diamond}$.

Accordingly, we separate into two scenarios through a condition (see \eqref{eq:easy_case_sketch}). Indeed, the local optimal control function \eqref{eq:udiamond_sketch} is introduced to establish a condition for $u^{\tau}$ to distinguish the two scenarios.

The first case is when there exist positive constants $\mu_0$ and $\tau_0$ such that
\begin{equation}\label{eq:easy_case_sketch}
\norm{u^{\tau} - u^{\tau \diamond}}_{L^2} \ge \mu_0 \norm{u^{\tau} - u^*}_{L^2}
\end{equation}
for all $\tau \ge \tau_0$. This assumption directly implies the PL condition, and thus the convergence analysis follows easily. The non-trivial task is of course to analyze the scenario when this assumption does not hold, which consists of the main technical work. 

When assumption \eqref{eq:easy_case_sketch} does not hold, we can find a sequence $\{\tau_k\}$, increasing to infinity, such that for each $k$
\begin{equation*}
\norm{u^{\tau_k} - u^{\tau_k\diamond}}_{L^2} \le \frac1k \norm{u^{\tau_k} - u^*}_{L^2}.
\end{equation*}
Denoting $u^{\tau_k}$ and $V_{u^{\tau_k}}$ by $u_k$ and $V_k$, we rewrite the above as
\begin{equation}\label{eq:hard_case_sketch}
\norm{u_k - u_k^{\diamond}}_{L^2} \le \frac1k \norm{u_k - u^*}_{L^2}.
\end{equation}
By Proposition \ref{prop:value_decreasing}, the value function $V_k(t,x)$ is decreasing in $k$, so it has a pointwise limit $V_{\infty}(t,x)$. Then we compare the maximum condition and the definition of the local optimal control
\begin{align*}
u^*(t,x) &= \argmax_{u \in \RR^{n'}} G\parentheses{t, x, u, -\nx V^*(t,x), -\nx^2 V^*(t,x)}\\
u^{\diamond}_k(t,x) &= \argmax_{u \in \RR^{n'}} G(t,x,u,-\nx V_k(t,x), -\nx^2 V_k(t,x))
\end{align*} 
and observe that $V_k$ is very ``close'' to the solution of the HJB equation $V^*$ in the sense that $u_k$ is very close to $u_k^{\diamond}$ (see \eqref{eq:hard_case_sketch}), i.e., the maximum condition is ``nearly'' satisfied. Therefore, the key idea is to modify the proof for the uniqueness of HJB equation and show 
\begin{equation}\label{eq:claim_sketch}
V_{\infty}(t,x) \equiv V^*(t,x).
\end{equation}
The rest of the proof will come naturally after \eqref{eq:claim_sketch}.
\end{proof}

\section{Conclusion and future directions}\label{sec:future}

In conclusion, we study the stochastic optimal control problem with controlled diffusion in continuous time. We propose a policy gradient framework to solve the problem, where the control dynamic follows the gradient flow of the cost functional \eqref{eq:PG_ideal}. We design a local optimal control function to analyze the convergence property of the algorithm and prove that the algorithm converges to the optimum under some regularity assumptions.

Our analysis can be extended in several directions. In this work, we concentrate on the time-invariant optimal control problem. It should not be too difficult to extend to more general time-dependent scenarios, although it may require additional regularity assumptions. 

Regarding regularity, we have focused on the classical solutions to the HJ and HJB equations. It is natural to ask about viscosity solutions with less stringent regularity assumptions, which is an important future research direction.

As already mentioned in Section \ref{sec:policy_gradient}, our setting can be viewed as a limiting case under the actor-critic framework with two time scales. It is of interest to establish the full convergence of the actor-critic method with the critic (policy evaluation) dynamics included. It is also interesting to extend the analysis to a single time-scale actor-critic method \cite{zhou2023single}, which couples together the control dynamics with policy evaluation \cite{zhou2024solving}. We will leave these for future works.

\smallskip
\noindent \textbf{Acknowledgement.} The work is supported in part by National Science Foundation via the grant DMS-2012286.


\bibliographystyle{plain}
\bibliography{ref.bib}

\begin{thebibliography}{10}

\bibitem{aronson1959fundamental}
DG~Aronson.
\newblock The fundamental solution of a linear parabolic equation containing a
  small parameter.
\newblock {\em Illinois Journal of Mathematics}, 3(4):580--619, 1959.

\bibitem{aronson1967bounds}
Donald~Gary Aronson.
\newblock Bounds for the fundamental solution of a parabolic equation.
\newblock {\em Bulletin of the American Mathematical society}, 73(6):890--896,
  1967.

\bibitem{aubin1979shadow}
Jean-Pierre Aubin and FH~Clarke.
\newblock Shadow prices and duality for a class of optimal control problems.
\newblock {\em SIAM Journal on Control and Optimization}, 17(5):567--586, 1979.

\bibitem{bahn2008stochastic}
Olivier Bahn, Alain Haurie, and Roland Malham{\'e}.
\newblock A stochastic control model for optimal timing of climate policies.
\newblock {\em Automatica}, 44(6):1545--1558, 2008.

\bibitem{bardi1997optimal}
Martino Bardi and Italo~Capuzzo Dolcetta.
\newblock {\em Optimal control and viscosity solutions of
  Hamilton-Jacobi-Bellman equations}, volume~12.
\newblock Springer, 1997.

\bibitem{bertsekas2012dynamic}
Dimitri Bertsekas.
\newblock {\em Dynamic programming and optimal control: Volume I}, volume~1.
\newblock Athena scientific, 2012.

\bibitem{bonnans2003consistency}
J~Fr{\'e}d{\'e}ric Bonnans and Housnaa Zidani.
\newblock Consistency of generalized finite difference schemes for the
  stochastic hjb equation.
\newblock {\em SIAM Journal on Numerical Analysis}, 41(3):1008--1021, 2003.

\bibitem{carmona2021convergence}
Ren{\'e} Carmona and Mathieu Lauri{\`e}re.
\newblock Convergence analysis of machine learning algorithms for the numerical
  solution of mean field control and games i: the ergodic case.
\newblock {\em SIAM Journal on Numerical Analysis}, 59(3):1455--1485, 2021.

\bibitem{carmona2022convergence}
Ren{\'e} Carmona and Mathieu Lauri{\`e}re.
\newblock Convergence analysis of machine learning algorithms for the numerical
  solution of mean field control and games: Ii—the finite horizon case.
\newblock {\em The Annals of Applied Probability}, 32(6):4065--4105, 2022.

\bibitem{chernousko1982method}
Felix~L Chernousko and AA~Lyubushin.
\newblock Method of successive approximations for solution of optimal control
  problems.
\newblock {\em Optimal Control Applications and Methods}, 3(2):101--114, 1982.

\bibitem{domingo2020mean}
Carles Domingo-Enrich, Samy Jelassi, Arthur Mensch, Grant Rotskoff, and Joan
  Bruna.
\newblock A mean-field analysis of two-player zero-sum games.
\newblock {\em Advances in {N}eural {I}nformation {P}rocessing {S}ystems},
  33:20215--20226, 2020.

\bibitem{firoozi2022exploratory}
Dena Firoozi and Sebastian Jaimungal.
\newblock Exploratory {LQG} mean field games with entropy regularization.
\newblock {\em Automatica}, 139:110177, 2022.

\bibitem{fleming2012deterministic}
Wendell~H Fleming and Raymond~W Rishel.
\newblock {\em Deterministic and stochastic optimal control}, volume~1.
\newblock Springer Science \& Business Media, 2012.

\bibitem{fouque2020deep}
Jean-Pierre Fouque and Zhaoyu Zhang.
\newblock Deep learning methods for mean field control problems with delay.
\newblock {\em Frontiers in Applied Mathematics and Statistics}, 6:11, 2020.

\bibitem{friedman2008partial}
Avner Friedman.
\newblock {\em Partial differential equations of parabolic type}.
\newblock Courier Dover Publications, 2008.

\bibitem{giegrich2022convergence}
Michael Giegrich, Christoph Reisinger, and Yufei Zhang.
\newblock Convergence of policy gradient methods for finite-horizon stochastic
  linear-quadratic control problems.
\newblock {\em arXiv preprint arXiv:2211.00617}, 2022.

\bibitem{gobet2022newton}
Emmanuel Gobet and Maxime Grangereau.
\newblock Newton method for stochastic control problems.
\newblock {\em SIAM Journal on Control and Optimization}, 60(5):2996--3025,
  2022.

\bibitem{guo2022entropy}
Xin Guo, Renyuan Xu, and Thaleia Zariphopoulou.
\newblock Entropy regularization for mean field games with learning.
\newblock {\em Mathematics of Operations research}, 47(4):3239--3260, 2022.

\bibitem{han2021recurrent}
Jiequn Han and Ruimeng Hu.
\newblock Recurrent neural networks for stochastic control problems with delay.
\newblock {\em Mathematics of Control, Signals, and Systems}, 33:775--795,
  2021.

\bibitem{han2020solving}
Jiequn Han, Jianfeng Lu, and Mo~Zhou.
\newblock Solving high-dimensional eigenvalue problems using deep neural
  networks: A diffusion monte carlo like approach.
\newblock {\em Journal of Computational Physics}, 423:109792, 2020.

\bibitem{heckman1992policy}
James~J Heckman.
\newblock Policy evaluation.
\newblock {\em Evaluating welfare and training programs}, page 201, 1992.

\bibitem{hu2023recent}
Ruimeng Hu and Mathieu Lauriere.
\newblock Recent developments in machine learning methods for stochastic
  control and games.
\newblock {\em arXiv preprint arXiv:2303.10257}, 2023.

\bibitem{ito2021neural}
Kazufumi Ito, Christoph Reisinger, and Yufei Zhang.
\newblock A neural network-based policy iteration algorithm with global h
  2-superlinear convergence for stochastic games on domains.
\newblock {\em Foundations of Computational Mathematics}, 21(2):331--374, 2021.

\bibitem{ito1953fundamental}
Seiz{\^o} It{\^o}.
\newblock The fundamental solution of the parabolic equation in a
  differentiable manifold.
\newblock {\em Osaka Mathematical Journal}, 5(1):75--92, 1953.

\bibitem{ji2022solving}
Shaolin Ji, Shige Peng, Ying Peng, and Xichuan Zhang.
\newblock Solving stochastic optimal control problem via stochastic maximum
  principle with deep learning method.
\newblock {\em Journal of Scientific Computing}, 93(1):30, 2022.

\bibitem{karimi2016linear}
Hamed Karimi, Julie Nutini, and Mark Schmidt.
\newblock Linear convergence of gradient and proximal-gradient methods under
  the {P}olyak-{\l}ojasiewicz condition.
\newblock In {\em Machine Learning and Knowledge Discovery in Databases:
  European Conference, ECML PKDD 2016, Riva del Garda, Italy, September 19-23,
  2016, Proceedings, Part I 16}, pages 795--811. Springer, 2016.

\bibitem{kerimkulov2020exponential}
Bekzhan Kerimkulov, David Siska, and Lukasz Szpruch.
\newblock Exponential convergence and stability of {H}oward's policy
  improvement algorithm for controlled diffusions.
\newblock {\em SIAM Journal on Control and Optimization}, 58(3):1314--1340,
  2020.

\bibitem{kerimkulov2021modified}
Bekzhan Kerimkulov, David {\v{S}}i{\v{s}}ka, and Lukasz Szpruch.
\newblock A modified {MSA} for stochastic control problems.
\newblock {\em Applied Mathematics \& Optimization}, 84(3):3417--3436, 2021.

\bibitem{konda1999actor}
Vijay Konda and John Tsitsiklis.
\newblock Actor-critic algorithms.
\newblock {\em Advances in {N}eural {I}nformation {P}rocessing {S}ystems}, 12,
  1999.

\bibitem{kopp1962pontryagin}
Richard~E Kopp.
\newblock Pontryagin maximum principle.
\newblock In {\em Mathematics in Science and Engineering}, volume~5, pages
  255--279. Elsevier, 1962.

\bibitem{kushner1990numerical}
Harold~J Kushner.
\newblock Numerical methods for stochastic control problems in continuous time.
\newblock {\em SIAM Journal on Control and Optimization}, 28(5):999--1048,
  1990.

\bibitem{ladyvzenskaja1988linear}
Olga~A Ladyzenskaja, Vsevolod~Alekseevich Solonnikov, and Nina~N Uralceva.
\newblock {\em Linear and quasi-linear equations of parabolic type}, volume~23.
\newblock American Mathematical Soc., 1988.

\bibitem{lee2022mean}
Wonjun Lee, Siting Liu, Wuchen Li, and Stanley Osher.
\newblock Mean field control problems for vaccine distribution.
\newblock {\em Research in the Mathematical Sciences}, 9(3):51, 2022.

\bibitem{min2021signatured}
Ming Min and Ruimeng Hu.
\newblock Signatured deep fictitious play for mean field games with common
  noise.
\newblock {\em arXiv preprint arXiv:2106.03272}, 2021.

\bibitem{mou2019remarks}
Chenchen Mou.
\newblock Remarks on schauder estimates and existence of classical solutions
  for a class of uniformly parabolic {H}amilton--{J}acobi--{B}ellman
  integro-{PDE}s.
\newblock {\em Journal of Dynamics and Differential Equations}, 31(2):719--743,
  2019.

\bibitem{munos2006policy}
R{\'e}mi Munos.
\newblock Policy gradient in continuous time.
\newblock {\em Journal of Machine Learning Research}, 7:771--791, 2006.

\bibitem{onken2022neural}
Derek Onken, Levon Nurbekyan, Xingjian Li, Samy~Wu Fung, Stanley Osher, and
  Lars Ruthotto.
\newblock A neural network approach for high-dimensional optimal control
  applied to multiagent path finding.
\newblock {\em IEEE Transactions on Control Systems Technology},
  31(1):235--251, 2022.

\bibitem{pardoux1992backward}
Etienne Pardoux and Shige Peng.
\newblock Backward stochastic differential equations and quasilinear parabolic
  partial differential equations.
\newblock In {\em Stochastic partial differential equations and their
  applications}, pages 200--217. Springer, 1992.

\bibitem{peng1990general}
Shige Peng.
\newblock A general stochastic maximum principle for optimal control problems.
\newblock {\em SIAM Journal on control and optimization}, 28(4):966--979, 1990.

\bibitem{pham2009continuous}
Huy{\^e}n Pham.
\newblock {\em Continuous-time stochastic control and optimization with
  financial applications}, volume~61.
\newblock Springer Science \& Business Media, 2009.

\bibitem{reisinger2022linear}
Christoph Reisinger, Wolfgang Stockinger, and Yufei Zhang.
\newblock Linear convergence of a policy gradient method for finite horizon
  continuous time stochastic control problems.
\newblock {\em arXiv preprint arXiv:2203.11758}, 2022.

\bibitem{risken1996fokker}
Hannes Risken.
\newblock Fokker-planck equation.
\newblock In {\em The Fokker-Planck Equation}, pages 63--95. Springer, 1996.

\bibitem{ruthotto2020machine}
Lars Ruthotto, Stanley~J Osher, Wuchen Li, Levon Nurbekyan, and Samy~Wu Fung.
\newblock A machine learning framework for solving high-dimensional mean field
  game and mean field control problems.
\newblock {\em Proceedings of the National Academy of Sciences},
  117(17):9183--9193, 2020.

\bibitem{sethi2022modified}
Deven Sethi and David {\v{S}}i{\v{s}}ka.
\newblock The modified {MSA}, a gradient flow and convergence.
\newblock {\em arXiv preprint arXiv:2212.05784}, 2022.

\bibitem{sethi2002optimal}
Suresh~P Sethi, Houmin Yan, Hanqin Zhang, and Qing Zhang.
\newblock Optimal and hierarchical controls in dynamic stochastic manufacturing
  systems: A survey.
\newblock {\em Manufacturing \& Service Operations Management}, 4(2):133--170,
  2002.

\bibitem{sirignano2018dgm}
Justin Sirignano and Konstantinos Spiliopoulos.
\newblock Dgm: A deep learning algorithm for solving partial differential
  equations.
\newblock {\em Journal of computational physics}, 375:1339--1364, 2018.

\bibitem{vsivska2020gradient}
David {\v{S}}i{\v{s}}ka and {\L}ukasz Szpruch.
\newblock Gradient flows for regularized stochastic control problems.
\newblock {\em arXiv preprint arXiv:2006.05956}, 2020.

\bibitem{sutton1999policy}
Richard~S Sutton, David McAllester, Satinder Singh, and Yishay Mansour.
\newblock Policy gradient methods for reinforcement learning with function
  approximation.
\newblock {\em Advances in {N}eural {I}nformation {P}rocessing {S}ystems}, 12,
  1999.

\bibitem{tang2022exploratory}
Wenpin Tang, Yuming~Paul Zhang, and Xun~Yu Zhou.
\newblock Exploratory {HJB} equations and their convergence.
\newblock {\em SIAM Journal on Control and Optimization}, 60(6):3191--3216,
  2022.

\bibitem{taylor2012partial}
Michael~Eugene Taylor.
\newblock {\em Partial differential equations. 3, Nonlinear equations}.
\newblock Springer, 2012.

\bibitem{wang2020reinforcement}
Haoran Wang, Thaleia Zariphopoulou, and Xun~Yu Zhou.
\newblock Reinforcement learning in continuous time and space: A stochastic
  control approach.
\newblock {\em J. Mach. Learn. Res.}, 21(198):1--34, 2020.

\bibitem{wang2021global}
Weichen Wang, Jiequn Han, Zhuoran Yang, and Zhaoran Wang.
\newblock Global convergence of policy gradient for linear-quadratic mean-field
  control/game in continuous time.
\newblock In {\em International Conference on Machine Learning}, pages
  10772--10782. PMLR, 2021.

\bibitem{wu2020finite}
Yue~Frank Wu, Weitong Zhang, Pan Xu, and Quanquan Gu.
\newblock A finite-time analysis of two time-scale actor-critic methods.
\newblock {\em Advances in {N}eural {I}nformation {P}rocessing {S}ystems},
  33:17617--17628, 2020.

\bibitem{yong1999stochastic}
Jiongmin Yong and Xun~Yu Zhou.
\newblock {\em Stochastic controls: Hamiltonian systems and HJB equations},
  volume~43.
\newblock Springer Science \& Business Media, 1999.

\bibitem{zang2022machine}
Yaohua Zang, Jihao Long, Xuanxi Zhang, Wei Hu, E~Weinan, and Jiequn Han.
\newblock A machine learning enhanced algorithm for the optimal landing
  problem.
\newblock In {\em Mathematical and Scientific Machine Learning}, pages
  319--334. PMLR, 2022.

\bibitem{zhou2021actor}
Mo~Zhou, Jiequn Han, and Jianfeng Lu.
\newblock Actor-critic method for high dimensional static
  {H}amilton--{J}acobi--{B}ellman partial differential equations based on
  neural networks.
\newblock {\em SIAM Journal on Scientific Computing}, 43(6):A4043--A4066, 2021.

\bibitem{zhou2023single}
Mo~Zhou and Jianfeng Lu.
\newblock Single timescale actor-critic method to solve the linear quadratic
  regulator with convergence guarantees.
\newblock {\em Journal of Machine Learning Research}, 24(222):1--34, 2023.

\bibitem{zhou2024solving}
Mo~Zhou and Jianfeng Lu.
\newblock Solving time-continuous stochastic optimal control problems:
  Algorithm design and convergence analysis of actor-critic flow.
\newblock {\em arXiv preprint arXiv:2402.17208}, 2024.

\bibitem{zhou2024deep}
Mo~Zhou, Stanley Osher, and Wuchen Li.
\newblock A deep learning algorithm for computing mean field control problems
  via forward-backward score dynamics.
\newblock {\em arXiv preprint arXiv:2401.09547}, 2024.

\bibitem{zhou2021curse}
Xun~Yu Zhou.
\newblock Curse of optimality, and how we break it.
\newblock {\em Available at SSRN 3845462}, 2021.

\end{thebibliography}

\newpage
\appendix

The appendix is organized as follows:
\begin{enumerate}
[noitemsep,topsep=0pt,parsep=0pt,partopsep=0pt]
\item We give a counter example to confirm the necessity of the assumption that $G$ is strongly concave in $u$ in Appendix \ref{sec:counter_example}.
\item We prove the propositions in Appendix \ref{sec:props}.
\item We state and prove some auxiliary lemmas in Appendix \ref{sec:lemmas} in preparation for proving the main theorems.
\item Finally, We prove the main theorems in Appendix \ref{sec:thms}.
\end{enumerate}

\section{Examples}

\subsection{A concrete example}\label{sec:concrete_example}
In this section, we provide an example that satisfies all assumptions. For simplicity, let us consider a $1$ dimensional example, i.e. $m=n=n'=1$. The state space is a one-dimensional unit torus, i.e., $[0,1]$ with periodic boundary condition. This example can be generalized to higher dimensions easily.

We set $T=1$, $b(x,u)=u$, $\sigma(x,u)=1$, $r(x,u) = \frac12 (1-u)^2 + \frac14 (1-u)^4$, and $h(x)=1$. We can see that the uniform ellipticity assumption is satisfied. We also recall that $x_0$ is uniformly distributed, so $\rho(0,x)=1$. In this setting, the generalized Hamiltonian is
$$G(t,x,u,p,P) = \frac12 P + pu - \frac12 (1-u)^2 - \frac14 (1-u)^4.$$
Its first and second order derivatives w.r.t. $u$ are
\begin{align}
& \partial_u G(t,x,u,p,P) = p + 1 - u + (1-u)^3, \label{eq:nbuG_example} \\
&\partial_u^2 G(t,x,u,p,P) = -1 -3 (1-u)^2 \le -1, \nonumber
\end{align}
which implies $G$ is $\lambda_G$-strongly concave with $\lambda_G=1$.

We check Assumption \ref{assump:basic} first. The functions $b$, $\sigma$, $r$, and $h$ are clearly smooth, with the corresponding derivatives bounded. Therefore, the first and second items in Assumption \ref{assump:basic} are satisfied. Solving the critical point equation $\partial_u G = 0$ through Cardano formula, the generalized Hamiltonian is maximized at 
$$u^* = 1 + \parentheses{\frac12p + \sqrt{\frac14p^2 + \frac{1}{27}}}^{\frac13} + \parentheses{\frac12p - \sqrt{\frac14p^2 + \frac{1}{27}}}^{\frac13},$$
which implies 
$$u^*(t,x) = 1 + \parentheses{-\frac12 V^*_x + \sqrt{\frac14(V^*_x)^2 + \frac{1}{27}}}^{\frac13} + \parentheses{-\frac12V^*_x - \sqrt{\frac14(V^*_x)^2 + \frac{1}{27}}}^{\frac13},$$
where we denote $\partial_x V^*(t,x)$ by $V^*_x$ for simplicity. Substitute this into the HJB equation 
$$- V^*_t(t,x) + \sup_{u\in\RR} \parentheses{-\frac12 V^*_{xx}(t,x) -  V^*_x(t,x) u - \frac12(1-u)^2 - \frac14(1-u)^4} = 0,$$
we obtain
\begin{align*}
& -V^*_t - \frac12 V^*_{xx} - \frac14 \sqbra{ \parentheses{-\frac12 V^*_x + \sqrt{\frac14(V^*_x)^2 + \frac{1}{27}}}^{\frac13} +  \parentheses{-\frac12V^*_x - \sqrt{\frac14(V^*_x)^2 + \frac{1}{27}}}^{\frac13}}^2\\
& - V_x^* \sqbra{1 + \frac34 \parentheses{-\frac12 V^*_x + \sqrt{\frac14(V^*_x)^2 + \frac{1}{27}}}^{\frac13} + \frac34 \parentheses{-\frac12V^*_x - \sqrt{\frac14(V^*_x)^2 + \frac{1}{27}}}^{\frac13}} = 0.
\end{align*}
The solution to this HJB equation is $V^*(t,x) = 1$. Hence, the corresponding optimal control is $u^*(t,x) = 1$. Therefore, the third item in Assumption \ref{assump:basic} is satisfied.

We check Assumption \ref{assump:u_smooth} next. We start with $u^0(t,x) = 0$ at $\tau=0$ and study the policy gradient dynamic. By definition \eqref{eq:value}, the initial ($\tau=0$) value function is $V_{u^0}(t,x) = \frac34(1-t) + 1$. Therefore, the initial control and value function are state homogeneous (i.e., do not depend on $x$).

Then, we consider the policy gradient dynamic. By definition \eqref{eq:PG_ideal} and \eqref{eq:nbuG_example} above, the policy gradient dynamic for $u^\tau$ is
\begin{equation}\label{eq:PG_example}
\dfrac{\rd}{\rd\tau} u^\tau(t,x) = \rho^{u^\tau}(t,x) \, \sqbra{-\partial_x V_{u^\tau}(t,x) + 1 - u^\tau(t,x) + (1 - u^\tau(t,x))^3}.
\end{equation}
Here, $\rho^{u^\tau}$ is the density function of the state dynamic under control $u^\tau$ and it satisfies the Fokker--Planck equation
$$\pt \rho^{u^\tau}(t,x) = -\partial_x \sqbra{u^\tau(t,x) \rho(t,x)} + \frac12 \partial_x^2 \rho^{u^\tau}(t,x) \quad\quad \rho^{u^\tau}(0,x) = 1.$$
We observe that $\rho^{u^0}(t,x) = 1$ at initial time $\tau=0$. Also, we have $\rho^{u^\tau}(t,x) = 1$ (i.e., is state homogeneous) as long as the control $u^\tau$ is state homogeneous. We also observe that the value function $V_{u^\tau}$ is state homogeneous as long as $u^\tau$ is state homogeneous. Therefore, according to the policy gradient dynamic \eqref{eq:PG_example}, the state homogeneity property of $u^\tau$ is preserved. So, $\rho^{u^\tau}(t,x) = 1$ holds for all $\tau\ge0$. Also, the value function $V_{u^\tau}$ is state homogeneous for all $\tau\ge0$. Therefore, the policy gradient flow \eqref{eq:PG_example} becomes
$$\dfrac{\rd}{\rd\tau} u^\tau(t,x) = 1 - u^\tau(t,x) + (1 - u^\tau(t,x))^3.$$
Using the separation of variable method, we obtain the explicit policy gradient dynamic
$$u^\tau(t,x) = 1 - \sqrt{ \dfrac{1}{2e^{2\tau}-1}}.$$
$u^\tau(t,x)$ is smooth and the corresponding derivatives are bounded. Therefore, Assumption \ref{assump:u_smooth} is satisfied.

Finally, let us check Assumption \ref{assump:actor_rate}. First, we can directly compute.
$$J[u^*] = \EE \sqbra{\int_0^1 r(x^*_t, u^*(t,x^*_t)) \,\rd t + h(x^*_T)} = 1.$$
Therefore, given any control function $u \in \mU$, let $x_t$ be the corresponding state dynamic, then
\begin{align*}
& \quad J[u] - J[u^*] = \EE\sqbra{\int_0^1 \parentheses{ \frac12 (1-u(t,x_t))^2 + \frac14 (1-u(t,x_t))^4} \,\rd t + h(x_T)} - 1 \\
& \ge \EE\sqbra{\int_0^1 \frac12 (1-u(t,x_t))^2 \,\rd t + h(x_T)} - 1 = \EE\sqbra{\int_0^1 \frac12 (1-u(t,x_t))^2 \,\rd t}\\
&  = \int_0^1 \int_\RR \frac12 (1-u(t,x))^2 \rho^u(t,x) \,\rd x \,\rd t \ge \int_0^1 \int_\RR \frac12 (1-u(t,x))^2 \rho_0 \,\rd x \,\rd t\\
&   = \int_0^1 \int_\RR \frac12 (u^*(t,x)-u(t,x))^2 \rho_0 \,\rd x \,\rd t = \frac12 \rho_0 \norm{u-u^*}^2_{L^2},
\end{align*}
where the second inequality is due to Proposition \ref{prop:rho}. Therefore, Assumption \ref{assump:actor_rate} is satisfied with $\omega(y) = \sqrt{2y/\rho_0}$.

Therefore, we establish that all assumptions in this paper are satisfied for this example.

\subsection{A counter example for multiple critical points of the gradient flow}\label{sec:counter_example}
In this section, we give a counter example to show the necessity of the strong concavity of $G$ in $u$. Let $n=n'=m=1$. Consider the HJB equation
$$-\partial_t V(t,x) + \sup_u  \sqbra{-\partial^2_x V(t,x) - \frac13 u^3 \partial_x V(t,x) - \frac14 u^4 + \frac12 u^2} = 0$$
with some nice terminal condition $V(T,x) = h(x)$. To simplify the notation, we use $V_t, V_{xx}, V_x$ to denote the derivatives. In order to obtain the optimal control, we need to find the maximum of $G$. i.e., we seek for the minimum of the quartic function $g(u; V_x) := \frac14 u^4 + \frac13 V_x u^3 - \frac12 u^2$ w.r.t. $u$ for any given $V_x \in \RR$ (cf. \eqref{eq:max1}). This quartic function has a local maximum $u=0$ and two local minimums
$$u_{1,2} = \frac12\parentheses{-V_x \pm \sqrt{V_x^2 + 4}}.$$
With some standard calculus, we obtain the optimal control
$$u^* = \frac12\parentheses{-V_x - \sign(V_x) \sqrt{V_x^2 + 4}}.$$
The HJB equation becomes
\begin{equation}\label{eq:HJB_eg}
- V_t(t,x) - V_{xx}(t,x) - g \parentheses{\frac12\parentheses{-V_x - \sign(V_x) \sqrt{V_x^2 + 4}}; V_x} = 0,
\end{equation}
which is semilinear. We define a second control (implicitly) through
$$\widetilde{u} = \frac12\parentheses{-V_x + \sign(V_x) \sqrt{V_x^2 + 4}}.$$
The the HJ equation corresponding to this control is
\begin{equation}\label{eq:HJ_eg}
- V_t(t,x) - V_{xx}(t,x) - g \parentheses{\frac12\parentheses{-V_x + \sign(V_x) \sqrt{V_x^2 + 4}}; V_x} = 0
\end{equation}
According to standard results in semi-linear parabolic PDE (see \cite{taylor2012partial} Chapter 15 for example), we have unique solutions $V^*$ and $\widetilde{V}$ to \eqref{eq:HJB_eg} and \eqref{eq:HJ_eg} respectively (if $T$ is not too large). Note that $\widetilde{u}(t,x)$ is a local but not global maximum for the map $u \mapsto G(t,x,u,-\widetilde{V}_x,-\widetilde{V}_{xx})$ almost everywhere and $\widetilde{V}$ is the value function for $\widetilde{u}$. Therefore, if our policy gradient algorithm reaches $\widetilde{u}$, it becomes static at this suboptimal solution. This example demonstrate the necessity of the concavity assumption of the generalized Hamiltonian $G$ in $u$.

\section{Proofs for the Propositions}\label{sec:props}
\begin{proof}[Proof of Proposition \ref{prop:cost_derivative}]
We fix an arbitrary perturbation function $\phi(t,x)$, then
\begin{equation*}
    \dfrac{\rd}{\rd \ve} J[u + \ve \phi] \bigg|_{\ve=0} = \inner{\fd{J}{u}}{\phi}_{L^2}.
\end{equation*}
We denote $x^{\ve}_t$ the SDE \eqref{eq:SDE_X} under control function $u^{\ve} := u + \ve\phi$ that start with $x^{\ve}_0 \sim \text{Unif}(\mX)$. Let $\rho^{\ve}(t,x)$ be its density. We also denote the corresponding value function by $V^{\ve}(t,x) := V_{u^{\ve}}(t,x)$ for simplicity. Then, $\rho^{\ve}$ and $V^{\ve}$ depend continuously on $\ve$ (cf. \cite{yong1999stochastic} section 4.4.1). By definition of the cost functional \eqref{eq:cost},
\begin{equation*}
\begin{aligned}
    J[u + \ve\phi] & = \EE\sqbra{\int_0^T r(x^{\ve}_t, u^{\ve}(t,x^{\ve}_t)) \rd t + h(x^{\ve}_T)} \\
    & = \int_0^T \int_{\mX} r(x, u^{\ve}(t, x)) \rho^{\ve}(t ,x) \,\rd x \, \rd t + \int_{\mX} V^{\ve}(T, x) \rho^{\ve}(T ,x) \,\rd x.
\end{aligned}
\end{equation*}
Taking derivative w.r.t. $\ve$, and note that $V^{\ve}(T,x)=h(x)$ does not depend on $\ve$, we obtain
\begin{equation}\label{eq:dJdep}
\begin{aligned}
    &\dfrac{\rd}{\rd \ve} J[u + \ve \phi] = \int_{\mX} V^{\ve}(T, x) \pve \rho^{\ve}(T ,x) \,\rd x + \\
    & + \int_0^T \int_{\mX} \sqbra{\nbu r\parentheses{x, u^{\ve}(t, x)}\tp \phi(t,x) \rho^{\ve}(t ,x) + r\parentheses{x, u^{\ve}(t, x)} \pve \rho^{\ve}(t ,x)} \,\rd x \, \rd t.
\end{aligned}
\end{equation}
In order to compute $V^{\ve}(T, x) \pve \rho^{\ve}(T ,x)$ in \eqref{eq:dJdep}, we write down the integral equation
\begin{equation}\label{eq:VT_drho}
    V^{\ve}(T,x) \rho^{\ve}(T,x) = V^{\ve}(0,x) \rho^{\ve}(0,x) + \int_0^T \sqbra{\partial_t \rho^{\ve}(t,x) V^{\ve}(t,x) + \rho^{\ve}(t,x) \partial_t V^{\ve}(t,x)} \rd t.
\end{equation}
We also have
\begin{equation}\label{eq:prop_dJdu_temp}
\begin{aligned}
& \quad \pve V^{\ve}(0,x) \rho^{\ve}(0,x) + \int_0^T \sqbra{ \partial_t \rho^{\ve}(t,x) \pve V^{\ve}(t,x) + \rho^{\ve}(t,x) \pve \partial_t V^{\ve}(t,x)} \rd t \\
& = \pve V^{\ve}(T,x) \rho^{\ve}(T,x) = \pve h(x) \rho^{\ve}(T,x) = 0.
\end{aligned}
\end{equation}
Next, taking derivative of \eqref{eq:VT_drho} w.r.t. $\ve$ (note that $V^{\ve}(T,x)=h(x)$ and $\rho^{\ve}(0,\cdot) \equiv 1$ do not depend on $\ve$), we obtain
\begin{equation}\label{eq:Vdrhodep}
\begin{aligned}
& \quad V^{\ve}(T, x) \, \pve \rho^{\ve}(T ,x) \\
& = \pve V^{\ve}(0,x) \rho^{\ve}(0,x) + \int_0^T \partial_{\ve} \sqbra{ \partial_t \rho^{\ve}(t,x) V^{\ve}(t,x) + \rho^{\ve}(t,x) \partial_t V^{\ve}(t,x) } \rd t \\
& = \int_0^T \sqbra{\pve \partial_t \rho^{\ve}(t,x) V^{\ve}(t,x) + \pve \rho^{\ve}(t,x) \partial_t V^{\ve}(t,x) } \rd t \\
& = \int_0^T \sqbra{\pve \partial_t \rho^{\ve}(t,x) V^{\ve}(t,x) + \pve \rho^{\ve}(t,x) G\parentheses{t, x, u(t, x), -\nx V^{\ve}, -\nx^2 V^{\ve}} } \rd t
\end{aligned}
\end{equation}
where we used \eqref{eq:prop_dJdu_temp} and the HJ equation \eqref{eq:HJ} in the second and third equality respectively.
Substitute \eqref{eq:Vdrhodep} into \eqref{eq:dJdep}, we obtain
\begin{equation*}
\begin{aligned}
& \quad \dfrac{\rd}{\rd \ve} J[u + \ve \phi] \\
& = \int_0^T \int_{\mX} \sqbra{\pve \rho^{\ve}(t,x) G\parentheses{t, x, u(t, x), -\nx V^{\ve}, -\nx^2 V^{\ve}} + \pve \partial_t \rho^{\ve}(t,x) V^{\ve}(t,x)} \,\rd x \, \rd t\\
& \hspace{0.4in} + \int_0^T \int_{\mX} \sqbra{\nbu r\parentheses{x, u^{\ve}(t, x)}\tp \phi(t,x) \rho^{\ve}(t ,x) + r\parentheses{x, u^{\ve}(t, x)} \pve \rho^{\ve}(t ,x)} \,\rd x \, \rd t
\end{aligned}
\end{equation*}
Taking derivative of the Fokker--Planck equation \eqref{eq:FokkerPlanck} w.r.t. $\ve$, we obtain
\begin{equation*}
\begin{aligned}
    & \quad \pve \partial_t \rho^{\ve}(t,x)\\
    & = -\nx \cdot \sqbra{\nbu b\parentheses{x, u^{\ve}(t, x)} \phi(t,x) \rho^{\ve}(t,x) + b\parentheses{x, u^{\ve}(t, x)} \pve \rho^{\ve}(t,x) } \\
    & \quad + \sum_{i,j=1}^n \partial_i \partial_j \sqbra{\nbu D_{ij} \parentheses{x, u^{\ve}(t, x)}\tp \phi(t,x) \rho^{\ve}(t,x) + D_{ij} \parentheses{x, u^{\ve}(t, x)} \pve \rho^{\ve}(t,x)}
\end{aligned}
\end{equation*}
Therefore
\begin{equation*}
\begin{aligned}
& \quad \dfrac{\rd}{\rd \ve} J[u + \ve \phi] = \int_0^T \int_{\mX} \left\{ \pve\rho^{\ve}(t,x) G\parentheses{t, x, u(t, x), -\nx V^{\ve}, -\nx^2 V^{\ve}} \right.\\
& + V^{\ve}(t,x) \{ -\nx \cdot \sqbra{\nbu b\parentheses{x, u^{\ve}(t, x)} \phi(t,x) \rho^{\ve}(t,x) + b\parentheses{x, u^{\ve}(t, x)} \pve \rho^{\ve}(t,x) } \\
& \hspace{0.4in} + \sum_{i,j=1}^n \partial_i \partial_j \sqbra{\nbu D_{ij} \parentheses{x, u^{\ve}(t, x)} \phi(t,x) \rho^{\ve}(t,x) + D_{ij} \parentheses{x, u^{\ve}(t, x)} \pve \rho^{\ve}(t,x)} \}\\
& \hspace{0.4in} + \left. \sqbra{\nbu r\parentheses{x, u^{\ve}(t, x)} \phi(t,x) \rho^{\ve}(t ,x) + r\parentheses{x, u^{\ve}(t, x)} \pve \rho^{\ve}(t ,x)} \right\}\rd x \, \rd t.
\end{aligned}
\end{equation*}
Applying integration by part in $x$, we get 
\begin{equation*}
\begin{aligned}
& \quad \dfrac{\rd}{\rd \ve} J[u + \ve \phi] = \int_0^T \int_{\mX} \left\{ \pve\rho^{\ve}(t,x) G\parentheses{t, x, u(t, x), -\nx V^{\ve}, -\nx^2 V^{\ve}} \right.\\
& \hspace{0.4in} +\nx V^{\ve}(t,x)\tp \sqbra{\nbu b\parentheses{x, u^{\ve}(t, x)} \phi(t,x) \rho^{\ve}(t,x) + b\parentheses{x, u^{\ve}(t, x)} \pve \rho^{\ve}(t,x) } \\
& \hspace{0.4in} + \sum_{i,j=1}^n \partial_i \partial_j V^{\ve}(t,x) \sqbra{\nbu D_{ij} \parentheses{x, u^{\ve}(t, x)} \phi(t,x) \rho^{\ve}(t,x) + D_{ij} \parentheses{x, u^{\ve}(t, x)} \pve \rho^{\ve}(t,x)}\\
& \hspace{0.4in} + \left. \sqbra{\nbu r\parentheses{x, u^{\ve}(t, x)}\tp \phi(t,x) \rho^{\ve}(t ,x) + r\parentheses{x, u^{\ve}(t, x)} \pve \rho^{\ve}(t ,x)} \right\}\rd x \, \rd t.
\end{aligned}
\end{equation*}
Making a rearrangement, we get 
\begin{equation*}
\begin{aligned}
& \quad \dfrac{\rd}{\rd \ve} J[u + \ve \phi]  = \int_0^T \int_{\mX} \bigg\{ \pve\rho^{\ve}(t,x) \Big[ G\parentheses{t, x, u(t, x), -\nx V^{\ve}, -\nx^2 V^{\ve}} \\
& \hspace{0.4in} + \nx V^{\ve}(t,x)\tp b\parentheses{x, u^{\ve}(t, x)} + \sum_{i,j=1}^n \partial_i \partial_j V^{\ve}(t,x) D_{ij} \parentheses{x, u^{\ve}(t, x)} + r\parentheses{x, u^{\ve}(t, x)} \Big] \\
& \hspace{0.4in} + \rho^{\ve}(t,x) \Big[ \nx V^{\ve}(t,x)\tp \nbu b\parentheses{x, u^{\ve}(t, x)} + \sum_{i,j=1}^n \partial_i \partial_j V^{\ve}(t,x) ~ \nbu D_{ij} \parentheses{x, u^{\ve}(t, x)}\\
& \hspace{1in} + \nbu r\parentheses{x, u^{\ve}(t, x)} \Big]  \phi(t,x) \bigg\} \rd x \, \rd t 
\end{aligned}
\end{equation*}
Therefore, by the definition of $G$ \eqref{eq:generalized_Hamiltonian},
\begin{equation*}
\begin{aligned}
& \quad \dfrac{\rd}{\rd \ve} J[u + \ve \phi] \\
& = \int_0^T \int_{\mX} \sqbra{\pve\rho^{\ve}(t,x) \cdot [0] - \rho^{\ve}(t,x) \nbu G\parentheses{t, x, u^{\ve}(t, x), -\nx V^{\ve}, -\nx^2 V^{\ve}} \phi(t,x)} \rd x \, \rd t\\
& = -\int_0^T \int_{\mX} \rho^{\ve}(t,x) \nbu G\parentheses{t, x, u^{\ve}(t, x), -\nx V^{\ve}, -\nx^2 V^{\ve}} \phi(t,x) \, \rd x \, \rd t.
\end{aligned}
\end{equation*}
Let $\ve = 0$, we get
\begin{equation}\label{eq:dJdve1}
\dfrac{\rd}{\rd \ve} J[u + \ve \phi] \bigg|_{\ve=0} = - \int_0^T \int_{\mX} \rho(t,x) \nbu G\parentheses{t, x, u(t, x), -\nx V_u, -\nx^2 V_u} \phi(t,x) \, \rd x \, \rd t.
\end{equation}
Therefore,
\begin{equation*}
\fd{J}{u}(t, x)  = - \rho(t,x) ~\nbu G\parentheses{t, x, u(t, x), -\nx V(t,x), -\nx^2 V(t,x) }.
\end{equation*}
i.e. \eqref{eq:actor_derivative} holds. The proof for \eqref{eq:dVdu} is almost the same. Firstly, changing the control function at $t < s$ will not affect $V^u(s,\cdot)$ by definition, so we just need to show \eqref{eq:dVdu} when $t \ge s$. We recall the definition of value function \eqref{eq:value}
\begin{equation*}
\begin{aligned}
V_u(s,y) &= \EE\sqbra{\int_{s}^{T} r(x_t, u_t) \,\rd t + h(x_T) \;\Big|\; x_s=y} \\
& = \int_s^T \int_{\mX} r(x,u(t,x)) p^u(t,x;s,y) \, \rd x \, \rd t + \int_{\mX} h(x) p^u(T,x;s,y) \,\rd x.
\end{aligned}
\end{equation*}
Here, $p^u(t,x;s,y)$ is the fundamental solution of the Fokker--Planck equation \eqref{eq:FokkerPlanck}, so $p^u(t,x;s,y)$, as a function of $(t,x)$, is the density of $x_t$ starting at $x_s=y$. Therefore, we only need to repeat the argument when proving \eqref{eq:dJdve1}. The only caveat we need to be careful is that $p^{\ve}(s,\cdot;s,y) = \delta_y$, so the classical derivative does not exist. This is not an essential difficulty because we can pick an arbitrary smooth probability density function $\psi(y)$ on $\mX$ and compute
\begin{equation}\label{eq:general_derivative}
\frac{\rd}{\rd \ve} \int_{\mX} V^{\ve}(s,y) \psi(y) \, \rd y \Big|_{\ve=0}.
\end{equation}
For example, when $s=0$ and $\psi \equiv 1$, \eqref{eq:general_derivative} becomes 
\begin{equation*}
\frac{\rd}{\rd \ve} \int_{\mX} V^{\ve}(0,y) \, \rd y \Big|_{\ve=0} = \frac{\rd}{\rd \ve} J[u^{\ve}] \Big|_{\ve=0}.
\end{equation*}
Therefore, we can repeat the argument to prove \eqref{eq:actor_derivative} and get
\begin{equation}\label{eq:general_derivative2}
\begin{aligned}
& \quad \frac{\rd}{\rd \ve} \int_{\mX} V^{\ve}(s,y) \psi(y) \, \rd y \Big|_{\ve=0} \\
& = - \int_s^T \int_{\mX} \rho^{u,s,\psi}(t,x) \nbu G\parentheses{t, x, u(t, x), -\nx V_u, -\nx^2 V_u} \phi(t,x) \, \rd x \, \rd t.
\end{aligned}
\end{equation}
where $\rho^{u,s,\psi}(t,x) := \int_{\mX} p^u(t,x;s,y) \psi(y) \, \rd y$ is the solution to the Fokker--Planck equation with initial condition $\rho^{u,s,\psi}(s,x) = \psi(x)$ and is also the density function of $x_t$, which starts with $x_s$, who follows a distribution of $\psi$. The only difference between proving \eqref{eq:dJdve1} and \eqref{eq:general_derivative2} is that we need to replace $\int_0^T$ by $\int_s^T$, replace $\rho^{\ve}(t,x)$ by $\rho^{\ve,s,\psi}(t,x) := \rho^{u^{\ve},s,\psi}(t,x)$, and replace $\rho^{\ve}(0,x)$ by $\rho^{\ve,s,\psi}(s,x)$. Therefore,
\begin{equation*}
\dfrac{\delta \parentheses{ \int_{\mX} V_u(s,y) \psi(y) \, \rd y}}{\delta u} = - \int_{\mX} p^u(t,x;s,y) \psi(y) \, \rd y \, \nbu G\parentheses{t, x, u(t, x), -\nx V_u, -\nx^2 V_u}.
\end{equation*}
Hence, \eqref{eq:dVdu} holds.
\end{proof}
\begin{rmk}
In the proof of this proposition, we have assumed sufficient regularity such that all the derivatives exist in the classical sense. We believe that the theorem still holds with weaker assumptions, which can be proved using the spike variation technique. See \cite{yong1999stochastic} section 3.4 Theorem 4.4 for example.
\end{rmk}

\begin{proof}[Proof for Proposition \ref{prop:rho}]
The Fokker--Planck equation has been well-studied. Let $p^u(t,x;s,y)$ denote the fundamental solution to \eqref{eq:FokkerPlanck}. Aronson found that the fundamental solution of a linear parabolic equation can be upper and lower bounded by fundamental solutions of heat equation (i.e. Gaussian functions) with different thermal diffusivity constant \cite{aronson1967bounds}. For example, let $\widetilde{p}^u(t,x;s,y)$ be the fundamental solution of the Fokker--Planck equation \eqref{eq:FokkerPlanck} in $\RR^n$ (where $b$ and $\sigma$ are extended periodically), then
\begin{equation}\label{eq:prop_rho_temp}
C^{-1} (t-s)^{-n/2} \exp\Bigl(-\dfrac{C \abs{x-y}^2}{t-s}\Bigr) \le \widetilde{p}^u(t,x;s,y) \le C (t-s)^{-n/2} \exp\Bigl(-\dfrac{C^{-1} \abs{x-y}^2}{t-s}\Bigr)
\end{equation}
for all $s<t\le T$ and $x,y \in \RR^n$, where $C$ only depends on $n$, $T$, and $K$. We are in the unit torus $\mX$ instead of $\RR^n$, so 
$$p^u(t,x;s,y) = \sum_{z \in \ZZ^n} \widetilde{p}^u(t,x + z;s,y),$$
where the $x,y$ on the left is in $\mX$, and the $x,y$ on the right can be viewed as their embedding into $\RR^n$.
Our solution to the Fokker--Planck equation, starting at uniform distribution $\rho(0,x) \equiv 1$, can be represented by
\begin{align*}
& \quad \rho^u(t,x) = \int_{\mX} p^u(t,x;0,y) \,\rd y = \int_{[0,1]^n} \sum_{z \in \ZZ^n} \widetilde{p}^u(t,x+z;0,y) \,\rd y \\
& = \int_{[0,1]^n} \sum_{z \in \ZZ^n} \widetilde{p}^u(t,x;0,y-z) \,\rd y = \int_{\RR^n} \widetilde{p}^u(t,x;0,y) \,\rd y 
\end{align*}

Substituting the lower and upper bound \eqref{eq:prop_rho_temp}, we obtain
$$\rho^u(t,x) \ge \int_{\RR^n}  C^{-1} t^{-n/2} \exp\parentheses{-\dfrac{C \abs{x-y}^2}{t}} \,\rd y =: \rho_0$$
and
$$\rho^u(t,x) \le \int_{\RR^n}  C t^{-n/2} \exp\parentheses{-\dfrac{C^{-1} \abs{x-y}^2}{t}} \,\rd y =: \rho_1.$$
Here, the two integrals above are invariant w.r.t. $t$ because of a simple change of variable. Therefore, we obtain a uniform lower bound $\rho_0$ and upper bound $\rho_1$ for $\rho^u(t,x)$, which depend only on $T$, $n$, and $K$.
\end{proof}

\section{Some auxiliary lemmas}\label{sec:lemmas}
We state and prove some lemmas in this section.
\begin{lem}\label{lem:gronwall}[Stochastic Gronwall inequality] 
Under Assumption \ref{assump:basic}, there exists a positive constant $C_1$ s.t. for any two control functions $u_1, u_2 \in \mU$, we have
\begin{equation}\label{eq:Gronwall1}
\sup_{t \in [0,T]} \EE  \abs{x^1_t - x^2_t}^2 \le C_1 \EE\abs{x^1_0-x^2_0}^2 + C_1 \EE \sqbra{ \int_0^T \abs{u_1(t,x^1_t) - u_2(t,x^1_t)}^2 \rd t},
\end{equation}
where $x^1_t$ and $x^2_t$ are the state process \eqref{eq:SDE_X} under controls $u_1$ and $u_2$ respectively. As a direct corollary, if $u_1=u_2$, then
\begin{equation}\label{eq:Gronwall2}
\sup_{t \in [0,T]} \EE  \abs{x^1_t - x^2_t}^2 \le C_1 \EE\abs{x^1_0-x^2_0}^2.
\end{equation}
Moreover, if $x^1_0=x^2_0 \sim \text{Unif}(\mX)$, then
\begin{equation}\label{eq:Gronwall3}
\sup_{t \in [0,T]} \EE \abs{x^1_t - x^2_t}^2 \le C_1 \norm{u_1-u_2}_{L^2}^2.
\end{equation}
\end{lem}
\begin{proof}
We denote $b^i_t=b(x^i_t, u_i(t,x^i_t))$, $\sigma^i_t=\sigma(x^i_t, u_i(t,x^i_t))$ for $i=1,2$, so $\rd x^i_t = b^i_t \,\rd t + \sigma^i_t \,\rd W_t$. By It\^o's lemma,
$$\rd \abs{x^1_t-x^2_t}^2 = \sqbra{\abs{\sigma^1_t-\sigma^2_t}^2 + 2\inner{x^1_t-x^2_t}{b^1_t-b^2_t}} \rd t + 2(x^1_t-x^2_t)\tp (\sigma^1_t-\sigma^2_t) \,\rd W_t.$$
Integrate and take expectation, we obtain
\begin{equation}\label{eq:square_diffx}
\EE \abs{x^1_T-x^2_T}^2 = \EE\abs{x^1_0-x^2_0}^2 + \EE \int_0^T \sqbra{\abs{\sigma^1_t-\sigma^2_t}^2 + 2\inner{x^1_t-x^2_t}{b^1_t-b^2_t}} \rd t.
\end{equation}
By the Lipschitz condition in Assumption \ref{assump:basic},
\begin{align*}
& \quad \abs{b^1_t - b^2_t} \le L \abs{x^1_t-x^2_t} + L \abs{u_1(t,x^1_t) - u_2(t,x^2_t)} \\
& \le (L+L^2) \abs{x^1_t-x^2_t} + L \abs{u_1(t,x^1_t) - u_2(t,x^1_t)}.
\end{align*}
So, 
\begin{equation}\label{eq:diffb_bound}
\abs{b^1_t - b^2_t}^2 \le 2(L+L^2)^2 \abs{x^1_t-x^2_t}^2 + 2L^2 \abs{u_1(t,x^1_t) - u_2(t,x^1_t)}^2.
\end{equation}
Similarly,
\begin{equation}\label{eq:diffsigma_bound}
\abs{\sigma^1_t - \sigma^2_t}^2 \le 2(L+L^2)^2 \abs{x^1_t-x^2_t}^2 + 2L^2 \abs{u_1(t,x^1_t) - u_2(t,x^1_t)}^2.
\end{equation}
Applying Cauchy's inequality, and substituting \eqref{eq:diffb_bound} and \eqref{eq:diffsigma_bound} into \eqref{eq:square_diffx}, we obtain
\begin{equation}\label{eq:square_diffx2}
\begin{aligned}
\EE \abs{x^1_T-x^2_T}^2 \le \EE\abs{x^1_0-x^2_0}^2 + \EE \int_0^T \sqbra{\abs{\sigma^1_t-\sigma^2_t}^2 + \abs{x^1_t-x^2_t}^2 + \abs{b^1_t-b^2_t}^2 }\rd t\\
\le \EE\abs{x^1_0-x^2_0}^2 + \EE \int_0^T \sqbra{17L^4 \abs{x^1_t-x^2_t}^2 + 4L^2 \abs{u_1(t,x^1_t) - u_2(t,x^1_t)}^2}\rd t
\end{aligned}
\end{equation}
Note that \eqref{eq:square_diffx2} still holds if we replace $T$ by some $T'<T$, so we can apply Gronwall's inequality and obtain
\begin{equation}\label{eq:square_diffx3}
\EE \abs{x^1_T-x^2_T}^2 \le e^{17L^4T} \EE\abs{x^1_0-x^2_0}^2 + 4L^2 e^{17L^4T} \EE \sqbra{ \int_0^T \abs{u_1(t,x^1_t) - u_2(t,x^1_t)}^2 \rd t}.
\end{equation}
Again, \eqref{eq:square_diffx3} still holds if we replace $T$ by some $T'<T$, so \eqref{eq:Gronwall1} holds. Moreover, if $x_0^1 = x_0^2 \sim \text{Unif}(\mX)$, then by Proposition \ref{prop:rho},
\begin{align*}
& \quad \EE \int_0^T \abs{u_1(t,x^1_t) - u_2(t,x^1_t)}^2 \rd t \\
& = \int_{\mX}\int_0^T \rho^{1}(t,x) \abs{u_1(t,x) - u_2(t,x)}^2 \rd t \rd x \le \rho_1 \norm{u_1-u_2}_{L^2}^2.
\end{align*}
Therefore, \eqref{eq:Gronwall3} holds.
\end{proof}

\begin{lem}\label{lem:regularity_Vu}
Under Assumption \ref{assump:basic}, there exists a positive constant $C_2$ s.t. for any two control functions $u_1, u_2 \in \mU$, we have
\begin{equation}\label{eq:regularity_Vu}
\normHtwo{V_{u_1} - V_{u_2}} \le C_2 \norm{u_1-u_2}_{L^2}.
\end{equation}
\end{lem}
\begin{proof}
We firstly give some notations. Following the notations in the previous lemma, $x^1_t$ and $x^2_t$ are the state process w.r.t.{} controls $u_1$ and $u_2$, starting at $x^1_0 = x^2_0 \sim \text{Unif}(\mX)$. For $i=1,2$, we have $b^i_t=b(x^i_t, u_i(t,x^i_t))$, $\sigma^i_t=\sigma(x^i_t, u_i(t,x^i_t))$, and $r^i_t=r(x^i_t, u_i(t,x^i_t))$. We also define the following gradient processes $\nx b^i_t=\nx b(x^i_t, u_i(t,x^i_t))$, $\nx \sigma^i_t =\nx \sigma(x^i_t, u_i(t,x^i_t))$, and $\nx r^i_t = \nx r(x^i_t, u_i(t,x^i_t))$. Here please note that $\nx$ only operate on the first argument in $b$, $\sigma$, and $r$. By Assumption \ref{assump:basic}, for $f = b, \sigma, r, \nx b, \nx \sigma, \nx r$, we have
\begin{equation*}
\begin{aligned}
& \quad \abs{f^1_t-f^2_t} = \abs{f(x^1_t, u_1(t,x^1_t))-f(x^2_t, u_2(t,x^2_t))}\\
& \le L \abs{x^1_t-x^2_t} + L \abs{u_1(t,x^1_t) - u_2(t,x^2_t)} \le (L+L^2) \abs{x^1_t-x^2_t} + L \abs{u_1(t,x^1_t) - u_2(t,x^1_t)}
\end{aligned}
\end{equation*}
and hence
\begin{equation*}
\abs{f^1_t-f^2_t}^2 \le 2(L+L^2)^2 \abs{x^1_t-x^2_t}^2 + 2L^2 \abs{u_1(t,x^1_t) - u_2(t,x^1_t)}^2.
\end{equation*}
If we make an integration and apply Lemma \ref{lem:gronwall}, we obtain
\begin{equation}\label{eq:all1-2}
\EE \int_0^T \abs{f^1_t-f^2_t}^2 \rd t \le C \EE \sqbra{ \int_0^T \abs{u_1(t,x^1_t) - u_2(t,x^1_t)}^2 \rd t} \le C \norm{u_1-u_2}_{L^2}^2.
\end{equation}
Next, we will show \eqref{eq:regularity_Vu} step by step.

\emph{Step 1.} We want to show
\begin{equation}\label{eq:regVu_0}
\norm{V_{u_1} - V_{u_2}}_{L^2} \le C \norm{u_1-u_2}_{L^2}.
\end{equation}
Applying It\^o's lemma on $V_{u_i}(t,x^i_t)$, we obtain
$$h(x^i_T) = V_{u_i}(0,x_0) + \int_0^T \parentheses{\partial_t V_{u_i}(t,x^i_t) + \mG_{u_i}V_{u_i}(t,x^i_t) }\rd t + \int_0^T \nx V_{u_i}(t,x^i_t)\tp \sigma^i_t \rd W_t,$$
where $\mG_{u_i}$ is the infinitesimal generator of the SDE \eqref{eq:SDE_X} under control $u_i$. Applying the HJ equation \eqref{eq:HJ} in the drift term and rearranging the terms, we get
\begin{equation}\label{eq:diff_V0}
V_{u_i}(0,x_0) = h(x^i_T) + \int_0^T r^i_t \rd t - \int_0^T \nx V_{u_i}(t,x^i_t)\tp \sigma^i_t \rd W_t.
\end{equation}
So,
\begin{equation}\label{eq:diff_v0EE}
V_{u_1}(0,x_0) - V_{u_2}(0,x_0) = \EE\sqbra{ h(x^1_T) -h(x^2_T) + \int_0^T (r^1_t - r^2_t) \rd t ~\Big|~ x_0}.
\end{equation}
Therefore,
\begin{equation*}
\begin{aligned}
& \quad \int_{\mX} \abs{V_{u_1}(0,x) - V_{u_2}(0,x)}^2 \rd x = \EE \sqbra{\parentheses{V_{u_1}(0,x_0) - V_{u_2}(0,x_0)}^2}\\
& = \EE \sqbra{\parentheses{\EE\sqbra{ h(x^1_T) -h(x^2_T) + \int_0^T (r^1_t - r^2_t) \rd t ~\Big|~ x_0}}^2} \\
& \le \EE \sqbra{\parentheses{\EE\sqbra{ L\abs{x^1_T - x^2_T} + \int_0^T \abs{r^1_t - r^2_t} \rd t ~\Big|~ x_0}}^2} \\
& \le \EE \sqbra{\parentheses{ L\abs{x^1_T - x^2_T} + \int_0^T \abs{r^1_t - r^2_t} \rd t }^2} \\
& \le \EE \sqbra{ 2L^2 \abs{x^1_T - x^2_T}^2 + 2T \int_0^T \abs{r^1_t - r^2_t}^2 \rd t} \le C \norm{u_1-u_2}_{L^2}^2,
\end{aligned}
\end{equation*}
where we have consecutively used: $x_0 \sim \text{Unif}(\mX)$; equation \eqref{eq:diff_v0EE}; Lipschitz condition of $h$ in Assumption \ref{assump:basic}; Jensen's inequality and tower property; Cauchy's inequality; Lemma \ref{lem:gronwall} and \eqref{eq:all1-2}. Therefore, we have shown
\begin{equation}\label{eq:diff_v0norm}
\norm{V_{u_1}(0,\cdot) - V_{u_2}(0,\cdot)}_{L^2}^2 \le C \norm{u_1-u_2}_{L^2}^2
\end{equation}
where this constant $C$ only depends on $K,n,T$. Also, \eqref{eq:diff_v0norm} holds with the same $C$ if the total time span $T$ decreases. Therefore, we can reformulate the control problem such that it start at $t \in (0,T)$ instead of $0$. Then the new state process starts at $x^i_t \sim \text{Unif}(\mX)$ and the new value function coincide with $V_{u_i}$ on $[t,T]$ by definition \eqref{eq:value}. We also remark that the constants $\rho_0, \rho_1$ in Proposition \ref{prop:rho} remain the same because $T$ decreases. Applying the argument for \eqref{eq:diff_v0norm} on the new control problem gives us
\begin{equation*}
\norm{V_{u_1}(t,\cdot) - V_{u_2}(t,\cdot)}_{L^2}^2 \le C \int_t^T \norm{u_1(s,\cdot)-u_2(s,\cdot)}_{L^2}^2 \rd s \le C \norm{u_1-u_2}_{L^2}^2.
\end{equation*}
Making an integration in $t$ gives us \eqref{eq:regVu_0}.

\emph{Step 2.} We want to show
\begin{equation}\label{eq:regVu_1}
\norm{\nx V_{u_1} - \nx V_{u_2}}_{L^2} \le C \norm{u_1 - u_2}_{L^2}.
\end{equation}
We recall that the first order adjoint equation is given by \eqref{eq:adjoint1}. We denote $p^i_t = -\nx V_{u_i}(t,x^i_t)$ and $q^i_t = -\nx^2 V_{u_i}(t,x^i_t) \sigma^i_t$ for $i=1,2$. They satisfy the equations
\begin{equation}\label{eq:adjoint2}
\left\{ \begin{aligned}
\rd p^i_t & = -\sqbra{(\nx b^i_t)\tp p^i_t + \nx \Tr\parentheses{(\sigma^i_t)\tp q^i_{t}} - \nx r^i_t} \rd t + q^i_t ~ \rd W_t \\
p^i_T & = -\nx h(x^i_T).
\end{aligned} \right.
\end{equation}
By assumption \ref{assump:basic}, we have $\abs{p^i_t} \le K$ and $\abs{q^i_t} \le K^2$.
By \eqref{eq:diff_V0},
$$\int_0^T \parentheses{(p^1_t)\tp \sigma^1_t - (p^2_t)\tp \sigma^2_t} \rd W_t = h(x^1_T)-h(x^2_T) + \int_0^T (r^1_t-r^2_t) \rd t - \parentheses{V_{u_1}(0,x_0) - V_{u_2}(0,x_0)}.$$
Taking a square expectation and using a Cauchy inequality, we obtain
\begin{equation}\label{eq:diff_psigma_square}
\begin{aligned}
& \quad \EE \int_0^T \abs{(p^1_t)\tp \sigma^1_t - (p^2_t)\tp \sigma^2_t}^2 \rd t = \EE \parentheses{\int_0^T \parentheses{(p^1_t)\tp \sigma^1_t - (p^2_t)\tp \sigma^2_t} \rd W_t}^2\\
& \le 3\EE \sqbra{\abs{h(x^1_T)-h(x^2_T)}^2 + T \int_0^T \abs{r^1_t-r^2_t}^2 \rd t + \abs{V_{u_1}(0,x_0) - V_{u_2}(0,x_0)}^2}\\
& \le 3L^2 \EE \abs{x^1_T - x^2_T}^2 + 3T ~ \EE \int_0^T \abs{r^1_t-r^2_t}^2 \rd t + 3 \norm{V_{u_1}(0,\cdot) - V_{u_2}(0,\cdot)}_{L^2}^2\\
& \le C \norm{u_1-u_2}_{L^2}^2,
\end{aligned}
\end{equation}
where the last inequality is because of the Gronwall inequality, estimate \eqref{eq:all1-2}, and the arguments in \emph{Step 1}. Also note that
$$(p^1_t)\tp \sigma^1_t - (p^2_t)\tp \sigma^2_t = (p^1_t)\tp \sigma^1_t - (p^2_t)\tp \sigma^1_t + (p^2_t)\tp \sigma^1_t - (p^2_t)\tp \sigma^2_t,$$
so
\begin{equation}\label{eq:diff_psigma1_square}
\abs{(p^1_t)\tp \sigma^1_t - (p^2_t)\tp \sigma^1_t}^2 \le 2\abs{(p^1_t)\tp \sigma^1_t - (p^2_t)\tp \sigma^2_t}^2 + 2K^2 \abs{\sigma^1_t - \sigma^2_t}^2.
\end{equation}
Therefore, taking an integration and expectation, we obtain
\begin{equation}\label{eq:diff_psquare}
\EE\int_0^T \abs{p^1_t-p^2_t}^2 \rd t \le \dfrac{1}{2 \sigma_0} \EE\int_0^T \abs{(p^1_t)\tp \sigma^1_t - (p^2_t)\tp \sigma^1_t}^2 \rd t \le C \norm{u_1-u_2}_{L^2}^2,
\end{equation}
where the first inequality is due to the uniform ellipticity assumption and the second is because of \eqref{eq:diff_psigma1_square}, \eqref{eq:diff_psigma_square}, and \eqref{eq:all1-2}.
Next, since
$$p^2_t - p^1_t = \nx V_{u_1}(t,x^1_t) - \nx V_{u_2}(t,x^1_t) + \nx V_{u_2}(t,x^1_t) - \nx V_{u_2}(t,x^2_t),$$
we have
\begin{equation}\label{eq:diff_nxV}
\begin{aligned}
& \quad \abs{\nx V_{u_1}(t,x^1_t) - \nx V_{u_2}(t,x^1_t)}^2 \\
&\le 2 \abs{p^1_t - p^2_t}^2 + 2 \abs{\nx V_{u_2}(t,x^1_t) - \nx V_{u_2}(t,x^2_t)}^2\\
&\le 2 \abs{p^1_t - p^2_t}^2 + 2 L^2 \abs{x^1_t-x^2_t}^2.
\end{aligned}
\end{equation}
Therefore,
\begin{equation*}
\begin{aligned}
& \quad \norm{\nx V_{u_1} - \nx V_{u_2}}_{L^2}^2 \le \dfrac{1}{\rho_0} \EE \int_0^T \abs{\nx V_{u_1}(t,x^1_t) - \nx V_{u_2}(t,x^1_t)}^2 \rd t \\
& \le C \EE \int_0^T \parentheses{\abs{p^1_t - p^2_t}^2 + \abs{x^1_t-x^2_t}^2} \rd t \le C \norm{u_1-u_2}_{L^2}^2,
\end{aligned}
\end{equation*}
where we have consecutively used: Proposition \ref{prop:rho}; equation \eqref{eq:diff_nxV}; equation \eqref{eq:diff_psquare} and Lemma \ref{lem:gronwall}. Therefore \eqref{eq:regVu_1} holds.

\emph{Step 3.} We want to show
\begin{equation}\label{eq:diff_qsquare}
\EE \int_0^T \abs{q^1_t - q^2_t}^2 \rd t \le C \norm{u_1-u_2}_{L^2}^2.
\end{equation}
The analysis for the second order derivative is the most difficult. We need to cut the interval $[0,T]$ into small pieces and estimate them separately. Let $N$ be an integer s.t. $\delta_0 := T/N \le 1/40K^2$. We denote $t_k = k \delta_0$ as time stamps. We will do the estimate for each interval $[t_{k-1},t_{k}]$ for $k=1,2,\ldots,N$. Computing the difference of the two adjoint equations \eqref{eq:adjoint2} gives us
\begin{equation}\label{eq:adjoint_diff}
\begin{aligned}
(q^1_t - q^2_t) \rd W_t & = \rd (p^1_t-p^2_t) - \parentheses{\nx r^1_t - \nx r^2_t} \rd t +  \sqbra{(\nx b^1_t)\tp p^1_t - (\nx b^2_t)\tp p^2_t} \rd t \\
& \quad + \sqbra{\nx \Tr\parentheses{(\sigma^1_t)\tp q^1_{t}} - \nx \Tr\parentheses{(\sigma^2_t)\tp q^2_{t}}}  \rd t.
\end{aligned}
\end{equation}

\emph{Step 3.1.} We consider $k=N$ first. Let $\delta \in [\delta_0, 2\delta_0]$. If we integrate \eqref{eq:adjoint_diff} on $[T-\delta,T]$  we obtain
\begin{equation*}
\begin{aligned}
& \int_{T-\delta}^T \parentheses{q^1_t - q^2_t} \rd W_t = -\parentheses{\nx h(x^1_T) - \nx h(x^2_T)} + \parentheses{p^1_{T-\delta} - p^2_{T-\delta}} - \int_{T-\delta}^T \parentheses{\nx r^1_t - \nx r^2_t} \rd t \\
& \quad + \int_{T-\delta}^T \sqbra{(\nx b^1_t)\tp p^1_t - (\nx b^2_t)\tp p^2_t} \rd t 
+ \int_{T-\delta}^T \sqbra{\nx \Tr\parentheses{(\sigma^1_t)\tp q^1_{t}} - \nx \Tr\parentheses{(\sigma^2_t)\tp q^2_{t}}}  \rd t.
\end{aligned}
\end{equation*}
We take a square expectation, apply Cauchy's inequalities, and get
\begin{equation}\label{eq:5terms_diffq}
\begin{aligned}
& \EE \int_{T-\delta}^T \abs{q^1_t - q^2_t}^2 \rd t \le 5 \EE \left[  \abs{\nx h(x^1_T) - \nx h(x^2_T)}^2 + \abs{p^1_{T-\delta} - p^2_{T-\delta}}^2  \right. \\
& \quad + \delta\int_{T-\delta}^T \abs{\nx r^1_t - \nx r^2_t}^2 \rd t + \delta \int_{T-\delta}^T \abs{(\nx b^1_t)\tp p^1_t - (\nx b^2_t)\tp p^2_t}^2 \rd t \\
& \quad \left. + \delta \int_{T-\delta}^T \abs{\nx \Tr\parentheses{(\sigma^1_t)\tp q^1_{t}} - \nx \Tr\parentheses{(\sigma^2_t)\tp q^2_{t}}}^2 \rd t \right] \\
& =: (\rom{1})+ (\rom{2})+ (\rom{3})+ (\rom{4})+ (\rom{5}).
\end{aligned}
\end{equation}
Next, we bound these terms one by one. For $(\rom{1})$, we have
\begin{equation}\label{eq:term1_diffq}
(\rom{1}) \le 5L^2 \EE \abs{x^1_T-x^2_T}^2 \le 5L^2C_1 \norm{u_1 - u_2}_{L^2}^2,
\end{equation}
where we used the Lipschitz condition in Assumption \ref{assump:basic} and Lemma \ref{lem:gronwall}. We skip $(\rom{2})$. For $(\rom{3})$, we have
\begin{equation}\label{eq:term3_diffq}
(\rom{3}) \le C \norm{u_1 - u_2}_{L^2}^2
\end{equation}
because of \eqref{eq:all1-2}. For $(\rom{4})$, we have
\begin{equation}\label{eq:term4_diffq}
\begin{aligned}
& \quad (\rom{4}) = 5\delta \,\EE \int_{T-\delta}^T \abs{(\nx b^1_t)\tp p^1_t - (\nx b^2_t)\tp p^2_t}^2 \rd t \\
& \le 10 \delta \,\EE \int_{T-\delta}^T \parentheses{\abs{(\nx b^1_t)\tp p^1_t - (\nx b^1_t)\tp p^2_t}^2 + \abs{(\nx b^1_t)\tp p^2_t - (\nx b^2_t)\tp p^2_t}^2}\rd t \\
& \le 10 \delta \,\EE \int_{T-\delta}^T K^2 \parentheses{ \abs{p^1_t - p^2_t}^2 + \abs{\nx b^1_t-\nx b^2_t}^2}\rd t \le C \norm{u_1-u_2}_{L^2}^2,
\end{aligned}
\end{equation}
where the second inequality is because of the boundedness of $\nx b$ and $\nx V_{u_i}$ and the third is because of \eqref{eq:diff_psquare} and \eqref{eq:all1-2}. For $(\rom{5})$, we have
\begin{equation}\label{eq:term5_diffq}
\begin{aligned}
& \quad (\rom{5}) = 5\delta \,\EE \int_{T-\delta}^T \abs{\nx \Tr\parentheses{(\sigma^1_t)\tp q^1_{t}} - \nx \Tr\parentheses{(\sigma^2_t)\tp q^2_{t}}}^2 \rd t\\
& \le 10 \delta \,\EE \int_{T-\delta}^T \left( \abs{\nx \Tr\parentheses{(\sigma^1_t)\tp q^1_{t}} - \nx \Tr\parentheses{(\sigma^1_t)\tp q^2_{t}}}^2 \right. \\
& ~~~~~~~~~~~~ \left. + \abs{\nx \Tr\parentheses{(\sigma^1_t)\tp q^2_{t}} - \nx \Tr\parentheses{(\sigma^2_t)\tp q^2_{t}}}^2 \right) \rd t \\
& \le 10 \delta \,\EE \int_{T-\delta}^T \parentheses{K^2 \abs{ q^1_{t} - q^2_{t}}^2 + K^4 \abs{\sigma^1_t - \sigma^2_t}} \rd t \\
& \le \dfrac{1}{2} \EE \int_{T-\delta}^T \abs{q^1_t - q^2_t}^2 \rd t + C\norm{u_1-u_2}_{L^2}^2,
\end{aligned}
\end{equation}
where the second inequality is because boundedness of $\nx \sigma$ and $\nx^2 V_{u_i}$ and the third is because $\delta \le 2\delta_0 \le 1/20K^2$ and \eqref{eq:all1-2}. Substituting \eqref{eq:term1_diffq}, \eqref{eq:term3_diffq}, \eqref{eq:term4_diffq}, and \eqref{eq:term5_diffq} into \eqref{eq:5terms_diffq}, we obtain
\begin{equation}\label{eq:diff_qtemp}
\dfrac{1}{2} \EE \int_{T-\delta}^T \abs{q^1_t - q^2_t}^2 \rd t \le 5\delta \EE \abs{p^1_{T-\delta} - p^2_{T-\delta}}^2 + C\norm{u_1-u_2}_{L^2}^2.
\end{equation}
Integrating \eqref{eq:diff_qtemp} w.r.t. $\delta$ on $[\delta_0, 2\delta_0]$, we obtain
\begin{equation*}
\begin{aligned}
& \quad \dfrac{1}{2} \delta_0 \EE \int_{T-\delta_0}^T \abs{q^1_t - q^2_t}^2 \rd t \le \int_{\delta_0}^{2\delta_0} \eqref{eq:diff_qtemp} \text{LHS} \,\rd \delta \le \int_{\delta_0}^{2\delta_0} \eqref{eq:diff_qtemp} \text{RHS} \,\rd \delta \\
& \le C \EE \int_{T-2\delta_0}^{T-\delta_0} \abs{p^1_t - p^2_t}^2 \rd t + C\norm{u_1-u_2}_{L^2}^2 \le C\norm{u_1-u_2}_{L^2}^2,
\end{aligned}
\end{equation*}
where the last inequality is due to \eqref{eq:diff_psquare}. Therefore, we have
\begin{equation}\label{eq:diff_q_k=N}
\EE \int_{T-\delta_0}^T \abs{q^1_t - q^2_t}^2 \rd t \le C\norm{u_1-u_2}_{L^2}^2.
\end{equation}

\emph{Step 3.2.} We consider $k=2,3,\ldots,N-1$ next. Let $\delta \in [0,\delta_0]$. We integrate \eqref{eq:adjoint_diff} on $[t_{k-1}-\delta, t_{k+1} - \delta]$ and obtain
\begin{equation*}
\begin{aligned}
& \int_{t_{k-1}-\delta}^{t_{k+1}-\delta} \parentheses{q^1_t - q^2_t} \rd W_t = \parentheses{p^1_{t_{k+1}-\delta} - p^2_{t_{k+1}-\delta}} - \parentheses{p^1_{t_{k-1}-\delta} - p^2_{t_{k-1}-\delta}} \\
& \quad - \int_{t_{k-1}-\delta}^{t_{k+1}-\delta} \parentheses{\nx r^1_t - \nx r^2_t} \rd t + \int_{t_{k-1}-\delta}^{t_{k+1}-\delta} \sqbra{(\nx b^1_t)\tp p^1_t - (\nx b^2_t)\tp p^2_t} \rd t \\
& \quad + \int_{t_{k-1}-\delta}^{t_{k+1}-\delta} \sqbra{\nx \Tr\parentheses{(\sigma^1_t)\tp q^1_{t}} - \nx \Tr\parentheses{(\sigma^2_t)\tp q^2_{t}}}  \rd t.
\end{aligned}
\end{equation*}
We take a square expectation, apply Cauchy's inequalities, and get
\begin{equation}\label{eq:5terms_diffq2}
\begin{aligned}
& \EE \int_{t_{k-1}-\delta}^{t_{k+1}-\delta} \abs{q^1_t - q^2_t}^2 \rd t \le 5 \EE \left[  \abs{p^1_{t_{k+1}-\delta} - p^2_{t_{k+1}-\delta}}^2 + \abs{p^1_{t_{k-1}-\delta} - p^2_{t_{k-1}-\delta}}^2  \right.\\
& + 2\delta_0 \int_{t_{k-1}-\delta}^{t_{k+1}-\delta} \abs{\nx r^1_t - \nx r^2_t}^2 \rd t + 2\delta_0 \int_{t_{k-1}-\delta}^{t_{k+1}-\delta} \abs{(\nx b^1_t)\tp p^1_t - (\nx b^2_t)\tp p^2_t}^2 \rd t\\
& \quad \left. + 2\delta_0 \int_{t_{k-1}-\delta}^{t_{k+1}-\delta} \abs{\nx \Tr\parentheses{(\sigma^1_t)\tp q^1_{t}} - \nx \Tr\parentheses{(\sigma^2_t)\tp q^2_{t}}}^2 \rd t \right]\\
& =: (\rom{1})+ (\rom{2})+ (\rom{3})+ (\rom{4})+ (\rom{5})
\end{aligned}
\end{equation}
with a little abuse of notation for the five terms in \eqref{eq:5terms_diffq}. We bound these five terms next. The techniques are exactly the same as in \emph{Step3.1}. We keep $(\rom{1})$ and $(\rom{2})$ unchanged. $(\rom{3})$ also satisfies \eqref{eq:term3_diffq} with the same reason. $(\rom{4})$ also satisfies \eqref{eq:term4_diffq}, where we only need to modify the interval for integration in the intermediate steps. For $(\rom{5})$, using the same argument in \eqref{eq:term4_diffq}, we obtain
$$(\rom{5}) \le \dfrac{1}{2} \EE \int_{t_{k-1}-\delta}^{t_{k+1}-\delta} \abs{q^1_t - q^2_t}^2 \rd t + C\norm{u_1-u_2}_{L^2}^2.$$
Combining the estimates into \eqref{eq:5terms_diffq2}, we obtain
\begin{equation}\label{eq:diff_qtemp2}
\begin{aligned}
& \quad \dfrac{1}{2} \EE \int_{t_{k-1}-\delta}^{t_{k+1}-\delta} \abs{q^1_t - q^2_t}^2 \rd t \\
& \le 5\delta \EE \sqbra{ \abs{p^1_{t_{k+1}-\delta} - p^2_{t_{k+1}-\delta}}^2 + \abs{p^1_{t_{k-1}-\delta} - p^2_{t_{k-1}-\delta}}^2 } + C\norm{u_1-u_2}_{L^2}^2.
\end{aligned}
\end{equation}
Integrating \eqref{eq:diff_qtemp2} w.r.t. $\delta$ on $[0, \delta_0]$, we obtain
\begin{equation*}
\begin{aligned}
& \quad \dfrac{1}{2} \delta_0 \EE \int_{t_{k-1}}^{t_k} \abs{q^1_t - q^2_t}^2 \rd t \le \int_{0}^{\delta_0} \eqref{eq:diff_qtemp2} \text{LHS} \,\rd \delta \le \int_{0}^{\delta_0} \eqref{eq:diff_qtemp2} \text{RHS} \,\rd \delta \\
& \le C \EE \parentheses{\int_{t_{k-2}}^{t_{k-1}} + \int_{t_{k}}^{t_{k+1}}} \abs{p^1_t - p^2_t}^2 \rd t + C\norm{u_1-u_2}_{L^2}^2 \le C\norm{u_1-u_2}_{L^2}^2,
\end{aligned}
\end{equation*}
where the last inequality is due to \eqref{eq:diff_psquare}. Therefore, for $k=2,3,\ldots,N-1$, we have
\begin{equation}\label{eq:diff_q_k=2}
\EE \int_{t_{k-1}}^{t_k} \abs{q^1_t - q^2_t}^2 \rd t \le C\norm{u_1-u_2}_{L^2}^2.
\end{equation}

\emph{Step 3.3.} We consider $k=1$ next. In this case we cannot integrate \eqref{eq:adjoint_diff} on $[t_{k-1}-\delta, t_{k+1} - \delta]$ because $t$ should be non-negative. But we can repeat the argument in \emph{Step 3.2}, with only a slight modification of our model. We extend the value function $V(t, x)$ to $t \in [-\delta_0, 0)$ by considering a modification of the control problem starting at $-\delta_0$ instead of time $0$.

We give detailed description of this extension to confirm that it works. Let us use a ``hat'' notation to denote the quantities for the new control problem. Firstly, the control functions $u_i(t,x)$ need to be extended to $[-\delta_0,T] \times \mX$ such that $\widehat{u}_i(t,x) = u_i(t,x)$ on $[0,T] \times \mX$. By definition \eqref{eq:value}, the new value functions $\widehat{V}_{u_i}: [-\delta_0, T] \times \mX \to \RR$ coincide with $V_{u_i}$ on $[0,T] \times \mX$. We also require the extension of $u_i$ to be smooth such that the bounds for control functions in Assumption \ref{assump:u_smooth} still hold and 
$$\norm{\widehat{u}_1 - \widehat{u}_2}_{L^2}^2 \le 2 \norm{u_1-u_2}_{L^2}^2.$$
The bounds for the value function (obtained from Schauder estimate) should still hold. The new state process start at $\widehat{x}^i_{-\delta_0} \sim \text{Unif}(\mX)$. The bounds $\rho_0, \rho_1$ in Proposition \ref{prop:rho} may need to change because the total time span is increased from $T$ to $T+\delta_0$, but they are still absolute constants if we follow the proof for Proposition \ref{prop:rho}. The Gronwall inequality should also hold, but with a larger constant $C_1$. Inequality \eqref{eq:all1-2} should still hold, with interval to be integrated replaced by $[-\delta_0 , T]$.

With these clarifications, we can repeat the arguments in \emph{Step 3.2} and obtain
\begin{equation}\label{eq:diff_q_k=1}
\EE \int_{0}^{\delta_0} \abs{\widehat{q}^1_t - \widehat{q}^2_t}^2 \rd t \le C \norm{\widehat{u}_1 - \widehat{u}_2}_{L^2}^2 \le 2C \norm{u_1-u_2}_{L^2}^2.
\end{equation}
Therefore, we did not perfectly recover \eqref{eq:diff_qsquare}, but the results \eqref{eq:diff_q_k=N}, \eqref{eq:diff_q_k=2}, and \eqref{eq:diff_q_k=1} are enough for us to proceed the next step. We will finish proving \eqref{eq:diff_qsquare} later.

\emph{Step 4.} We want to show
\begin{equation}\label{eq:regVu_2}
\norm{\nx^2 V_{u_1} - \nx^2 V_{u_2}}_{L^2} \le C \norm{u_1 - u_2}_{L^2}.
\end{equation}
Since
\begin{equation}\label{eq:diff_q_3terms}
\begin{aligned}
q^1_t - q^2_t & = -\nx^2 V_{u_1}(t,x^1_t) \sigma^1_t + \nx^2 V_{u_2}(t,x^2_t) \sigma^2_t = -\nx^2 V_{u_1}(t,x^1_t) \sigma^1_t + \nx^2 V_{u_2}(t,x^1_t) \sigma^1_t\\
& -\nx^2 V_{u_2}(t,x^1_t) \sigma^1_t + \nx^2 V_{u_2}(t,x^2_t) \sigma^1_t -\nx^2 V_{u_2}(t,x^2_t) \sigma^1_t + \nx^2 V_{u_2}(t,x^2_t) \sigma^2_t,
\end{aligned}
\end{equation}
we have
\begin{equation}\label{eq:diff_Vxxtemp}
\begin{aligned}
& \quad \abs{\nx^2 V_{u_1}(t,x^1_t) \sigma^1_t - \nx^2 V_{u_2}(t,x^1_t) \sigma^1_t}^2 \le 3 \abs{-\nx^2 V_{u_2}(t,x^1_t) \sigma^1_t - \nx^2 V_{u_2}(t,x^2_t) \sigma^1_t}^2 \\
& \hspace{1.8in} 3\abs{\nx^2 V_{u_2}(t,x^2_t) \sigma^1_t - \nx^2 V_{u_2}(t,x^2_t) \sigma^2_t}^2 + 3\abs{q^1_t - q^2_t}^2 \\
& \le 3K^2 L^2 \abs{x^1_t-x^2_t}^2 + 3K^2 \abs{\sigma^1_t - \sigma^2_t}^2 + 3\abs{q^1_t - q^2_t}^2
\end{aligned}
\end{equation}
Therefore,
\begin{equation}\label{eq:diff_Vxx1}
\begin{aligned}
& \quad \int_{\delta_0}^T \norm{\nx^2 V_{u_1}(t,\cdot) - \nx^2 V_{u_2}(t,\cdot)}_{L^2}^2 \rd t \le \dfrac{1}{\rho_0} \EE \int_{\delta_0}^T \abs{\nx^2 V_{u_1}(t,x^1_t) - \nx^2 V_{u_2}(t,x^1_t)}^2 \rd t \\
& \le \dfrac{1}{2\sigma_0 \rho_0} \EE \int_{\delta_0}^T \abs{\nx^2 V_{u_1}(t,x^1_t) \sigma^1_t - \nx^2 V_{u_2}(t,x^1_t) \sigma^1_t}^2 \rd t \\
& \le C \EE \int_{\delta_0}^T \parentheses{\abs{x^1_t-x^2_t}^2 + \abs{\sigma^1_t - \sigma^2_t}^2 + \abs{q^1_t - q^2_t}^2} \rd t \le C \norm{u_1-u_2}_{L^2}^2.
\end{aligned}
\end{equation}
where we have consecutively used: Proposition \ref{prop:rho}; uniform ellipticity of $\sigma$; equation \eqref{eq:diff_Vxxtemp}; Lemma \ref{lem:gronwall}, equations \eqref{eq:all1-2}, \eqref{eq:diff_q_k=2} and \eqref{eq:diff_q_k=N}. Applying the same argument to the new control problem that start at $t=-\delta_0$, we obtain
\begin{equation}\label{eq:diff_Vxx2}
\int_0^{\delta_0} \norm{\nx^2 V_{u_1}(t,\cdot) - \nx^2 V_{u_2}(t,\cdot)}_{L^2}^2 \rd t \le C \norm{\widehat{u}_1 - \widehat{u}_2}_{L^2}^2 \le 2C \norm{u_1-u_2}_{L^2}^2.
\end{equation}
Combining \eqref{eq:diff_Vxx1} and \eqref{eq:diff_Vxx2}, we obtain \eqref{eq:regVu_2}. As a follow up, with \eqref{eq:regVu_2} holds, we can use \eqref{eq:diff_q_3terms} to bound $\abs{q^1_t - q^2_t}$, and obtain \eqref{eq:diff_qsquare}. Therefore, the result for \emph{Step 3} is perfectly proved.

Finally, combining \eqref{eq:regVu_0}, \eqref{eq:regVu_1}, and \eqref{eq:regVu_2}, we get \eqref{eq:regularity_Vu}, so the lemma is proved. We remark that we can also write down the second order adjoint equation (see \cite{peng1990general} for example) and prove that
\begin{equation*}
\norm{\nx^3 V_{u_1} - \nx^3 V_{u_2}}_{L^2} \le C \norm{u_1 - u_2}_{L^2},
\end{equation*}
using the same method in \emph{Step 3-4}.
\end{proof}

\begin{lem}\label{lem:J_quadratic}
Under Assumption \ref{assump:basic}, there exists a positive constant $C_3$ s.t.
\begin{equation}\label{eq:J_quadratic}
J[u]- J[u^*] \le C_3 \norm{u-u^*}_{L^2}^2
\end{equation}
for any $u \in \mU$.
\end{lem}
\begin{proof}
Denote $\ve_0 = \norm{u-u^*}_{L^2}$ and let $u = u^*+\ve_0\phi$, then $\norm{\phi}_{L^2}=1$. We denote $u^{\ve} = u^* + \ve \phi$. Denote the corresponding value function $V_{u^{\ve}}$ by $V^{\ve}$. Denote the corresponding density function by $\rho^{\ve}$, with initial condition $\rho^{\ve}(0,\cdot) \equiv 1$.
By Proposition \ref{prop:cost_derivative},
\begin{equation}\label{eq:dJdve}
\begin{aligned}
& \quad \dfrac{\rd}{\rd \ve}J[u^{\ve}] = \inner{\fd{J}{u}[u^{\ve}]}{\phi}_{L^2} \\
& = -\int_0^T \int_{\mX} \rho^{\ve}(t,x) \inner{\nbu G(t,x,u^{\ve}(t,x),-\nx V^{\ve},-\nx^2 V^{\ve})}{\phi(t,x)} \rd x \, \rd t.
\end{aligned}
\end{equation}

In order to show \eqref{eq:J_quadratic}, it is sufficient to show that $\dfrac{\rd}{\rd \ve} J[u^{\ve}] \le C\ve$ for some uniform constant $C$ (that does not depend on $\phi$), because
$$J[u^*+\ve_0\phi] - J[u^*] = \int_0^{\ve_0} \dfrac{\rd}{\rd \ve}J[u^{\ve}] \, \rd \ve.$$
We estimate $\nbu G$ in \eqref{eq:dJdve} first. 
\begin{equation}\label{eq:nuGve}
\begin{aligned}
& \quad \abs{\nbu G(t,x,u^{\ve}(t,x),-\nx V^{\ve},-\nx^2 V^{\ve})} \\
& = \abs{\nbu G(t,x,u^{\ve}(t,x),-\nx V^{\ve},-\nx^2 V^{\ve}) - \nbu G(t,x,u^{*}(t,x),-\nx V^*,-\nx^2 V^*)}\\
& \le \abs{\nbu G(t,x,u^{\ve}(t,x),-\nx V^{\ve},-\nx^2 V^{\ve}) - \nbu G(t,x,u^{*}(t,x),-\nx V^{\ve},-\nx^2 V^{\ve})}\\
& \quad + \abs{\nbu G(t,x,u^{*}(t,x),-\nx V^{\ve},-\nx^2 V^{\ve}) - \nbu G(t,x,u^{*}(t,x),-\nx V^*,-\nx^2 V^*)}\\
& =: (\rom{1}) + (\rom{2}),
\end{aligned}
\end{equation}
where we used the maximum condition \eqref{eq:max1} in the first equality. Recall that we have denoted $D=\frac12\sigma\sigma\tp$. Let us also denote $u^{\ve}(t,x)$ and $u^*(t,x)$ by $u^{\ve}$ and $u^*$ for simplicity. For $(\rom{1})$, we have
\begin{equation}\label{eq:lem3rom1}
\begin{aligned}
(\rom{1}) &\le \abs{\nbu r(x,u^{\ve}) - \nbu r(x,u^*)} +  \abs{\parentheses{\nbu b(x,u^{\ve}) - \nbu b(x,u^*)}\tp \nx V^{\ve}}\\
& \quad + \abs{\nbu \Tr\sqbra{ \parentheses{D(x,u^{\ve}) - D(x,u^*)} \nx^2 V^{\ve}}} \\
& \le L \ve \abs{\phi(t,x)} + L \ve \abs{\phi(t,x)}K + L \ve \abs{\phi(t,x)}K \le C \ve \abs{\phi(t,x)},
\end{aligned}
\end{equation}
where we have used the Lipschitz conditions in Assumption \ref{assump:basic} and boundedness of the value function's derivatives. For $(\rom{2})$, we have
\begin{equation}\label{eq:lem3rom2}
\begin{aligned}
& \quad (\rom{2}) \le \abs{ \nbu b(x,u^*) \tp \parentheses{\nx V^{\ve} - \nx V^*} } + \abs{\nbu \Tr\sqbra{D(x,u^*) \parentheses{ \nx^2 V^{\ve} - \nx^2 V^*}}}\\
& \le K \parentheses{\abs{\nx V^{\ve} - \nx V^*} + \abs{\nx^2 V^{\ve} - \nx^2 V^*}}.
\end{aligned}
\end{equation}
Combining \eqref{eq:lem3rom1} and \eqref{eq:lem3rom2} into \eqref{eq:nuGve}, we obtain
\begin{equation}\label{eq:bound_nuGve}
\begin{aligned}
& \quad \abs{\nbu G(t,x,u^{\ve}(t,x),-\nx V^{\ve},-\nx^2 V^{\ve})} \\
& \le C \parentheses{\ve\abs{\phi(t,x)} + \abs{\nx V^{\ve} - \nx V^*} + \abs{\nx^2 V^{\ve} - \nx^2 V^*}}.
\end{aligned}
\end{equation}
Therefore, \eqref{eq:dJdve} has estimate
\begin{equation*}
\begin{aligned}
& \quad \abs{\dfrac{\rd}{\rd \ve}J[u^{\ve}]} \\
& \le \rho_1 \int_0^T\int_{\mX} \parentheses{\dfrac{1}{\ve} \abs{\nbu G(t,x,u^{\ve}(t,x),-\nx V^{\ve},-\nx^2 V^{\ve})}^2 + \ve \abs{\phi(t,x)}^2} \rd x \,\rd t \\
& \le C \parentheses{\ve \norm{\phi}_{L^2}^2 + \dfrac{1}{\ve} \norm{\nx V^{\ve} - \nx V^*}_{L^2}^2 + \dfrac{1}{\ve} \norm{\nx^2 V^{\ve} - \nx^2 V^*}_{L^2}^2 + \ve \norm{\phi}_{L^2}^2} \le C\ve,
\end{aligned}
\end{equation*}
where we have consecutively used: Proposition \ref{prop:rho} and Cauchy's inequality; inequality \eqref{eq:bound_nuGve}; Lemma \ref{lem:regularity_Vu}. Therefore, \eqref{eq:J_quadratic} holds.
\end{proof}

\begin{lem}\label{lem:quadratic_Vu}
Under Assumption \ref{assump:basic}, there exists a positive constant $C_4$ s.t. for any control function $u \in \mU$, we have
\begin{equation}\label{eq:quadratic_Vu}
\normHtwo{V_{u} - V^*} \le C_4 \norm{u-u^*}_{L^2}^{1+\alpha},
\end{equation}
with $\alpha = \frac{1}{n+3}$.
\end{lem}
\begin{rmk}
We believe that \eqref{eq:quadratic_Vu} holds with $\alpha=1$, but encountered some technical difficulty to prove it. We give the intuition here. Following the notation in the previous lemma, we denote $u^{\ve} = u^* + \ve \phi$. Since $V^{\ve}(t,x)$ reaches its minimum at $\ve=0$ for any $(t,x)$, we have
$$\pve V^{\ve}(t,x) \big|_{\ve=0} =0.$$
With sufficient regularity, we have
$$\parentheses{\pve \nx V^{\ve}}|_{\ve=0} = \parentheses{\nx \pve V^{\ve}}|_{\ve=0} = \nx \parentheses{\pve V^{\ve}}|_{\ve=0} = 0,$$
$$\parentheses{\pve \nx^2 V^{\ve}}|_{\ve=0} = \parentheses{\nx^2 \pve V^{\ve}}|_{\ve=0} = \nx^2 \parentheses{\pve V^{\ve}}|_{\ve=0} = 0.$$
Making a local Taylor expansion w.r.t. $\ve$, we know that $\nx V^{\ve} - \nx V^*$ and $\nx^2 V^{\ve} - \nx^2 V^*$ are of order $\mO(\ve^2)$, which implies \eqref{eq:quadratic_Vu} holds with $\alpha=1$.
\end{rmk}
\begin{proof}
We will inherit some notations from the previous lemma. Denote $\ve_0 = \norm{u-u^*}_{L^2}$ and let $u = u^*+\ve_0\phi$, then $\norm{\phi}_{L^2}=1$. We denote $u^{\ve} = u^* + \ve \phi$. Denote the corresponding value function $V_{u^{\ve}}$ by $V^{\ve}$. Denote the corresponding density function by $\rho^{\ve}$, with initial condition $\rho^{\ve}(0,\cdot) \equiv 1$. The key difficulty for the proof is that $\phi(t,x)$ may not lie in $\mU$ like $u^*(t,x)$ or $u^{\ve}(t,x)$, which has $K$ as a bound for itself and its derivatives. $\phi = (u^{\ve} - u^*)/ \ve$ do have some regularity, but the constant for the bounds have a factor of $2/\ve$. We will prove the lemma in three steps, which is similar to Lemma \ref{lem:regularity_Vu}.

\emph{Step 1.} We want to show
\begin{equation}\label{eq:lem4step1}
\norm{V_u - V^*}_{L^2} \le C \norm{u - u^*}_{L^2}^{1+\alpha}.
\end{equation}
Note that $V_u \ge V^*$, so
\begin{equation*}
\begin{aligned}
& \quad \int_{\mX} \abs{V_u(0,x)-V^*(0,x)} \rd x = \int_{\mX} \parentheses{V_u(0,x)-V^*(0,x)} \rd x \\
& = J[u] - J[u^*] \le C_3 \norm{u-u^*}_{L^2}^2,
\end{aligned}
\end{equation*}
where we used Lemma \ref{lem:J_quadratic} in the last inequality. i.e., $\norm{V^{\ve}(0,\cdot) - V^*(0,\cdot)}_{L^1} \le C_3\norm{u-u^*}_{L^2}^2$. A similar argument with $t \in (0,T)$ as the starting time gives us $\norm{V^{\ve}(t,\cdot) - V^*(t,\cdot)}_{L^1} \le C\norm{u-u^*}_{L^2}^2$. Therefore, we have
$$\norm{V^{\ve} - V^*}_{L^1} \le C\norm{u-u^*}_{L^2}^2,$$
which implies \eqref{eq:lem4step1} because $V^{\ve}$, $V^*$, $u$, $u^*$ are bounded.

\emph{Step 2.} We want to show
\begin{equation}\label{eq:lem4step2}
\norm{\nx V_u - \nx V^*}_{L^2} \le C \norm{u - u^*}_{L^2}^{1+\alpha}.
\end{equation}
It is sufficient to show the partial derivative in each dimension at $t=0$ satisfies the estimate in $L^1$ norm:
\begin{equation}\label{eq:lem4step2temp}
\norm{\partial_i V_u(0,\cdot) - \partial_i V^*(0,\cdot)}_{L^1} \le C \norm{u - u^*}_{L^2}^{1+\alpha},
\end{equation}
because we can repeat the argument in other dimensions and for other $t \in (0,T)$, and the derivatives of the value functions are bounded.

\emph{Step 2.1.} We reformulate the problem using finite difference in this step. Let $x_1 \in \mX$ be a variable and denote $x_2 = x_1 + \delta e_i$ a perturbation. We assume $\delta >0$ without loss of generality. We have
\begin{equation}\label{eq:lem4_s2d1}
\begin{aligned}
& \quad \norm{\partial_i V_u(0,\cdot) - \partial_i V^*(0,\cdot)}_{L^1} = \int_{\mX} \abs{\partial_i V_u(0,x_1) - \partial_i V^*(0,x_1)} \rd x_1 \\
& = \int_{\mX} \abs{ \int_0^{\ve_0} \pve \partial_i V^{\ve}(0,x_1) \rd\ve } \rd x_1 = \int_{\mX} \abs{ \int_0^{\ve_0} \partial_i \pve  V^{\ve}(0,x_1) \rd\ve} \rd x_1 \\
&= \int_{\mX} \abs{ \int_0^{\ve_0} \lim_{\delta \to 0} \dfrac{1}{\delta} \parentheses{\pve  V^{\ve}(0,x_2) - \pve  V^{\ve}(0,x_1)} \rd\ve } \rd x_1 \\
& \le \liminf_{\delta \to 0} \dfrac{1}{\delta}  \int_0^{\ve_0} \int_{\mX} \abs{ \pve  V^{\ve}(0,x_2) - \pve  V^{\ve}(0,x_1)}\rd x_1 \, \rd\ve,
\end{aligned}
\end{equation}
where the last inequality is because of Fatou's lemma. Now, we denote $\xyet$ and $\xeet$ the state processes under control $u^{\ve}$ that start at $x^{1,\ve}_0=x_1$ and $x^{2,\ve}_0=x_2$. Here $\xyet$ and $\xeet$ share the same realization of Brownian motion. By Proposition \ref{prop:rho},
\begin{align*}
-\pve V^{\ve}(0,x_1) &= \EE \sqbra{\int_0^T \inner{\nbu G(t, \xyet, u^{\ve}(t,\xyet), -\nx V^{\ve}, -\nx^2 V^{\ve})}{\phi(t,x^{1,\ve}_t)} \rd t }\\
&=:  \EE \sqbra{\int_0^T \inner{\nbu G^{1,\ve}_t}{\phi^{1,\ve}_t} \rd t }.
\end{align*}
Similarly, $-\pve V^{\ve}(0,x_2) = \EE \sqbra{\int_0^T \inner{\nbu G^{2,\ve}_t}{\phi^{2,\ve}_t} \rd t }$.
So
\begin{equation}\label{eq:lem4_s2d1a}
\begin{aligned}
& \quad \int_{\mX} \abs{  \pve  V^{\ve}(0,x_2) - \pve  V^{\ve}(0,x_1)}\rd x_1 \\
& = \int_{\mX} \abs{ \EE \sqbra{\int_0^T \inner{\nbu G^{1,\ve}_t}{\phi^{1,\ve}_t} - \inner{\nbu G^{2,\ve}_t}{\phi^{2,\ve}_t} \rd t } }\rd x_1 
\end{aligned}
\end{equation}
By \eqref{eq:lem4_s2d1} and \eqref{eq:lem4_s2d1a}, in order to show \eqref{eq:lem4step2temp}, it is sufficient to show that
\begin{equation}\label{eq:lem4step2temp2}
\int_{\mX} \abs{ \EE \sqbra{\int_0^T \inner{\nbu G^{1,\ve}_t}{\phi^{1,\ve}_t} - \inner{\nbu G^{2,\ve}_t}{\phi^{2,\ve}_t} \rd t}}\rd x_1 \le C \delta \ve^{\alpha}
\end{equation}
for some uniform constant $C$. We can assume $\delta \le \ve$ because $\nbu G^{j,\ve}_t|_{\ve=0}=0$ ($j=1,2$), hence \eqref{eq:lem4step2temp2} is obvious when $\ve=0$.

\emph{Step 2.2.} We split into two sub-tasks to show \eqref{eq:lem4step2temp2} in this step. Let us denote $\rho^{1,\ve}(t,x)$ and $\rho^{2,\ve}(t,x)$ the density functions of $\xyet$ and $\xeet$, then $\rho^{j,\ve}(0,\cdot)=\delta_{x_j} \, (j=1,2)$ and a sufficient condition for \eqref{eq:lem4step2temp2} is
\begin{equation}\label{eq:lem4step2temp3}
\begin{aligned}
& \int_{\mX} \int_0^T \int_{\mX} \Big| \inner{\nbu G(t,x,u^{\ve}, -\nx V^{\ve}, -\nx^2 V^{\ve})}{\phi(t,x)} \\
& \hspace{0.8in} \parentheses{\rho^{1,\ve}(t,x) - \rho^{2,\ve}(t,x)} \Big| \, \rd x\, \rd t\,\rd x_1 \le C \delta \ve^{\alpha}.
\end{aligned}
\end{equation}
The idea to prove \emph{Step 2} is to decompose the time interval $[0,T]$ into two sub-intervals $[0,\ve^{2\alpha}]$ and $(\ve^{2\alpha},T]$ and prove \eqref{eq:lem4step2temp2} and \eqref{eq:lem4step2temp3} with $\int_0^T$ replaced by the corresponding intervals respectively. For the first part, we take advantage that $\ve^{2\alpha}$ is small, while for the second part, we we use the fact that $\rho^{1,\ve}(t,x)$ and $\rho^{2,\ve}(t,x)$ are nicely mixed.

\emph{Step 2.3.} We estimate the integration in the interval $[0,\ve^{2\alpha}]$ in this step. We want to show
\begin{equation}\label{eq:lem4step2temp4}
\int_{\mX} \EE \sqbra{\int_0^{\ve^{2\alpha}} \abs{  \inner{\nbu G^{1,\ve}_t}{\phi^{1,\ve}_t} - \inner{\nbu G^{2,\ve}_t}{\phi^{2,\ve}_t}  } \rd t}\rd x_1 \le C \delta \ve^{\alpha}.
\end{equation}
Using the Lipschitz property and boundedness of $\nbu r$, $\nbu b$, $\nbu D$, $\nx V^{\ve}$, and $\nx^2 V^{\ve}$, we can show
$$\abs{\nbu G^{1,\ve}_t - \nbu G^{2,\ve}_t} \le C \abs{\xyet-\xeet}.$$
Also, we have $\abs{\pyet-\peet} \le \dfrac{2L}{\ve} \abs{\xyet-\xeet}$. Therefore,
\begin{equation*}
\begin{aligned}
& \quad \abs{ \inner{\nbu G^{1,\ve}_t}{\phi^{1,\ve}_t} - \inner{\nbu G^{2,\ve}_t}{\phi^{2,\ve}_t} }\\
& = \abs{ \inner{\nbu G^{1,\ve}_t - \nbu G^{2,\ve}_t}{\phi^{1,\ve}_t} + \inner{\nbu G^{2,\ve}_t}{\phi^{1,\ve}_t-\phi^{2,\ve}_t} }\\
& \le \abs{\nbu G^{1,\ve}_t - \nbu G^{2,\ve}_t} \, \abs{\phi^{1,\ve}_t} + \abs{\nbu G^{2,\ve}_t} \, \abs{\phi^{1,\ve}_t-\phi^{2,\ve}_t}\\
& \le C \left(\abs{\pyet} + \abs{\peet} + \dfrac{1}{\ve} \abs{\nx V^{\ve}(t,\xeet) - \nx V^*(t,\xeet)} + \right.\\
& \hspace{0.5in} \left. \dfrac{1}{\ve} \abs{\nx^2 V^{\ve}(t,\xeet) - \nx^2 V^*(t,\xeet)} \right) \abs{\xyet-\xeet}
\end{aligned}
\end{equation*}
where we have used \eqref{eq:bound_nuGve} to estimate $\abs{\nbu G^{2,\ve}_t}$ in the last inequality. Substituting the estimate above into \eqref{eq:lem4step2temp4} left, we obtain
\begin{equation}\label{eq:lem4_step23}
\begin{aligned}
& \quad \int_{\mX} \EE \sqbra{\int_0^{\ve^{2\alpha}} \abs{  \inner{\nbu G^{1,\ve}_t}{\phi^{1,\ve}_t} - \inner{\nbu G^{2,\ve}_t}{\phi^{2,\ve}_t}  } \rd t}\rd x_1 \\
& \le C \int_{\mX} \EE \left[ \int_0^{\ve^{2\alpha}} \left(  \delta \ve^{\alpha} \abs{\pyet}^2 + \delta \ve^{\alpha} \abs{\peet}^2 + \delta\ve^{\alpha-2} \abs{\nx V^{\ve}(t,\xeet) - \nx V^*(t,\xeet)}^2 \right. \right. \\
& \hspace{0.4in} \left. \left. \delta\ve^{\alpha-2} \abs{\nx^2 V^{\ve}(t,\xeet) - \nx^2 V^*(t,\xeet)}^2 + \dfrac{1}{\delta \ve^{\alpha}} \abs{\xyet-\xeet}^2 \right) \rd t \right] \rd x_1 \\
& \le C \rho_1 \parentheses{ 2\delta \ve^{\alpha} \norm{\phi}_{L^2}^2 + \delta\ve^{\alpha-2} \normHtwo{V^{\ve}-V^*}^2 } + C\int_{\mX} \dfrac{1}{\delta \ve^{\alpha}} \int_0^{\ve^{2\alpha}} \EE \abs{\xyet-\xeet}^2 \rd t \, \rd x_1 \\
& \le C \parentheses{ \delta \ve^{\alpha} + \delta\ve^{\alpha-2} \ve^2  } + C \int_{\mX} \dfrac{1}{\delta \ve^{\alpha}} \int_0^{\ve^{2\alpha}} C_1 \abs{x_1-x_2}^2 \rd t \, \rd x_1 \\
& \le C \delta \ve^{\alpha} + \dfrac{C}{\delta \ve^{\alpha}} \ve^{2\alpha} \delta^2 \le C \delta \ve^{\alpha}.
\end{aligned}
\end{equation}
Here, the first inequality is just Cauchy's inequality. For the third inequality, we used Lemma \ref{lem:regularity_Vu} and the Gronwall inequality \eqref{eq:Gronwall2}. In the fourth inequality, we used $\abs{x_1-x_2}=\delta$. We give an explanation of the second inequality in \eqref{eq:lem4_step23} next. After confirming this second inequality, we get \eqref{eq:lem4step2temp4}.

We will use a similar argument many times later in this proof. Although $x^{1,\ve}_0=x_1$ and $x^{2,\ve}_0=x_2$ are fixed points, we are integrating $x_1$ over $\mX$ (with $x_2-x_1 = \delta e_i$ fixed). So, we can define two new processes $\overline{x}^{1,\ve}_t$ and $\overline{x}^{2,\ve}_t$ that have the same dynamic as $x^{1,\ve}_t$ and $x^{2,\ve}_t$, but start at uniform distribution in $\mX$, with $\overline{x}^{2,\ve}_t - \overline{x}^{1,\ve}_t \equiv \delta e_i$. The densities for $\overline{x}^{2,\ve}_t$ and $\overline{x}^{1,\ve}_t$ (denoted by $\overline{\rho}^{1,\ve}(t,x)$ and $\overline{\rho}^{2,\ve}(t,x)$) satisfies the estimate in Proposition \ref{prop:rho}. Therefore,
\begin{equation}\label{eq:view_uniform}
\begin{aligned}
& \quad \int_{\mX} \EE \sqbra{\int_0^{\ve^{2\alpha}} \abs{\phi^{1,\ve}_t}^2 \rd t} \rd x_1 \equiv \int_{\mX} \EE \sqbra{\int_0^{\ve^{2\alpha}} \abs{\phi(t,\xyet)}^2 \rd t ~\Big|~ x^{1,\ve}_0 = x_1} \rd x_1\\
&= \EE_{\overline{x}^{1,\ve}_0 \sim \text{Unif}(\mX)} \EE \sqbra{\int_0^{\ve^{2\alpha}} \abs{\phi(t, \overline{x}^{1,\ve}_t)}^2 \rd t ~\Big|~ \overline{x}^{1,\ve}_0} = \EE \int_0^{\ve^{2\alpha}} \abs{\phi(t, \overline{x}^{1,\ve}_t)}^2 \rd t \\
& = \int_{\mX} \int_0^{\ve^{2\alpha}} \abs{\phi(t, x)}^2 \overline{\rho}^{1,\ve}(t,x) \, \rd t \, \rd x \le \rho_1 \int_{\mX} \int_0^{\ve^{2\alpha}} \abs{\phi(t, x)}^2 \rd t \, \rd x \le \rho_1 \norm{\phi}_{L^2}^2.
\end{aligned}
\end{equation}
$\phi^{2,\ve}_t$ satisfies the same inequality. The analysis for the $\nx V$ and $\nx^2 V$ terms are exactly the same. Therefore, we can apply Proposition \ref{prop:rho}, and the second inequality in \eqref{eq:lem4_step23} holds. Hence, we confirm that \eqref{eq:lem4step2temp4} holds.

\emph{Step 2.4.} We estimate the integration in the interval $[\ve^{2\alpha},T]$ in this step. We want to show 
\begin{equation}\label{eq:lem4step2temp5}
\int_{\ve^{2\alpha}}^T \int_{\mX} \abs{  \inner{\nbu G(t,x,u^{\ve}, -\nx V^{\ve}, -\nx^2 V^{\ve})}{\phi(t,x)} \parentheses{\rho^{1,\ve}(t,x) - \rho^{2,\ve}(t,x)} }\rd x\, \rd t \le C \delta \ve^{\alpha}.
\end{equation}
We recall that $\rho$ is the solution of the Fokker--Planck equation $\partial_t \rho = \mG_{\ve}^{\dagger} \rho$, where $\mG_{\ve}$ is the infinitesimal generator of the state process with control $u^{\ve}$ and $\mG_{\ve}^{\dagger}$ is its adjoint. Let us use $p^{\ve}(t,x;s,y)$ ($t \ge s$) to denote the fundamental solution of this PDE. Then, $\rho^{j,\ve}(t,x) = p^{\ve}(t,x;0,x_j)$ for $j=1,2$. The fundamental solution of linear parabolic PDE is well-studied, and a comprehensive description can be found in \cite{friedman2008partial}. A key observation of the fundamental solution $p^{\ve}$ is that $q^{\ve}(t,x;s,y) := p^{\ve}(s,y;t,x)$ is the fundamental solution of the backward Kolmogorov equation $\partial_t \psi + \mG_{\ve} \psi = 0$ \cite{ito1953fundamental}. Therefore, the regularity of $p^{\ve}(t,x;s,y)$ in $y$ (here $t \ge s$) is equivalent to the regularity of $q^{\ve}(t,x;s,y)$ in $x$ (here $s \ge t$).  Aronson proved (in \cite{aronson1959fundamental} Lemma 4.2) that
\begin{equation}\label{eq:lemma4.2}
\abs{\nx^k q^{\ve}(t,x;s,y)} \le C^{(k)} (s-t)^{-(n+k)/2}.
\end{equation}
Applying a standard mean value theorem and this lemma \eqref{eq:lemma4.2} with $k=1$ to $q^{\ve}$, we obtain
\begin{equation}\label{eq:diff_rho}
\begin{aligned}
& \quad \abs{\rho^{1,\ve}(t,x) - \rho^{2,\ve}(t,x)} = \abs{ q^{\ve}(0,x_1;t,x) - q^{\ve}(0,x_2;t,x) } \\
& = \abs{\inner{\nx q^{\ve}(0,(1-c)x_1 + cx_2;t,x)}{x_1 - x_2} } \le C t^{-(1+n)/2} \abs{x_1-x_2} = C t^{-(1+n)/2} \delta,
\end{aligned}
\end{equation}
where we clarify that $\nx$ is operated on the second (not fourth) argument on $q^{\ve}(t,x;s,y)$. Therefore,
\begin{equation*}
\begin{aligned}
& \quad \int_{\ve^{2\alpha}}^T \int_{\mX} \abs{  \inner{\nbu G(t,x,u^{\ve}, -\nx V^{\ve}, -\nx^2 V^{\ve})}{\phi(t,x)} \parentheses{\rho^{1,\ve}(t,x) - \rho^{2,\ve}(t,x)} }\rd x\, \rd t \\
& \le C \int_{\ve^{2\alpha}}^T \int_{\mX} \parentheses{\ve\abs{\phi(t,x)} + \abs{\nx V^{\ve} - \nx V^*} + \abs{\nx^2 V^{\ve} - \nx^2 V^*}} \abs{\phi(t,x)} t^{-(1+n)/2} \delta \, \rd x\, \rd t \\
& \le C \delta \ve^{-\alpha-n\alpha} \int_{\ve^{2\alpha}}^T \int_{\mX} \parentheses{ 
\ve \abs{\phi(t,x)}^2 + \dfrac{1}{\ve} \abs{\nx V^{\ve} - \nx V^*}^2 + \dfrac{1}{\ve} \abs{\nx^2 V^{\ve} - \nx^2 V^*}^2 }\rd x\, \rd t \\
& \le C \delta \ve^{-\alpha-n\alpha} \parentheses{\ve \norm{\phi}_{L^2}^2 + \dfrac{1}{\ve} \normHtwo{V^\ve-V^*}^2} \le C \delta \ve^{1-\alpha-n\alpha} \le C \delta \ve^{\alpha}.
\end{aligned}
\end{equation*}
We used \eqref{eq:bound_nuGve} and \eqref{eq:diff_rho} in the first inequality, and used Lemma \ref{lem:regularity_Vu} in the thourth inequality. So, \eqref{eq:lem4step2temp5} holds.

To conclude, we combine \eqref{eq:lem4step2temp4} and \eqref{eq:lem4step2temp5} and recover \eqref{eq:lem4step2temp2}. Therefore, \eqref{eq:lem4step2temp}, hence \eqref{eq:lem4step2} holds.

\emph{Step 3.}  We want to show
\begin{equation}\label{eq:lem4step3}
\norm{\nx^2 V_u - \nx^2 V^*}_{L^2} \le C \norm{u - u^*}_{L^2}^{1+\alpha}.
\end{equation}
The idea to show \eqref{eq:lem4step3} is similar as in \emph{Step 2}. It is sufficient to show
\begin{equation}\label{eq:lem4step3temp0}
\norm{\nx^2 V_u(0,\cdot) - \nx^2 V^*(0,\cdot)}_{L^1} \le C \norm{u - u^*}_{L^2}^{1+\alpha}
\end{equation}
because the same argument applies to other $t\in(0,T)$. We will use the idea of finite difference and cut $[0,T]$ into two intervals with separate estimate.

\emph{Step 3.1.} We reformulate the problem using finite difference in this step. Let $x_1 \in \mX$ be a variable.
\begin{equation*}
\begin{aligned}
& \quad \norm{\nx^2 V_u(0,\cdot) - \nx^2 V^*(0,\cdot)}_{L^1} = \int_{\mX} \abs{\nx^2 V_u(0,x_1) - \nx^2 V^*(0,x_1)} \rd x_1 \\
& = \int_{\mX} \abs{ \int_0^{\ve_0} \pve \nx^2 V^{\ve}(0,x_1) \rd\ve } \rd x_1 = \int_{\mX} \abs{ \int_0^{\ve_0} \nx^2 \pve  V^{\ve}(0,x_1) \rd\ve} \rd x_1 \\
& \le \int_0^{\ve_0} \int_{\mX} \abs{ \nx^2 \pve  V^{\ve}(0,x_1)} \rd x_1 \, \rd\ve.
\end{aligned}
\end{equation*}
So, it is sufficient to show
\begin{equation}\label{eq:lem4step3temp1}
\int_{\mX} \abs{ \nx^2 \pve  V^{\ve}(0,x_1)} \rd x_1 \le C \ve^{\alpha}
\end{equation}
for all $\ve \in (0,\ve_0)$ in order to recover \eqref{eq:lem4step3temp0}.
To compute the Hessian, let $z$ be a perturbation vector with $\abs{z}=1$. Denote $x_0 = x_1 - \delta z$ and $x_2 = x_1 + \delta z$. Without loss of generality, we just consider $\delta>0$. Then,
$$\lim_{\delta\to 0} \dfrac{1}{\delta^2} \parentheses{\pve V^{\ve}(0,x_0) + \pve V^{\ve}(0,x_2) - 2\pve V^{\ve}(0,x_1)} = z\tp \nx^2 \pve V^{\ve}(0,x_1)z.$$
We recall that $\norm{\cdot}_2$ denotes the matrix spectrum norm. Since
$$\abs{ \nx^2 \pve  V^{\ve}(0,x_1)} \le \sqrt{n} \norm{ \nx^2 \pve  V^{\ve}(0,x_1)}_2,$$
we only need to show that there exists $C$ that does not depend on $z$, such that
\begin{equation}\label{eq:lem4temp0}
\int_{\mX} \abs{ z\tp \nx^2 \pve  V^{\ve}(0,x_1) z} \rd x_1 \le C \ve^{\alpha}
\end{equation}
for all $\abs{z}=1$, in order to get \eqref{eq:lem4step3temp1}. The left hand side of \eqref{eq:lem4temp0} satisfies
\begin{equation}\label{eq:lem4temp1}
\begin{aligned}
& \quad \int_{\mX} \abs{ z\tp \nx^2 \pve  V^{\ve}(0,x_1)z} \, \rd x_1 \\
& = \int_{\mX} \lim_{\delta\to 0} \dfrac{1}{\delta^2} \abs{\pve V^{\ve}(0,x_0) + \pve V^{\ve}(0,x_2) - 2\pve V^{\ve}(0,x_1)} \, \rd x_1 \\
& \le \liminf_{\delta\to 0} \dfrac{1}{\delta^2} \int_{\mX}  \abs{\pve V^{\ve}(0,x_0) + \pve V^{\ve}(0,x_2) - 2\pve V^{\ve}(0,x_1)} \, \rd x_1,
\end{aligned}
\end{equation}
where we used Fatou's lemma in the last inequality. And
\begin{equation}\label{eq:lem4temp2}
\begin{aligned}
& \quad \int_{\mX}  \abs{\pve V^{\ve}(0,x_0) + \pve V^{\ve}(0,x_2) - 2\pve V^{\ve}(0,x_1)} \, \rd x_1\\
& = \int_{\mX} \abs{ \EE \sqbra{\int_0^T   \inner{\nbu G^{0,\ve}_t}{\phi^{0,\ve}_t} + \inner{\nbu G^{2,\ve}_t}{\phi^{2,\ve}_t} -2 \inner{\nbu G^{1,\ve}_t}{\phi^{1,\ve}_t}  \rd t}} \, \rd x_1 \\
& \le \int_{\mX} \EE \sqbra{\int_0^T \abs{  \inner{\nbu G^{0,\ve}_t}{\phi^{0,\ve}_t} + \inner{\nbu G^{2,\ve}_t}{\phi^{2,\ve}_t} -2 \inner{\nbu G^{1,\ve}_t}{\phi^{1,\ve}_t} } \rd t} \, \rd x_1 \\
& \le \int_{\mX} \int_0^T \int_{\mX} \abs{  \inner{\nbu G(t,x,u^{\ve},-\nx V^{\ve}, -\nx^2 V^{\ve})}{\phi(t,x)}} \abs{\rho^{0,\ve} + \rho^{2,\ve} - 2\rho^{1,\ve}} \rd x \, \rd t \, \rd x_1,
\end{aligned}
\end{equation}
where we used \eqref{eq:dVdu} in Proposition \ref{prop:cost_derivative} for the first equality. Combining \eqref{eq:lem4temp1} and \eqref{eq:lem4temp2}, it is sufficient to show
\begin{equation}\label{eq:lem4step3temp2}
\int_{\mX} \EE \sqbra{\int_0^{\ve^{2\alpha}} \abs{  \inner{\nbu G^{0,\ve}_t}{\phi^{0,\ve}_t} + \inner{\nbu G^{2,\ve}_t}{\phi^{2,\ve}_t} -2 \inner{\nbu G^{1,\ve}_t}{\phi^{1,\ve}_t} } \rd t} \, \rd x_1 \le C \delta^2 \ve^{\alpha}
\end{equation}
and
\begin{equation}\label{eq:lem4step3temp3}
\int_{\ve^{2\alpha}}^T \int_{\mX} \abs{  \inner{\nbu G(t,x,u^{\ve},-\nx V^{\ve}, -\nx^2 V^{\ve})}{\phi(t,x)}} \abs{\rho^{0,\ve} + \rho^{2,\ve} - 2\rho^{1,\ve}} \rd x \, \rd t \le C \delta^2 \ve^{\alpha}
\end{equation}
in order to recover \eqref{eq:lem4temp0} and hence \eqref{eq:lem4step3temp1}. These two inequalities are tasks for later steps. Again, we only need to verify them when $\delta \le \ve$. 

\emph{Step 3.2.} We derive some generalizations of mean value theorem and Gronwall inequalities in this step.
Let $f(s): \RR \to \RR$ be a smooth function. Then
\begin{equation*}
f(1)+f(-1)-2f(0) = \int_0^1 (f'(s)-f'(-s)) \,\rd s = \int_0^1 \int_{-s}^s f''(\tau) \,\rd \tau \,\rd s.
\end{equation*}
Using the mean value theorem, we know that there exists $c \in [-1,1]$ s.t. $f(1)+f(-1)-2f(0) = f''(c)$. More generally, let $x_j \in \mX ~ (j=0,1,2)$ with $2x_1=x_0+x_2$ and let $f:s \mapsto g(x_1 + s(x_2-x_1))$ for some smooth function $g$. Then
\begin{equation}\label{eq:mean_value}
\begin{aligned}
& \quad g(x_2)+g(x_0)-2g(x_1) = f(1)+f(-1)-2f(0) = f''(c) \\
& = (x_2-x_1)\tp \nx^2 g((1-c)x_1 + cx_2) (x_2-x_1)
\end{aligned}
\end{equation}
for some $c\in [-1,1]$. We will use this result in later steps. We give some Gronwall inequalities next. 

Let $x^j_t ~ (j=0,1,2)$ be the state processes that start at  $x^j_0=x_j$ with $x_0 = x_1-\delta z$ and $x_2 = x_1 + \delta z$ and $\abs{z}=1$. Here $x^j_t$ share the same control function $u$, which could be $u^*$ or $u^{\ve}$. As two corollaries of \eqref{eq:Gronwall2}, We want to show
\begin{equation}\label{eq:Gronwall4}
\sup_{t \in [0,T]} \EE  \abs{x^1_t - x^2_t}^4 \le C \EE\abs{x^1_0-x^2_0}^4 = C\delta^4
\end{equation}
and
\begin{equation}\label{eq:Gronwall5}
\sup_{t \in [0,T]} \EE  \abs{x^0_t + x^2_t - 2x^1_t}^2 \le C\delta^4.
\end{equation}
By It\^o's formula,
\begin{align*}
& \rd \abs{x^1_t-x^2_t}^4 = \left[ 4 (x^1_t-x^2_t)\tp \parentheses{\sigma^1_t-\sigma^2_t}  \parentheses{\sigma^1_t-\sigma^2_t}\tp (x^1_t-x^2_t) + 2\abs{x^1_t-x^2_t}^2 \right. \\
& \left. \parentheses{\abs{\sigma^1_t-\sigma^2_t}^2 + 2\inner{x^1_t-x^2_t}{b^1_t-b^2_t}} \right] \rd t  + 4 \abs{x^1_t-x^2_t}^2 (x^1_t-x^2_t)\tp (\sigma^1_t-\sigma^2_t) \rd W_t,
\end{align*}
where we have inherit the notation $b^j_t = b(x^j_t,u(t,x^j_t))$ and $\sigma^j_t = \sigma(x^j_t,u(t,x^j_t))$ in the proof of Lemma \ref{lem:gronwall}. Integrating, taking expectation, and using \eqref{eq:diffb_bound} and \eqref{eq:diffsigma_bound}, we obtain
\begin{equation*}
\EE \abs{x^1_T - x^2_T}^4 \le \EE \abs{x^1_0 - x^2_0}^4 + C \EE \int_0^T \EE \abs{x^1_t - x^2_t}^4 = \delta^4 + C \EE \int_0^T \EE \abs{x^1_t - x^2_t}^4,
\end{equation*}
where we used the Lipschitz condition for $b$ and $\sigma$ in Assumption \ref{assump:basic}. This inequality also holds for $T' < T$. Therefore, applying a Gronwall's inequality gives us
$\EE \abs{x^1_T - x^2_T}^4 \le C \delta^4$. Therefore, \eqref{eq:Gronwall4} holds.

We show \eqref{eq:Gronwall5} next. By It\^o's formula,
\begin{equation}\label{eq:d_diff2x}
\begin{aligned}
\rd \abs{x^0_t + x^2_t - 2x^1_t}^2 = & \sqbra{\abs{\sigma^0_t + \sigma^2_t - 2\sigma^1_t}^2 + 2\inner{x^0_t + x^2_t - 2x^1_t}{b^0_t - b^2_t - 2b^1_t}} \rd t\\
& + 2(x^0_t + x^2_t - 2x^1_t)\tp (\sigma^0_t + \sigma^2_t - 2\sigma^1_t) \rd W_t.
\end{aligned}
\end{equation}
We pick the $i-$th entry of the vector valued function $b$ and denote the map $(t,x) \mapsto b_i(x,u(t,x))$ by $B_i(t,x)$ for $i=1,2,\ldots,n$. By Assumption \ref{assump:basic}, $B_i(t,x)$ is Lipschitz in $x$ and has bounded Hessian in $x$.
Apply the mean value theorem \eqref{eq:mean_value}, we get 
\begin{equation*}
\begin{aligned}
& \quad \abs{B_i(t, x^0_t) + B_i(t, x^2_t) - 2B_i(t,x^1_t)} \\
& \le \abs{B_i(t,2x^1_t-x^2_t) + B_i(t, x^2_t) - 2B_i(t,x^1_t)} + \abs{B_i(t, 2x^1_t-x^2_t) - B_i(t,x^0_t)} \\
& \le (x^2_t - x^1_t)\tp \nx^2 B_i(t, x^1_t + c(x^2_t-x^1_t)) \, (x^2_t - x^1_t) + C \abs{x^0_t + x^2_t - 2x^1_t} \\
& \le C \parentheses{\abs{x^2_t - x^1_t}^2 + \abs{x^0_t + x^2_t - 2x^1_t}}.
\end{aligned}
\end{equation*}
Therefore,
\begin{equation}\label{eq:diff2b_bound}
\abs{b^0_t - b^2_t - 2b^1_t} \le C \parentheses{\abs{x^2_t - x^1_t}^2 + \abs{x^0_t + x^2_t - 2x^1_t}}.
\end{equation}
Similarly, we have
\begin{equation}\label{eq:diff2sigma_bound}
\abs{\sigma^0_t - \sigma^2_t - 2\sigma^1_t} \le C \parentheses{\abs{x^2_t - x^1_t}^2 + \abs{x^0_t + x^2_t - 2x^1_t}}.
\end{equation}
Integrating \eqref{eq:d_diff2x}, taking expectation,  we obtain
\begin{equation}\label{eq:Gronwall5_temp}
\begin{aligned}
& \quad \EE\abs{x^0_T + x^2_T - 2x^1_T}^2 = \EE \abs{x^0_0 + x^2_0 - 2x^1_0}^2 \\
& \quad + \EE \int_0^T \parentheses{\abs{\sigma^0_t + \sigma^2_t - 2\sigma^1_t}^2 + 2\inner{x^0_t + x^2_t - 2x^1_t}{b^0_t - b^2_t - 2b^1_t}} \rd t \\
& \le 0 + \EE \int_0^T \parentheses{\abs{\sigma^0_t + \sigma^2_t - 2\sigma^1_t}^2 + \abs{x^0_t + x^2_t - 2x^1_t}^2 + \abs{b^0_t - b^2_t - 2b^1_t}^2} \rd t \\
& \le C \EE \int_0^T \parentheses{\abs{x^2_t - x^1_t}^4 + \abs{x^0_t + x^2_t - 2x^1_t}^2} \rd t \le C\delta^4 + C \EE \int_0^T  \abs{x^0_t + x^2_t - 2x^1_t}^2 \rd t,
\end{aligned}
\end{equation}
where we used \eqref{eq:diff2b_bound} and \eqref{eq:diff2sigma_bound} in the second inequality, and used \eqref{eq:Gronwall4} in the third. Applying Gronwall's inequality on \eqref{eq:Gronwall5_temp}, we get
$$\EE\abs{x^0_T + x^2_T - 2x^1_T}^2 \le C \delta^4$$
and hence recover \eqref{eq:Gronwall5}.

\emph{Step 3.3.} We reformulate \eqref{eq:lem4step3temp2} and estimate the first two terms in this step. Let us rewrite \eqref{eq:lem4step3temp2} first.
\begin{equation}\label{eq:step3_4terms}
\begin{aligned}
& \quad \int_{\mX} \EE \sqbra{\int_0^{\ve^{2\alpha}} \abs{  \inner{\nbu G^{0,\ve}_t}{\phi^{0,\ve}_t} + \inner{\nbu G^{2,\ve}_t}{\phi^{2,\ve}_t} -2 \inner{\nbu G^{1,\ve}_t}{\phi^{1,\ve}_t} } \rd t} \, \rd x_1 \\
& = \int_{\mX} \EE \left[\int_0^{\ve^{2\alpha}} \left| \inner{\nbu G^{0,\ve}_t + \nbu G^{2,\ve}_t -2\nbu G^{1,\ve}_t}{\phi^{1,\ve}_t} + \inner{\nbu G^{1,\ve}_t}{\phi^{0,\ve}_t + \phi^{2,\ve}_t -2\phi^{1,\ve}_t} \right.\right.\\
& \hspace{0.5in} \left.\left. \inner{\nbu G^{2,\ve}_t - \nbu G^{1,\ve}_t}{\phi^{2,\ve}_t - \phi^{1,\ve}_t} + \inner{\nbu G^{0,\ve}_t - \nbu G^{1,\ve}_t}{\phi^{0,\ve}_t - \phi^{1,\ve}_t} \right| \rd t\right] \, \rd x_1 \\
& \le: (\rom{1}) + (\rom{2}) + (\rom{3}) + (\rom{4}),
\end{aligned}
\end{equation}
where we use triangle inequality in the last step and bound \eqref{eq:step3_4terms} by four separate integrals, which are denoted by $(\rom{1}) - (\rom{4})$. Because of Assumption \ref{assump:basic} and the regularity of $V^{\ve}$, the map $(t,x) \mapsto \nbu G^{\ve}(t,x,u^{\ve}(t,x),-\nx V^{\ve},-\nx^2 V^{\ve})$ is lipschitp in $x$ and has bounded Hessian in $x$. So, we can repeat the argument for \eqref{eq:diff2b_bound} and obtain
\begin{equation}\label{eq:diff2G_bound}
\abs{\nbu G^{0,\ve}_t + \nbu G^{2,\ve}_t - 2\nbu G^{1,\ve}_t} \le C \parentheses{\abs{\xeet - \xyet}^2 + \abs{\xlet + \xeet - 2\xyet}},
\end{equation}
where we have used boundedness of $\nx^4 V^{\ve}$. Therefore, we can estimate $(\rom{1})$ through
\begin{equation}\label{eq:lem4_term1}
\begin{aligned}
& (\rom{1}) = \int_{\mX} \EE \sqbra{\int_0^{\ve^{2\alpha}} \abs{\inner{\nbu G^{0,\ve}_t + \nbu G^{2,\ve}_t -2\nbu G^{1,\ve}_t}{\phi^{1,\ve}_t} } \rd t}  \rd x_1 \\
& \le \int_{\mX} \EE \sqbra{\int_0^{\ve^{2\alpha}} \parentheses{\dfrac{1}{\delta^2 \ve^{\alpha}} \abs{\nbu G^{0,\ve}_t + \nbu G^{2,\ve}_t -2\nbu G^{1,\ve}_t}^2 + \delta^2 \ve^{\alpha} \abs{\phi^{1,\ve}_t}^2 } \rd t} \rd x_1 \\
& \le \int_{\mX} \int_0^{\ve^{2\alpha}} \EE \sqbra{\dfrac{1}{\delta^2 \ve^{\alpha}} \abs{\xeet - \xyet}^2 + \dfrac{1}{\delta^2 \ve^{\alpha}} \abs{\xlet + \xeet - 2\xyet}^2 } \rd t \, \rd x_1 + \delta^2 \ve^{\alpha} \rho_1 \norm{\phi}_{L^2}^2 \\
& \le \int_{\mX} \int_0^{\ve^{2\alpha}} \sqbra{\dfrac{1}{\delta^2 \ve^{\alpha}} C \delta^4 + \dfrac{1}{\delta^2 \ve^{\alpha}}C\delta^4 } \rd t \, \rd x_1 + \delta^2 \ve^{\alpha} \rho_1 \le C \delta^2 \ve^{\alpha}.
\end{aligned}
\end{equation}
In the second inequality above, we used \eqref{eq:diff2G_bound}. Also, for the $\phi$ term, the argument is the same as in \emph{Step 2.3}, see \eqref{eq:view_uniform}. In the third inequality above,  we used \eqref{eq:Gronwall4} and \eqref{eq:Gronwall5}. 

Let us consider $(\rom{2})$ next. Similar to \eqref{eq:diff2G_bound}, we can estimate the $\phi$ term in $(\rom{2})$, but note that the constant should scaled by $2/\ve$. (This is explained at the beginning of the proof). So, we have
\begin{equation}\label{eq:diff2phi_bound}
\abs{\phi^{0,\ve}_t + \phi^{2,\ve}_t - 2\phi^{1,\ve}_t} \le \dfrac{C}{\ve} \parentheses{\abs{\xeet - \xyet}^2 + \abs{\xlet + \xeet - 2\xyet}}.
\end{equation}
Therefore
\begin{equation}\label{eq:lem4_term2}
\begin{aligned}
& (\rom{2}) = \int_{\mX} \EE \sqbra{\int_0^{\ve^{2\alpha}} \abs{\inner{\nbu G^{1,\ve}_t}{\phi^{0,\ve}_t + \phi^{2,\ve}_t -2\phi^{1,\ve}_t} } \rd t}  \rd x_1 \\
& \le \int_{\mX} \int_0^{\ve^{2\alpha}} \EE \sqbra{ \parentheses{ \dfrac{\delta^2}{\ve^{2-\alpha}} \abs{\nbu G^{1,\ve}_t}^2 + \dfrac{\ve^{2-\alpha}}{\delta^2}\abs{\phi^{0,\ve}_t + \phi^{2,\ve}_t -2\phi^{1,\ve}_t}^2 } } \rd t \, \rd x_1 \\
& \le C \int_{\mX} \int_0^{\ve^{2\alpha}} \EE \left[ \dfrac{\delta^2}{\ve^{2-\alpha}} \left( \abs{\nx V^{\ve}(t,\xyet) - \nx V^*(t,\xyet)}^2  + \Big|\nx^2 V^{\ve}(t,\xyet) - \right. \right.\\
& \left. \left. \nx^2 V^*(t,\xyet) \Big|^2 + \ve^2 \abs{\phi^{1,\ve}_t}^2 \right) + \dfrac{1}{\delta^2 \ve^{\alpha}} \parentheses{\abs{\xeet - \xyet}^4 + \abs{\xlet + \xeet - 2\xyet}^2} \right] \rd t \, \rd x_1 \\
& \le C \dfrac{\delta^2}{\ve^{2-\alpha}} \parentheses{\rho_1 \normHtwo{V^{\ve}-V^*}^2 + \ve^2 \rho_1 \norm{\phi}_{L^2}^2} + C \int_{\mX} \int_0^{\ve^{2\alpha}} \dfrac{1}{\delta^2 \ve^{\alpha}} C \delta^4 \rd t \, \rd x_1 \le C \delta^2 \ve^{\alpha},
\end{aligned}
\end{equation}
where we have consecutively used: Cauchy's inequality; the estimate of $\nbu G$ in \eqref{eq:bound_nuGve} and $\phi$ terms in \eqref{eq:diff2phi_bound}; the argument at the end of \emph{Step 2.3} for $\phi$ in \eqref{eq:view_uniform} and the Gronwall inequalities \eqref{eq:Gronwall4}, \eqref{eq:Gronwall5}; Lemma \ref{lem:regularity_Vu}. 

\emph{Step 3.4.} We reformulate $(\rom{3})$ in \eqref{eq:step3_4terms} in this step. $(\rom{4})$ can be analyzed in the same way. We want establish the following estimate in the following few steps.
\begin{equation}\label{eq:step3_term3}
(\rom{3}) = \int_{\mX} \EE \sqbra{\int_0^{\ve^{2\alpha}} \abs{\inner{\nbu G^{2,\ve}_t - \nbu G^{1,\ve}_t}{\phi^{2,\ve}_t - \phi^{1,\ve}_t} } \rd t}  \rd x_1 \le C \delta^2 \ve^{\alpha}.
\end{equation}
The idea is similar to the derivation of \eqref{eq:bound_nuGve}. Denote $u^{j,*}_t = u^*(t,x^{j,\ve}_t)$ for $j=0,1,2$. Denote $f^{j,\ve}_t = f(x^{j,\ve}_t, u^{\ve}(t, x^{j,\ve}_t))$ and $f^{j,*}_t = f(x^{j,\ve}_t, u^{*}(t, x^{j,\ve}_t))$ for $f=r, b, \sigma, D$ and $j=0,1,2$. Denote $\nbu G^{j,*}_t = \nbu G(t, x^{j,\ve}_t, u^{j,*}_t, -\nx V^*, -\nx^2 V^*)$ for $j=0,1,2$. Note that $\nbu G^{j,*}_t=0$ due to maximum condition \eqref{eq:max1}. Denote $V^{j,\ve}_t = V^{\ve}(t,x^{j,\ve}_t)$ and $V^{j,*}_t = V^{*}(t,x^{j,\ve}_t)$. By definition of $G$ \eqref{eq:generalized_Hamiltonian},
\begin{equation*}
\begin{aligned}
& \quad \nbu G^{2,\ve}_t - \nbu G^{1,\ve}_t = \parentheses{\nbu G^{2,\ve}_t - \nbu G^{2,*}_t} - \parentheses{\nbu G^{1,\ve}_t - \nbu G^{1,*}_t} \\
& = \parentheses{\nbu r^{1,\ve}_t - \nbu r^{1,*}_t} - \parentheses{\nbu r^{2,\ve}_t - \nbu r^{2,*}_t} \\
& \quad + \parentheses{\nbu b^{1,\ve\top}_t \nx V^{1,\ve}_t - \nbu b^{1,*\top}_t \nx V^{1,*}_t} - \parentheses{\nbu b^{2,\ve\top}_t \nx V^{2,\ve}_t - \nbu b^{2,*\top}_t \nx V^{2,*}_t} \\
& \quad + \nbu \Tr\parentheses{D^{1,\ve}_t \nx^2 V^{1,\ve}_t - D^{1,*}_t \nx^2 V^{1,*}_t} - \nbu \Tr\parentheses{D^{2,\ve}_t \nx^2 V^{2,\ve}_t - D^{2,*}_t \nx^2 V^{2,*}_t}
\end{aligned}
\end{equation*}
Note that the $\nbu$ only operate on $D$ in the last two terms. Therefore,
\begin{equation}\label{eq:step3_term3decomp}
\int_{\mX} \EE \sqbra{\int_0^{\ve^{2\alpha}} \abs{\inner{\nbu G^{2,\ve}_t - \nbu G^{1,\ve}_t}{\phi^{2,\ve}_t - \phi^{1,\ve}_t} } \rd t}  \rd x_1 \le (\rom{5}) + (\rom{6}) + (\rom{7})
\end{equation}
where
\begin{equation*}
(\rom{5}) := \int_{\mX} \EE \sqbra{\int_0^{\ve^{2\alpha}} \abs{\inner{\parentheses{\nbu r^{1,\ve}_t - \nbu r^{1,*}_t} - \parentheses{\nbu r^{2,\ve}_t - \nbu r^{2,*}_t}}{\phi^{2,\ve}_t - \phi^{1,\ve}_t} } \rd t}  \rd x_1,
\end{equation*}
\begin{equation*}
\begin{aligned}
(\rom{6}) &:=  \int_{\mX} \EE \left[ \int_0^{\ve^{2\alpha}} \left| \left<\parentheses{\nbu b^{1,\ve\top}_t \nx V^{1,\ve}_t - \nbu b^{1,*\top}_t \nx V^{1,*}_t} \right. \right. \right. \\
& \hspace{1in} \left. \left. \left. - \parentheses{\nbu b^{2,\ve\top}_t \nx V^{2,\ve}_t - \nbu b^{2,*\top}_t \nx V^{2,*}_t} , \phi^{2,\ve}_t - \phi^{1,\ve}_t \right> \right| \rd t\right]  \rd x_1,
\end{aligned}
\end{equation*}
and
\begin{equation*}
\begin{aligned}
(\rom{7}) &:= \int_{\mX} \EE \left[ \int_0^{\ve^{2\alpha}} \left| \left<\nbu \Tr\parentheses{D^{1,\ve}_t \nx^2 V^{1,\ve}_t - D^{1,*}_t \nx^2 V^{1,*}_t} \right. \right. \right. \\
& \hspace{1in} \left. \left. \left. - \nbu \Tr\parentheses{D^{2,\ve}_t \nx^2 V^{2,\ve}_t - D^{2,*}_t \nx^2 V^{2,*}_t} , \phi^{2,\ve}_t - \phi^{1,\ve}_t \right> \right| \rd t\right]  \rd x_1.
\end{aligned}
\end{equation*}
We want to show that each term above is less than $C \delta^2 \ve^{\alpha}$ in order to recover \eqref{eq:step3_term3}.

\emph{Step 3.5.} In this step, we want to show
\begin{equation}\label{eq:lem4_phi2}
\ve \int_{\mX} \EE \sqbra{\int_0^{\ve^{2\alpha}} \abs{ \phi^{2,\ve}_t - \phi^{1,\ve}_t }^2 \rd t} \rd x_1 \le C \delta^2 \ve^{\alpha}.
\end{equation}
We first give a generalization of the argument for \eqref{eq:view_uniform}. Like before, we define three new processes $\overline{x}^{j,\ve}_t$ for $j=0,1,2$ that start at $\overline{x}^{j,\ve}_0 \sim \text{Unif}(\mX)$ with $\overline{x}^{2,\ve}_0 - \overline{x}^{1,\ve}_0 \equiv \overline{x}^{1,\ve}_0 - \overline{x}^{2,\ve}_0 \equiv \delta z$. $\overline{x}^{j,\ve}_t$ follow the same dynamic as $x^{j,\ve}_t$. Again, denote $\overline{\rho}^{j,\ve}(t,x)$ the density for $\overline{x}^{j,\ve}_t$. Then, $\overline{\rho}^{0,\ve}(t,x) = \overline{\rho}^{1,\ve}(t,x) = \overline{\rho}^{2,\ve}(t,x)$ because they share the same distribution. So,
\begin{equation}\label{eq:lem4_step3_view1}
\begin{aligned}
& \quad \int_{\mX} \EE \sqbra{\int_0^{\ve^{2\alpha}}  \inner{\phi^{1,\ve}_t}{\phi^{1,\ve}_t} \rd t} \rd x_1 \\
& \equiv \int_{\mX} \EE \sqbra{\int_0^{\ve^{2\alpha}}  \inner{\phi(t,\xyet)}{\phi(t,\xyet)} \rd t ~\Big|~ x^{1,\ve}_0=x_1} \rd x_1 \\
& = \EE_{\overline{x}^{1,\ve}_0 \sim \text{Unif}(\mX)} \EE \sqbra{\int_0^{\ve^{2\alpha}}  \inner{\phi(t,\overline{x}^{1,\ve}_t)}{\phi(t,\overline{x}^{1,\ve}_t)} \rd t ~\Big|~ \overline{x}^{1,\ve}_0} \\
& = \EE \sqbra{\int_0^{\ve^{2\alpha}}  \inner{\phi(t,\overline{x}^{1,\ve}_t)}{\phi(t,\overline{x}^{1,\ve}_t)} \rd t} = \int_{\mX} \int_0^{\ve^{2\alpha}} \inner{\phi(t,x)}{\phi(t,x)} \overline{\rho}^{1,\ve}(t,x) \rd t \, \rd x \\
& = \int_{\mX} \int_0^{\ve^{2\alpha}} \inner{\phi(t,x)}{\phi(t,x)} \overline{\rho}^{2,\ve}(t,x) \rd t \, \rd x = \int_{\mX} \EE \sqbra{\int_0^{\ve^{2\alpha}}  \inner{\phi^{2,\ve}_t}{\phi^{2,\ve}_t} \rd t} \rd x_1.
\end{aligned}
\end{equation}
Similarly, since $(\overline{x}^{0,\ve}_t, \overline{x}^{1,\ve}_t)$ and $(\overline{x}^{1,\ve}_t, \overline{x}^{2,\ve}_t)$ share the same joint distribution, we can show
\begin{equation}\label{eq:lem4_step3_view2}
\int_{\mX} \EE \sqbra{\int_0^{\ve^{2\alpha}}  \inner{\phi^{1,\ve}_t}{\phi^{0,\ve}_t} \rd t} \rd x_1 = \int_{\mX} \EE \sqbra{\int_0^{\ve^{2\alpha}}  \inner{\phi^{2,\ve}_t}{\phi^{1,\ve}_t} \rd t} \rd x_1.
\end{equation}
Therefore,
\begin{equation*}
\begin{aligned}
& \quad \ve \int_{\mX} \EE \sqbra{\int_0^{\ve^{2\alpha}} \abs{ \phi^{2,\ve}_t - \phi^{1,\ve}_t }^2 \rd t} \rd x_1 \\
& = \ve \int_{\mX} \EE \sqbra{\int_0^{\ve^{2\alpha}} \parentheses{\inner{\phi^{2,\ve}_t}{\phi^{2,\ve}_t} + \inner{\phi^{1,\ve}_t}{\phi^{1,\ve}_t} - 2\inner{\phi^{2,\ve}_t}{\phi^{1,\ve}_t}} \rd t} \rd x_1 \\
& = \ve \int_{\mX} \EE \sqbra{\int_0^{\ve^{2\alpha}} \parentheses{2\inner{\phi^{1,\ve}_t}{\phi^{1,\ve}_t} - \inner{\phi^{2,\ve}_t}{\phi^{1,\ve}_t} - \inner{\phi^{1,\ve}_t}{\phi^{0,\ve}_t}} \rd t} \rd x_1 \\
& = \int_{\mX} \EE \sqbra{\int_0^{\ve^{2\alpha}} \inner{\phi^{1,\ve}_t}{ \ve \parentheses{2\phi^{1,\ve}_t - \phi^{2,\ve}_t  - \phi^{0,\ve}_t}} \rd t} \rd x_1 \\
& \le C \int_{\mX} \EE \sqbra{\int_0^{\ve^{2\alpha}} \sqbra{\delta^2 \ve^{\alpha} \abs{\phi^{1,\ve}_t}^2 + \dfrac{1}{\delta^2 \ve^{\alpha}} \parentheses{\abs{2\xyet-\xeet-\xlet}^2 + \abs{\xyet-\xeet}^4}} \rd t} \rd x_1 \\
& \le C \parentheses{\delta^2 \ve^{\alpha} \rho_1 \norm{\phi}_{L^2}^2 + \delta^{-2} \ve^{2\alpha-\alpha} \delta^4} \le C \delta^2 \ve^{\alpha}.
\end{aligned}
\end{equation*}
We used \eqref{eq:lem4_step3_view1} and \eqref{eq:lem4_step3_view2} in the second equality above. In the first inequality, we used estimate for $\phi$ in \eqref{eq:diff2phi_bound} and Cauchy's inequality. The second inequality is because of \eqref{eq:view_uniform} and the Gronwall inequalities \eqref{eq:Gronwall4} \eqref{eq:Gronwall5}. Therefore, \eqref{eq:lem4_phi2} holds.

\emph{Step 3.6.} We estimate $(\rom{5})$ in this step. Let us pick one dimension $i \le m$ and use mean value theorem. We get
\begin{align*}
& \quad \partial_{u_i} r^{1,\ve}_t - \partial_{u_i} r^{2,\ve}_t \equiv \partial_{u_i} r(\xyet, u^{1,\ve}_t) - \partial_{u_i} r(\xeet, u^{2,\ve}_t)\\
& = \inner{\nabla_{x,u} \partial_{u_i} r\parentheses{ (1-c)\xyet +c\xeet, (1-c)u^{1,\ve}_t + c u^{2,\ve}_t }}{(\xyet-\xeet, u^{1,\ve}_t - u^{2,\ve}_t)}
\end{align*}
for some $c \in [0,1]$. Therefore,
\begin{equation*}
\begin{aligned}
& \quad \abs{ \parentheses{\partial_{u_i} r^{1,\ve}_t - \partial_{u_i} r^{2,\ve}_t} - \parentheses{\partial_{u_i} r^{1,*}_t - \partial_{u_i} r^{2,*}_t} } \\
& = \left| \inner{\nabla_{x,u} \partial_{u_i} r\parentheses{ (1-c)\xyet +c\xeet, (1-c)u^{1,\ve}_t + c u^{2,\ve}_t }}{(\xyet-\xeet, u^{1,\ve}_t - u^{2,\ve}_t)} \right. \\
& \quad - \left. \inner{\nabla_{x,u} \partial_{u_i} r\parentheses{ (1-c')x^{1,\ve}_t +c'x^{2,\ve}_t, (1-c')u^{1,*}_t + c' u^{2,*}_t }}{(x^{1,\ve}_t-x^{2,\ve}_t, u^{1,*}_t - u^{2,*}_t)} \right| \\
& \le L \parentheses{ \abs{\xyet-\xeet} + \abs{u^{1,\ve}_t - u^{2,\ve}_t} + \abs{u^{1,\ve}_t - u^{1,*}_t} + \abs{u^{2,\ve}_t - u^{2,*}_t} } \cdot \\
& \quad \parentheses{\abs{\xyet-\xeet} + \abs{u^{1,\ve}_t - u^{2,\ve}_t}} + K \abs{\parentheses{u^{1,\ve}_t - u^{2,\ve}_t} -  \parentheses{u^{1,*}_t - u^{2,*}_t}} \\
& \le L \sqbra{(L+1)\abs{\xyet-\xeet} + \ve \parentheses{\abs{\phi^{1,\ve}_t} + \abs{\phi^{2,\ve}_t}}} (L+1) \abs{\xyet-\xeet} + K\ve \abs{\phi^{1,\ve}_t - \phi^{2,\ve}_t} \\
& \le C \parentheses{\abs{\xyet-\xeet}^2 + \ve \abs{\xyet-\xeet} \parentheses{\abs{\phi^{1,\ve}_t} + \abs{\phi^{2,\ve}_t}} + \ve \abs{\phi^{1,\ve}_t - \phi^{2,\ve}_t} }.
\end{aligned}
\end{equation*}
In the first equality, we apply the mean value theorem twice, with $c,c' \in [0,1]$. In the first inequality, we used: 
$$\abs{\inner{a_1}{b_1} - \inner{a_2}{b_2}} \le \abs{a_1-a_2} \, \abs{b_1} + \abs{a_2} \, \abs{b_1-b_2},$$
 and the Lipschitz and boundedness property (in Assumption \ref{assump:basic}) of the derivatives for $r$. In the second inequality, we use $u^{\ve} = u^* + \ve \phi$ and the Lipschitz property of $u^{\ve}$. Applying the same argument in all the dimensions, we get
\begin{equation*}
\begin{aligned}
& \quad \abs{\parentheses{\nbu r^{1,\ve}_t - \nbu r^{1,*}_t} - \parentheses{\nbu r^{2,\ve}_t - \nbu r^{2,*}_t}} \\
& \le C \parentheses{\abs{\xyet-\xeet}^2 + \ve \abs{\xyet-\xeet} \parentheses{\abs{\phi^{1,\ve}_t} + \abs{\phi^{2,\ve}_t}} + \ve \abs{\phi^{1,\ve}_t - \phi^{2,\ve}_t} }.
\end{aligned}
\end{equation*}
Therefore
\begin{equation*}
\begin{aligned}
& (\rom{5}) = \int_{\mX} \EE \sqbra{\int_0^{\ve^{2\alpha}} \abs{\inner{\parentheses{\nbu r^{1,\ve}_t - \nbu r^{1,*}_t} - \parentheses{\nbu r^{2,\ve}_t - \nbu r^{2,*}_t}}{\phi^{2,\ve}_t - \phi^{1,\ve}_t} } \rd t}  \rd x_1 \\
& \le C \int_{\mX} \EE \left[ \int_0^{\ve^{2\alpha}} \left( \abs{\xyet-\xeet}^2 + \ve \abs{\xyet-\xeet} \parentheses{\abs{\phi^{1,\ve}_t} + \abs{\phi^{2,\ve}_t}} \right. \right.\\
& \hspace{1in} \left. \left. + \ve \abs{\phi^{1,\ve}_t - \phi^{2,\ve}_t} \right) \abs{ \phi^{2,\ve}_t - \phi^{1,\ve}_t } \rd t \right]  \rd x_1.
\end{aligned}
\end{equation*}
Continuing the analysis, we get
\begin{equation}\label{eq:step3_term5_bound}
\begin{aligned}
& (\rom{5}) \le C \int_{\mX} \EE \left[ \int_0^{\ve^{2\alpha}} \left( \dfrac{1}{\ve} \abs{\xyet-\xeet}^3 + \delta^2 \ve^{\alpha} \parentheses{\abs{\phi^{1,\ve}_t}^2 + \abs{\phi^{2,\ve}_t}^2} \right. \right. \\
& \hspace{1in} \left. \left. + \dfrac{1}{\delta^2 \ve^{\alpha}}\abs{\xyet-\xeet}^4 + \ve  \abs{ \phi^{2,\ve}_t - \phi^{1,\ve}_t}^2
 \right)  \rd t \right]  \rd x_1\\
& \le C \parentheses{ C\ve^{2\alpha-1} \delta^3 + 2\delta^2 \ve^{\alpha} \rho_1 \norm{\phi}_{L^2}^2 + C \ve^{2\alpha-\alpha} \delta^{-2+4} + \delta^2 \ve^{\alpha}} \le C \delta^2 \ve^{\alpha}.
\end{aligned}
\end{equation}
In the first inequality, we used the Lipschitz condition
$$\abs{\pyet-\peet} \le \dfrac{2L}{\ve} \abs{\xyet - \xeet}$$
and Cauchy's inequality. In the second inequality, we used Gronwall inequality of order 3, \eqref{eq:view_uniform}, Gronwall inequality of order 4 \eqref{eq:Gronwall4}, and \eqref{eq:lem4_phi2} in \emph{Step 3.5}. We did not prove Gronwall inequality of order 3, but it can be obtained directly from the order 2 \eqref{eq:Gronwall2} and order 4 \eqref{eq:Gronwall4} inequalities using Cauchy's inequality. Therefore, we finished estimation of $(\rom{5})$.

\emph{Step 3.7.} We estimate $(\rom{6})$ and $(\rom{7})$ in this step. Recall the definition
\begin{equation*}
\begin{aligned}
(\rom{6}) &=  \int_{\mX} \EE \left[ \int_0^{\ve^{2\alpha}} \left| \left<\parentheses{\nbu b^{1,\ve\top}_t \nx V^{1,\ve}_t - \nbu b^{1,*\top}_t \nx V^{1,*}_t} \right. \right. \right. \\
& \hspace{1in} \left. \left. \left. - \parentheses{\nbu b^{2,\ve\top}_t \nx V^{2,\ve}_t - \nbu b^{2,*\top}_t \nx V^{2,*}_t} , \phi^{2,\ve}_t - \phi^{1,\ve}_t \right> \right| \rd t\right]  \rd x_1.
\end{aligned}
\end{equation*}
We further decompose $(\rom{6})$ into two parts. A simple triangle inequality gives us
$$(\rom{6}) \le (\rom{8}) + (\rom{9}),$$
where
\begin{equation}\label{eq:step3_term8}
\begin{aligned}
(\rom{8}) &:=  \int_{\mX} \EE \left[ \int_0^{\ve^{2\alpha}} \left| \left<\parentheses{\nbu b^{1,\ve\top}_t \nx V^{1,\ve}_t - \nbu b^{1,*\top}_t \nx V^{1,\ve}_t} \right. \right. \right. \\
& \hspace{1in} \left. \left. \left. - \parentheses{\nbu b^{2,\ve\top}_t \nx V^{2,\ve}_t - \nbu b^{2,*\top}_t \nx V^{2,\ve}_t} , \phi^{2,\ve}_t - \phi^{1,\ve}_t \right> \right| \rd t\right]  \rd x_1.
\end{aligned}
\end{equation}
and
\begin{equation}\label{eq:step3_term9}
\begin{aligned}
(\rom{9}) &:=  \int_{\mX} \EE \left[ \int_0^{\ve^{2\alpha}} \left| \left<\parentheses{\nbu b^{1,*\top}_t \nx V^{1,\ve}_t - \nbu b^{1,*\top}_t \nx V^{1,*}_t} \right. \right. \right. \\
& \hspace{1in} \left. \left. \left. - \parentheses{\nbu b^{2,*\top}_t \nx V^{2,\ve}_t - \nbu b^{2,*\top}_t \nx V^{2,*}_t} , \phi^{2,\ve}_t - \phi^{1,\ve}_t \right> \right| \rd t\right]  \rd x_1.
\end{aligned}
\end{equation}
The analysis for $(\rom{8})$ is exactly the same as the analysis for $(\rom{5})$ in \emph{Step 3.6}, except that the function $\nbu r$ is replaced by $\nbu b\tp \nx V^{\ve}$. By Assumption \ref{assump:basic}, and the regularity for the value functions, $\nbu b\tp \nx V^{\ve}$ and $\nbu r$ share the same properties that are necessary to prove \eqref{eq:step3_term5_bound}. Therefore, we can show
\begin{equation}\label{eq:step3_term8_bound}
(\rom{8}) \le C \delta^2 \ve^{\alpha}.
\end{equation}

We consider $(\rom{9})$ next. We want to show
\begin{equation}\label{eq:step3_term9_bound}
\begin{aligned}
& (\rom{9}) =  \int_{\mX} \EE \left[ \int_0^{\ve^{2\alpha}} \left| \left< \nbu b^{1,*\top}_t \parentheses{\nx V^{1,\ve}_t - \nx V^{1,*}_t} \right. \right. \right. \\
& \hspace{1in} \left. \left. \left. - \nbu b^{2,*\top}_t \parentheses{\nx V^{2,\ve}_t - \nx V^{2,*}_t} , \phi^{2,\ve}_t - \phi^{1,\ve}_t \right> \right| \rd t\right]  \rd x_1 \le C \delta^2 \ve^{\alpha}.
\end{aligned}
\end{equation}
Again, we consider one single dimension. We pick $i\le m$, $l \le n$ and denote $b^{j,*,l}_t$ the $l$-th entry of $b^{j,*}_t$ for $j=0,1,2$. We have
\begin{equation}\label{lem4_bV4terms}
\begin{aligned}
& \quad \abs{\partial_{u_i} b^{1,*,l}_t \parentheses{\partial_{x_l} V^{1,\ve}_t - \partial_{x_l} V^{1,*}_t} - \partial_{u_i} b^{2,*,l}_t \parentheses{\partial_{x_l} V^{2,\ve}_t - \partial_{x_l} V^{2,*}_t} } \\
& \le \abs{ \parentheses{\partial_{u_i} b^{1,*,l}_t - \partial_{u_i} b^{2,*,l}_t} \parentheses{\partial_{x_l} V^{1,\ve}_t - \partial_{x_l} V^{1,*}_t}} \\
& \quad + \abs{\partial_{u_i} b^{2,*,l}_t \sqbra{\parentheses{\partial_{x_l} V^{1,\ve}_t - \partial_{x_l} V^{1,*}_t} - \parentheses{\partial_{x_l} V^{2,\ve}_t - \partial_{x_l} V^{2,*}_t}}}.
\end{aligned}
\end{equation}
We estimate the two terms in \eqref{lem4_bV4terms} next. For the first, we have
\begin{equation}\label{lem4_bV4terms1}
\begin{aligned}
\abs{ \parentheses{\partial_{u_i} b^{1,*,l}_t - \partial_{u_i} b^{2,*,l}_t} \parentheses{\partial_{x_l} V^{1,\ve}_t - \partial_{x_l} V^{1,*}_t}} \le L \abs{\xyet - \xeet} \abs{\nx V^{1,\ve}_t - \nx V^{1,*}_t}
\end{aligned}
\end{equation}
because of Assumption \ref{assump:basic}. For the second, we use the technique in \emph{Step 3.6}. We have
\begin{equation}\label{lem4_bV4terms2}
\begin{aligned}
& \quad \abs{ \parentheses{\partial_{x_l} V^{1,\ve}_t - \partial_{x_l} V^{2,\ve}_t} - \parentheses{\partial_{x_l} V^{1,*}_t - \partial_{x_l} V^{2,*}_t}} \\
& = \left| \inner{\nx \partial_{x_l} V^{\ve}\parentheses{t, (1-c)\xyet + c\xeet}}{\xyet - \xeet} \right.\\
& \quad - \left. \inner{\nx \partial_{x_l} V^*\parentheses{t, (1-c')\xyet + c'\xeet}}{\xyet - \xeet} \right| \\
& \le \abs{\nx \partial_{x_l} V^{\ve}\parentheses{t, (1-c)\xyet + c\xeet} - \nx \partial_{x_l} V^*\parentheses{t, (1-c')\xyet + c'\xeet}} \abs{\xyet - \xeet} \\
& \le \parentheses{\abs{ \nx \partial_{x_l} V^{1,\ve}_t - \nx \partial_{x_l} V^{1,*}_t } + L(c+c') \abs{\xyet - \xeet} } \abs{\xyet - \xeet},
\end{aligned}
\end{equation}
where we have consecutively used (two) mean value theorems, Cauchy's inequality, and the Lipschitz property for the derivatives of $V^{\ve}$ and $V^*$. Combining \eqref{lem4_bV4terms1}, \eqref{lem4_bV4terms2}, and $\abs{\nbu b} \le K$ into \eqref{lem4_bV4terms}, and repeat the same argument in all the dimensions, we obtain
\begin{equation}\label{eq:lem4_bV4terms_bound}
\begin{aligned}
& \quad \abs{\nbu b^{1,*\top}_t \parentheses{\nx V^{1,\ve}_t - \nx V^{1,*}_t} - \nbu b^{2,*\top}_t \parentheses{\nx V^{2,\ve}_t - \nx V^{2,*}_t} } \\
& \le C \abs{\xyet - \xeet} \parentheses{ \abs{\nx V^{1,\ve}_t - \nx V^{1,*}_t} + \abs{\nx^2 V^{1,\ve}_t - \nx^2 V^{1,*}_t} + \abs{\xyet - \xeet}}.
\end{aligned}
\end{equation}
Therefore,
\begin{equation*}
\begin{aligned}
& (\rom{9}) \le \int_{\mX} \EE \left[ \int_0^{\ve^{2\alpha}} \abs{\nbu b^{1,*\top}_t \parentheses{\nx V^{1,\ve}_t - \nx V^{1,*}_t} - \nbu b^{2,*\top}_t \parentheses{\nx V^{2,\ve}_t - \nx V^{2,*}_t} } \right. \\
& \hspace{1.3in} \left. \abs{\phi^{2,\ve}_t - \phi^{1,\ve}_t } \rd t\right]  \rd x_1 \\
& \le C \int_{\mX} \EE \left[ \int_0^{\ve^{2\alpha}}  \abs{\xyet - \xeet} \left( \abs{\nx V^{1,\ve}_t - \nx V^{1,*}_t} + \abs{\nx^2 V^{1,\ve}_t - \nx^2 V^{1,*}_t} \right. \right. \\
& \hspace{1in} \left. \left. + \abs{\xyet - \xeet}\right) \dfrac{2L}{\ve} \abs{\xyet - \xeet} \rd t\right]  \rd x_1, \\
\end{aligned}
\end{equation*}
where we used \eqref{eq:lem4_bV4terms_bound} and the Lipschitz condition for $\phi$
$$\abs{\pyet-\peet} \le \dfrac{2L}{\ve} \abs{\xyet - \xeet}.$$ 
Continuing the analysis, we get
\begin{equation*}
\begin{aligned}
& (\rom{9}) \le C \int_{\mX} \EE \left[ \int_0^{\ve^{2\alpha}}  \abs{\xyet - \xeet}^2 \dfrac{1}{\ve} \left( \abs{\nx V^{1,\ve}_t - \nx V^{1,*}_t} + \abs{\nx^2 V^{1,\ve}_t - \nx^2 V^{1,*}_t} \right) \right. \\
& \hspace{1in}  \left. + \dfrac{1}{\ve} \abs{\xyet - \xeet}^3 \rd t\right]  \rd x_1, \\
& \le C \int_{\mX} \int_0^{\ve^{2\alpha}} \EE \left[ \dfrac{1}{\delta^2 \ve^{\alpha}} \abs{\xyet - \xeet}^4 + \dfrac{\delta^2\ve^{\alpha}}{\ve^2} \abs{\nx V^{1,\ve}_t - \nx V^{1,*}_t}^2   \right. \\
& \hspace{1in}  \left. + \dfrac{\delta^2\ve^{\alpha}}{\ve^2} \abs{\nx^2 V^{1,\ve}_t - \nx^2 V^{1,*}_t}^2 + \dfrac{1}{\ve} \abs{\xyet - \xeet}^3 \right] \rd t  \rd x_1,
\end{aligned}
\end{equation*}
where we used Cauchy's inequality in the second inequality. Next, using Gronwall inequalities and the analysis for the $\nx V$ and $\nx^2 V$ terms in \emph{Step 2.3}, we have
\begin{equation*}
\begin{aligned}
& (\rom{9}) \le  C \int_{\mX} \int_0^{\ve^{2\alpha}} \parentheses{ \dfrac{\delta^4}{\delta^2 \ve^{\alpha}} + \dfrac{\delta^3}{\ve} }  \rd t \, \rd x_1 + C \dfrac{\delta^2\ve^{\alpha}}{\ve^2} \normHtwo{V^{\ve} - V^*}^2 \le C \delta^2 \ve^{\alpha}.
\end{aligned}
\end{equation*}
Note that we used $\delta \le \ve$ and Lemma \ref{lem:regularity_Vu} in the last inequality. Therefore, \eqref{eq:step3_term9_bound} holds. Combining \eqref{eq:step3_term8_bound} and \eqref{eq:step3_term9_bound}, we obtain
\begin{equation}\label{eq:step3_term6_bound}
\begin{aligned}
(\rom{6}) &=  \int_{\mX} \EE \left[ \int_0^{\ve^{2\alpha}} \left| \left<\parentheses{\nbu b^{1,\ve\top}_t \nx V^{1,\ve}_t - \nbu b^{1,*\top}_t \nx V^{1,*}_t} \right. \right. \right. \\
& \hspace{0.3in} \left. \left. \left. - \parentheses{\nbu b^{2,\ve\top}_t \nx V^{2,\ve}_t - \nbu b^{2,*\top}_t \nx V^{2,*}_t} , \phi^{2,\ve}_t - \phi^{1,\ve}_t \right> \right| \rd t\right]  \rd x_1 \le C \delta^2 \ve^{\alpha}.
\end{aligned}
\end{equation}
The analysis for $(\rom{7})$ is the same as the analysis for $(\rom{6})$. We decompose $(\rom{7})$ into two parts like \eqref{eq:step3_term8} and \eqref{eq:step3_term9}. The first one can be analyzed using the same technique in \emph{Step 3.6}. The second one be analyzed using the same technique in this step: considering each dimension separately and combine them. Therefore,
\begin{equation}\label{eq:step3_term7_bound}
\begin{aligned}
(\rom{7}) &= \int_{\mX} \EE \left[ \int_0^{\ve^{2\alpha}} \left| \left<\nbu \Tr\parentheses{D^{1,\ve}_t \nx^2 V^{1,\ve}_t - D^{1,\ve}_t \nx^2 V^{1,\ve}_t} \right. \right. \right. \\
& \hspace{0.5in} \left. \left. \left. - \nbu \Tr\parentheses{D^{2,\ve}_t \nx^2 V^{2,\ve}_t - D^{2,\ve}_t \nx^2 V^{2,\ve}_t} , \phi^{2,\ve}_t - \phi^{1,\ve}_t \right> \right| \rd t\right]  \rd x_1 \le C \delta^2 \ve^{\alpha}.
\end{aligned}
\end{equation}
Combining \eqref{eq:step3_term5_bound}, \eqref{eq:step3_term6_bound}, and \eqref{eq:step3_term7_bound} into \eqref{eq:step3_term3decomp}, we obtain \eqref{eq:step3_term3}. i.e., we get the bound for $(\rom{3})$. $(\rom{4})$ can be analyzed in exactly the same way, so we also have
\begin{equation}\label{eq:step3_term4}
(\rom{4}) = \int_{\mX} \EE \sqbra{\int_0^{\ve^{2\alpha}} \abs{\inner{\nbu G^{0,\ve}_t - \nbu G^{1,\ve}_t}{\phi^{0,\ve}_t - \phi^{1,\ve}_t} } \rd t}  \rd x_1 \le C \delta^2 \ve^{\alpha}.
\end{equation}
Combining \eqref{eq:lem4_term1}, \eqref{eq:lem4_term2}, \eqref{eq:step3_term3}, and \eqref{eq:step3_term4} into \eqref{eq:step3_4terms}, we obtain \eqref{eq:lem4step3temp2}.

\emph{Step 3.8.} We want to show \eqref{eq:lem4step3temp3}
\begin{equation*}
\int_{\ve^{2\alpha}}^T \int_{\mX} \abs{  \inner{\nbu G(t,x,u^{\ve},-\nx V^{\ve}, -\nx^2 V^{\ve})}{\phi(t,x)}} \abs{\rho^{0,\ve} + \rho^{2,\ve} - 2\rho^{1,\ve}} \rd x \, \rd t \le C \delta^2 \ve^{\alpha}
\end{equation*}
in this step. The spirit and notation are much the same as \emph{Step 2.4}. $\rho^{j,\ve}(t,x) = p^{\ve}(t,x;0,x_j)$ for $j=0,1,2$, where $p^{\ve}(t,x;s,y)$ is the fundamental solution of the Fokker--Planck equation $\partial_t \rho = \mG_{\ve}^{\dagger} \rho$. $q^{\ve}(t,x;s,y) := p^{\ve}(s,y;t,x)$ is the fundamental solution of the backward Kolmogorov equation $\partial_t \psi + \mG_{\ve} \psi = 0$. Applying the generalized mean value theorem \eqref{eq:mean_value} and lemma \eqref{eq:lemma4.2} with $k=2$ to $q^{\ve}$, we obtain 
\begin{equation}\label{eq:fund_d2}
\begin{aligned}
& \quad \abs{\rho^{0,\ve}(t,x) - \rho^{2,\ve}(t,x) - 2\rho^{1,\ve}(t,x)} \\
& = \abs{ q^{\ve}(0,x_0;t,x) + q^{\ve}(0,x_2;t,x) -2 q^{\ve}(0,x_1;t,x)}  \\
& =  \abs{(x_2-x_1)\tp \nx^2 q^{\ve}(0,(1-c)x_1 + cx_2;t,x) \, (x_2-x_1)}\\
& \le C t^{-(2+n)/2} \abs{x_1-x_2}^2 = C t^{-(2+n)/2} \delta^2.
\end{aligned}
\end{equation}
Therefore, 
\begin{equation*}
\begin{aligned}
& \quad \int_{\ve^{2\alpha}}^T \int_{\mX} \abs{  \inner{\nbu G(t,x,u^{\ve},-\nx V^{\ve}, -\nx^2 V^{\ve})}{\phi(t,x)}} \abs{\rho^{0,\ve} + \rho^{2,\ve} - 2\rho^{1,\ve}} \rd x \, \rd t \\
& \le C \int_{\ve^{2\alpha}}^T \int_{\mX}  \parentheses{\ve\abs{\phi(t,x)} + \abs{\nx V^{\ve} - \nx V^*} + \abs{\nx^2 V^{\ve} - \nx^2 V^*}}  \abs{\phi(t,x)} t^{-(2+n)/2} \delta^2 \,  \rd x \, \rd t \\
& \le C \delta^2 \ve^{-(2+n)\alpha} \int_{\ve^{2\alpha}}^T \int_{\mX} \parentheses{ \ve \abs{\phi(t,x)}^2 + \dfrac{1}{\ve} \abs{\nx V^{\ve} - \nx V^*}^2 + \dfrac{1}{\ve} \abs{\nx^2 V^{\ve} - \nx^2 V^*}^2  } \rd x \, \rd t \\
& \le C \delta^2 \parentheses{ \ve^{1-(2+n)\alpha} \norm{\phi}_{L^2}^2 + \ve^{-1-(2+n)\alpha} \normHtwo{V^{\ve} - V^*}^2 } \le C \delta^2 \ve^{1-(2+n)\alpha} = C \delta^2 \ve^{\alpha}.
\end{aligned}
\end{equation*}
In the first inequality, we used \eqref{eq:bound_nuGve} and \eqref{eq:fund_d2}. In the second inequality, we extract a constant $\delta^2 \ve^{-(2+n)\alpha} \ge \delta^2 t^{-(2+n)/2}$ and use Cauchy's inequality. In the last inequality, we used Lemma \ref{lem:regularity_Vu}. Therefore, \eqref{eq:lem4step3temp3} holds. To conclude, combining \eqref{eq:lem4step3temp2} and \eqref{eq:lem4step3temp3}, we recover \eqref{eq:lem4step3temp1}, which implies \eqref{eq:lem4step3temp0}. Hence, \eqref{eq:lem4step3} holds, and we finish \emph{Step 3}.

Finally, combining the three steps \eqref{eq:lem4step1}, \eqref{eq:lem4step2}, and \eqref{eq:lem4step3}, we get \eqref{eq:quadratic_Vu} and finish proving Lemma \ref{lem:quadratic_Vu}.
\end{proof}

We recall the definition of $G$ for the reader convenience:
\begin{equation*}
G(t,x,u,p,P) = \Tr\parentheses{P \, D(x,u)} + \inner{p}{b(x,u)} - r(x,u).
\end{equation*}
For a control function $u(t,x)$, we define a corresponding new function through
\begin{equation}\label{eq:udiamond}
u^{\diamond}(t,x) := \argmax_{u \in \RR^m} G(t,x,u,-\nx V_u(t,x), -\nx^2 V_u(t,x)).
\end{equation}
This $u^{\diamond}$ is well-defined because $G$ is strongly concave in $u$.
By the uniqueness of the solution to the HJB equation, $u = u^{\diamond}$ if and only if $u$ is the optimal control.
Also, using the $\mu_G$-strong concavity, we have
\begin{equation}\label{eq:strong_concave}
\begin{aligned}
& \quad \abs{\nbu G(t,x,u(t,x),-\nx V_u, -\nx^2 V_u)} \\
& = \abs{\nbu G(t,x,u(t,x),-\nx V_u, -\nx^2 V_u) - \nbu G(t,x,u^{\diamond}(t,x),-\nx V_u, -\nx^2 V_u)} \\
& \ge \mu_G \abs{u(t,x) - u^{\diamond}(t,x)}
\end{aligned}
\end{equation}
We give the following lemma regarding this implicit function
\begin{lem}[Lipschitz condition of the implicit function induced by $G$]\label{lem:implicit}
Let Assumption \ref{assump:basic} hold. Then there exists a constant $C_5>0$ such that for any two control functions $u_1, u_2 \in \mU$, and any $(t,x) \in [0,T] \times \mX$,
\begin{equation}\label{eq:implicit}
\begin{aligned}
& \quad \abs{u_1^{\diamond}(t,x) - u_2^{\diamond}(t,x)} \\
& \le C_5 \parentheses{\abs{\nx V_{u_1}(t,x) - \nx V_{u_2}(t,x)} + \abs{\nx^2 V_{u_1}(t,x) - \nx^2 V_{u_2}(t,x)}}.
\end{aligned}
\end{equation}
\end{lem}
\begin{proof}
By strong concavity of $G$ in $u$, $u^{\diamond}(t,x)$ in \eqref{eq:udiamond} is given by the equation
$$\nbu G(t,x,u^{\diamond}(t,x),-\nx V_u(t,x), -\nx^2 V_u(t,x)) = 0.$$
Therefore, for fixed $(t,x)$, we can view $u^{\diamond} = u^{\diamond}(t,x)$ as an implicit function of $p=-\nx V_u(t,x)$ and $P=-\nx^2 V_u(t,x)$.
So, \eqref{eq:implicit} is nothing but the Lipschitz condition of this implicit function. Therefore, it is sufficient to show the boundedness of the Jacobian of this implicit function.
Compute the Jacobian of $\nbu G(t,x,u^{\diamond},p,P) = 0$ w.r.t. $(p,P) \in \RR^{n+n^2}$, we obtain
\begin{equation*}
0 = \nbu^2 G(t,x,u^{\diamond},p,P) \cdot \pd{u^{\diamond}}{(p,P)} + (\nbu b(x,u),\, \nbu D(x,u)).
\end{equation*}
So
\begin{equation}\label{eq:Jacobian}
\pd{u^{\diamond}}{(p,P)} = - \parentheses{\nbu^2 G(t,x,u^{\diamond},p,P)}^{-1} (\nbu b(x,u),\, \nbu D(x,u)).
\end{equation}
By assumption \ref{assump:basic}, $\abs{\nbu b(x,u)} \le K$ and $\abs{\nbu D(x,u)} \le K$.
Since $G$ is $\mu_G$-strongly concave in $u$, $-(\nbu^2 G)^{-1}$ is positive definite with spectrum norm less than $1/\mu_G$. Therefore, \eqref{eq:Jacobian} implies
$$\abs{\pd{u^{\diamond}}{(p,P)}} \le \dfrac{2K}{\mu_G}.$$
So \eqref{eq:implicit} holds.
\end{proof}

\section{Proof for the theorems}\label{sec:thms}
Now, we are ready to prove Theorem \ref{thm:actor_converge} and \ref{thm:actor_rate}.
\begin{proof}[Proof of Theorem \ref{thm:actor_converge} and \ref{thm:actor_rate}]

\emph{Case 1.} We firstly consider an easy case where there exists positive constants $\mu_0$ and $\tau_0$ such that
\begin{equation}\label{eq:easy_case}
\norm{u^{\tau} - u^{\tau \diamond}}_{L^2} \ge \mu_0 \norm{u^{\tau} - u^*}_{L^2}
\end{equation}
for all $\tau \ge \tau_0$. Under such condition, we have
\begin{equation}\label{eq:Polyak}
\begin{aligned}
& \quad \dfrac{\rd}{\rd \tau} J[u^{\tau}] = \inner{\fd{J}{u}[u^{\tau}]}{\dfrac{\rd}{\rd \tau} u^{\tau}}_{L^2} \\
& = - \norm{\rho^{u^{\tau}}(t,x) \nbu G(t,x,u^{\tau}(t,x),-\nx V_{u^{\tau}},-\nx^2 V_{u^{\tau}})}_{L^2}^2 \\
& \le - \rho_0^2 \norm{\nbu G(t,x,u^{\tau}(t,x),-\nx V_{u^{\tau}},-\nx^2 V_{u^{\tau}})}_{L^2}^2 \le - \rho_0^2 \, \mu_G^2 \norm{u^{\tau} - u^{\tau \diamond}}_{L^2}^2 \\
& \le - \rho_0^2 \, \mu_G^2 \, \mu_0^2 \norm{u^{\tau} - u^*}_{L^2}^2 \le - \rho_0^2 \, \mu_G^2 \, \mu_0^2 \frac{1}{C_3} \parentheses{J[u^{\tau}]- J[u^*]},
\end{aligned}
\end{equation}
where we have consecutively used: chain rule; proposition \ref{prop:cost_derivative} and the control dynamic \eqref{eq:PG_ideal}; proposition \ref{prop:rho}; inequality \eqref{eq:strong_concave}; assumption \eqref{eq:easy_case}; and Lemma \ref{lem:J_quadratic} respectively. Equation \eqref{eq:Polyak} implies 
$$J[u^{\tau}] - J[u^*] \le e^{-c (\tau-\tau_0)} \parentheses{J[u^{\tau_0}] - J[u^*]}$$
holds with $c = \rho_0^2 \, \mu_G^2 \, \mu_0^2 \frac{1}{C_3}$. So \eqref{eq:actor_rate} holds. Therefore, the two theorems hold under this easy case.

\emph{Case 2.} Next, we focus on the harder case when \eqref{eq:easy_case} does not hold. Then we can find a sequence $\{\tau_k\}$, increasing to infinity, such that
\begin{equation*}
\norm{u^{\tau_k} - u^{\tau_k\diamond}}_{L^2} \le \frac1k \norm{u^{\tau_k} - u^*}_{L^2}.
\end{equation*}
For notational simplicity, we denote $u^{\tau_k}$ by $u_k$ and the corresponding value function $V_{u^{\tau_k}}$ by $V_k$. So we have
\begin{equation}\label{eq:hard_case}
\norm{u_k - u_k^{\diamond}}_{L^2} \le \frac1k \norm{u_k - u^*}_{L^2}.
\end{equation}
By Proposition \ref{prop:value_decreasing}, the value function $V_{u^{\tau}}(t,x)$ is decreasing in $\tau$, so it has a pointwise limit $V_{\infty}(t,x)$. Since $V_{u^{\tau}}(t,x) \ge V^*(t,x)$, we have $V_{\infty}(t,x) \ge V^*(t,x)$. We claim that
\begin{equation}\label{eq:claim}
V_{\infty}(t,x) \equiv V^*(t,x).
\end{equation}
The proof of this claim is quite long and technical, so we leave it to the next lemma \ref{lem:long} and focus on the rest of the proof first. With the claim holds, we know that $V_{u^{\tau}}(0,\cdot)$ converges to $V^*(0,\cdot)$ uniformly using the Lipschitz condition and Arzel\'a–Ascoli theorem. Therefore, using the relationship
$$J[u^{\tau}] = \int_{\mX} \rho^{u^{\tau}}(0,x) V_{u^{\tau}}(0,x) \,\rd x = \int_{\mX} V_{u^{\tau}}(0,x) \,\rd x ~~ \text{and} ~~ J[u^*] = \int_{\mX} V^*(0,x)\, \rd x, $$
we can show \eqref{eq:actor_convergence} and thus confirm Theorem \ref{thm:actor_converge}. 

Next, we show the convergence rate in Theorem \ref{thm:actor_rate}. By Assumption \ref{assump:actor_rate}, $$\lim_{\tau \to \infty} \norm{u^{\tau}-u^*}_{L^2} = 0,$$ hence $\lim_{k \to \infty} \norm{u_k-u^*}_{L^2} = 0$.
By Lemma \ref{lem:implicit}, we have 
$$\abs{u_k^{\diamond}(t,x) - u^*(t,x)} \le C_5 \parentheses{\abs{\nx V_k(t,x) - \nx V^*(t,x)} + \abs{\nx^2 V_k(t,x) - \nx^2 V^*(t,x)}} ,$$
hence
$$\norm{u_k^{\diamond} - u^*}_{L^2}^2 \le 2C_5^2 \parentheses{ \norm{\nx V_k - \nx V^*}_{L^2}^2 + \norm{\nx^2 V_k - \nx^2 V^*}_{L^2}^2}.$$
Therefore, by lemma \ref{lem:quadratic_Vu}, we have
$$\norm{u_k^{\diamond} - u^*}_{L^2} \le \sqrt{2} C_5 \normHtwo{V_k - V^*} \le \sqrt{2} C_5 C_4 \norm{u_k-u^*}_{L^2}^{1+\alpha}.$$
Therefore, we obtain
$$\norm{u_k-u^*}_{L^2} \le \norm{u_k - u_k^{\diamond}}_{L^2} + \norm{u_k^{\diamond} - u^*}_{L^2} \le \frac1k \norm{u_k-u^*}_{L^2} + C\norm{u_k-u^*}_{L^2}^{1+\alpha}.$$
However, this cannot hold when $k$ is sufficiently large (i.e., when $\norm{u_k-u^*}_{L^2}$ is sufficiently small) because we assume \eqref{eq:hard_case}. Therefore, the assumption \eqref{eq:hard_case} cannot hold under Assumption \ref{assump:actor_rate}. Hence \eqref{eq:easy_case} must hold and Theorem \ref{thm:actor_rate} is proved.
\end{proof}

\begin{lem}[claim \eqref{eq:claim}]\label{lem:long}
Under assumption \eqref{eq:hard_case} and all the assumptions in theorem \ref{thm:actor_converge}, \eqref{eq:claim} holds.
\end{lem}
\begin{proof}
We assume to the contrary that there exists $\btx \in [0,T] \times \mX$ s.t. $V_{\infty}\btx - V^*\btx \ge \eta >0$. This implies that
\begin{equation}\label{eq:contrary}
V_k\btx - V^*\btx \ge \eta >0 ~~~ \forall k.
\end{equation}
By the Arzel\'a–Ascoli theorem, $V_k$ converges to $V_{\infty}$ uniformly and $V_{\infty}$ (hence $V_{\infty} - V^*$) is continuous. So, we can assume $\bt > 0$.
For any $\ve,\,\delta,\,\lambda \in (0,1)$, we define two continuous functions on $(0,T] \times \mX \times (0,T] \times \mX$
\begin{equation}\label{eq:varphi}
\varphi(t,x,s,y) := \dfrac{1}{2\ve} \abs{t-s}^2 + \dfrac{1}{2\delta}\abs{x-y}^2 + \dfrac{\lambda}{t} + \dfrac{\lambda}{s}
\end{equation}
and
\begin{equation}\label{eq:Phi_k}
\Phi_k(t,x,s,y) := V_k(t,x) - V^*(s,y) - \varphi(t,x,s,y).
\end{equation}
Since the domain of $\Phi_k$ is bounded and $\lim_{t\wedge s \to 0+} \Phi_k(t,x,s,y) = - \infty$, $\Phi_k(t,x,s,y)$ achieves its maximum at some point $(t_k, x_k, s_k, y_k) \in (0,T] \times \mX \times (0,T] \times \mX$. Note that $(t_k, x_k, s_k, y_k)$ depends on $\ve,\,\delta,\,\lambda$, and $k$. Using the inequality
$$2\Phi_k(t_k, x_k, s_k, y_k) \ge \Phi_k(t_k, x_k, t_k, x_k) + \Phi_k(s_k, y_k, s_k, y_k),$$
we obtain
\begin{equation*}
\begin{aligned}
\dfrac{1}{\ve} \abs{t_k-s_k}^2 + \dfrac{1}{\delta}\abs{x_k-y_k}^2 &\le V_k(t_k,x_k) - V_k(s_k,y_k) + V^*(t_k,x_k) - V^*(s_k,y_k) \\
&\le 2L \abs{(t_k,x_k) - (s_k,y_k)},
\end{aligned}
\end{equation*}
where we used the Lipschitz condition of $V^*$ and $V_k$ in the second inequality. Therefore,
\begin{equation*}
\frac{1}{\ve + \delta}\abs{(t_k,x_k) - (s_k,y_k)}^2 = \frac{1}{\ve + \delta} \parentheses{\abs{t_k-s_k}^2 + \abs{x_k-y_k}^2} \le 2L \abs{(t_k,x_k) - (s_k,y_k)}.
\end{equation*}
Hence,
\begin{equation}\label{eq:bound1}
\abs{(t_k,x_k) - (s_k,y_k)} \le 2L(\ve + \delta)
\end{equation}
and
\begin{equation}\label{eq:bound2}
\dfrac{1}{\ve} \abs{t_k-s_k}^2 + \dfrac{1}{\delta}\abs{x_k-y_k}^2 \le 4L^2(\ve + \delta)
\end{equation}
hold. We can also see that $\abs{t_k-s_k},\,\abs{x_k-y_k} \to 0$ as $\ve,\,\delta \to 0$. 
Another direct result we have is that
\begin{align*}
& \quad V_k\btx - V^*\btx - \varphi(\bt,\bx,\bt,\bx) = \Phi_k(\bt,\bx,\bt,\bx)\\
& \le \Phi_k(t_k, x_k, s_k, y_k) = V_k(t_k, x_k) - V^*(s_k,y_k) - \varphi(t_k, x_k, s_k, y_k),
\end{align*}
which implies
\begin{equation}\label{eq:ineq1}
\quad V_k\btx - V^*\btx - \dfrac{2\lambda}{\bt}
\le V_k(t_k, x_k) - V^*(s_k,y_k) - \dfrac{\lambda}{t_k} - \dfrac{\lambda}{s_k}.
\end{equation}

Next, we separate into two cases. The idea is that: when $t_k$ or $s_k$ are close to $T$, we use the fact that $V_{\infty}(T,\cdot) = V^*(T,\cdot) = h(\cdot)$ to derive a contradiction; when $t_k$ and $s_k$ are not close to $T$, we use positivity of $\lambda$ to derive a contradiction.

\emph{Case 1.} For any $K_0$ and $\alpha>0$, we can find $k \ge K_0$ and $\ve,\delta,\lambda < \alpha$ s.t. $t_k \vee s_k \ge T - \frac{\eta}{3L}$. Under this assumption, we can find a sequence $\{k_i,\ve_i,\delta_i,\lambda_i\}_{i=1}^{\infty}$ s.t. $k_i$ increases to infinity, $(\ve_i,\delta_i,\lambda_i)$ decrease to $0$s, and the corresponding $t_{k_i}, x_{k_i}, s_{k_i}, y_{k_i}$ satisfies $t_{k_i} \vee s_{k_i} \ge T - \frac{\eta}{3L}$ for all $i$.
Since $[0,T] \times \mX$ is bounded, we can pick a subsequence of this $\{k_i, \ve_i, \delta_i, \lambda_i\}_{i=1}^{\infty}$ (without changing notations) such that $t_{k_i}, x_{k_i}, s_{k_i}, y_{k_i}$ all converge.

Let $i \to \infty$. Note that \eqref{eq:bound1} implies $s_{k_i},t_{k_i}$ converge to some same limit $t_{\infty} \ge T - \frac{\eta}{3L}$ and $x_{k_i}, y_{k_i}$ converge to some same limit $x_{\infty}$. So, \eqref{eq:ineq1} becomes
\begin{equation*}
\begin{aligned}
& \quad V_{\infty}\btx - V^*\btx 
\le V_{\infty}(t_{\infty}, x_{\infty}) - V^*(t_{\infty}, x_{\infty}) \\
& = V_{\infty}(t_{\infty}, x_{\infty}) - V_{\infty}(T, x_{\infty}) + V^*(T, x_{\infty}) - V^*(t_{\infty}, x_{\infty}) \\
& \le L \abs{T - t_{\infty}} + L \abs{T - t_{\infty}} \le 2L \dfrac{\eta}{3L} = \dfrac{2}{3} \eta < \eta,
\end{aligned}
\end{equation*}
which contradicts to \eqref{eq:contrary}.

\emph{Case 2.} There exist $K_0$ and $\alpha_0>0$ s.t. for any $k\ge K_0$ and $\ve,\delta,\lambda < \alpha_0$, we have $t_k \vee s_k < T - \frac{\eta}{3L}$. In this second case, we will only focus on the situation when $k\ge K_0$ and $\ve,\delta,\lambda < \alpha_0$. Without loss of generality, we assume $K_0 \ge 1$ and $\alpha_0 \le 1$. We fix $\lambda < \alpha_0$ and let $k, \ve, \delta$ vary. Define $M := 4K+2T+2/\bt$ and $r_0 := \min\{ \frac{\lambda}{M (M+1)}, \frac{\eta}{6L} \}$. Note that $\lambda$ is fixed, so $r_0$ is an absolute constant. We also define
$$Q_0 := \{ (t,x,s,y) ~|~ \lambda/M \le t,s \le T-2r_0 \}.$$
Then $\Phi_k$ achieves its maximum $(t_k, x_k, s_k, y_k)$ in $Q_0$ because \eqref{eq:ineq1} cannot hold if $t_k \le \lambda/M$ or $s_k \le \lambda/M$. Next, we define
\begin{equation}\label{eq:Q_set}
Q := \{ (t,x,s,y) ~|~ \lambda/(M+1) < t,s < T-r_0 \}.
\end{equation}
We find $Q_0 \subset Q$. Also,
\begin{equation*}
t_k - \dfrac{\lambda}{M+1} \ge \dfrac{\lambda}{M} - \dfrac{\lambda}{M+1} = \dfrac{\lambda}{M(M+1)} \ge r_0,
\end{equation*}
and $T - r_0 - t_k \ge r_0$. $s_k$ also satisfies the two inequalities. Restricted in $Q$, $\Phi_k$ has bounded derivatives. Next, we want to show that the maximum is still in $Q$ if we make a small perturbation on $\Phi_k$.

We define $\mu>0$ by
\begin{equation}\label{eq:mu_def}
2\mu(K+1) + \mu K^2 = \dfrac{\lambda}{2T^2}.
\end{equation}
Let $r_1 := \mu r_0 / 4$. We pick $(q,p,\hq,\hp) \in \RR^{1+n+1+n}$ s.t. $\abs{p},\abs{q},\abs{\hp},\abs{\hq} \le r_1$. Then we define a new function
\begin{equation}\label{eq:hPhik}
\begin{aligned}
    \hPhik(t,x,s,y) =& \Phi_k(t,x,s,y) - \dfrac{\mu}{2}\parentheses{\abs{t-t_k}^2 + \abs{x-x_k}^2 + \abs{s-s_k}^2 + \abs{y-y_k}^2}\\
    &+ q(t-t_k) + \inner{p}{x-x_k} + \hq(s-s_k) + \inner{\hp}{y-y_k}.
\end{aligned}
\end{equation}
If we do not have the second line in \eqref{eq:hPhik}, then $\hPhik$ achieves a strict maximum at $(t_k, x_k, s_k, y_k)$. This second line can be viewed as a linear perturbation. 
So, $\hPhik$ achieves a maximum at some other point in $\RR^{1+n+1+n}$, denoted by $(\htk,\hxk,\hsk,\hyk)$. By this optimality, $(\htk,\hxk,\hsk,\hyk)$ must lie in the set
\begin{align*}
& \left\{ (t,x,s,y) ~\Big|~ \dfrac{\mu}{2}\parentheses{\abs{t-t_k}^2 + \abs{x-x_k}^2 + \abs{s-s_k}^2 + \abs{y-y_k}^2} \right.\\
& \hspace{0.8in} \left. \le q(t-t_k) + \inner{p}{x-x_k} + \hq(s-s_k) + \inner{\hp}{y-y_k} \right\}.
\end{align*}
So,
$$\dfrac{\mu}{2} \abs{ (\htk,\hxk,\hsk,\hyk) - (t_k, x_k, s_k, y_k) }^2 \le \abs{ (\htk,\hxk,\hsk,\hyk) - (t_k, x_k, s_k, y_k) } \cdot \abs{(q,p,\hq,\hp) }.$$
Therefore,
$$\abs{ (\htk,\hxk,\hsk,\hyk) - (t_k, x_k, s_k, y_k) } \le \dfrac{2}{\mu} \abs{(q,p,\hq,\hp) } \le \dfrac{2}{\mu} ~ 2r_1 = r_0.$$
So $\abs{\htk - t_k}, \abs{\hsk - s_k} \le r_0$, which implies $(\htk,\hxk,\hsk,\hyk) \in Q$. More importantly, $(\htk,\hxk,\hsk,\hyk)$ lies in the interior of $(0,T] \times \mX \times (0,T] \times \mX$. So, by the optimality of $(\htk,\hxk,\hsk,\hyk)$, we have
\begin{equation}\label{eq:optimality_hPhik}
\left\{ \begin{aligned}
&0 = \pt\, \hPhik(\htk,\hxk,\hsk,\hyk) = \pt V_k(\htk,\hxk) - \pt\,  \varphi(\htk,\hxk,\hsk,\hyk) - \mu(\htk-t_k) + q\\
&0 = \ps\, \hPhik(\htk,\hxk,\hsk,\hyk) = -\ps V^*(\hsk,\hyk) - \ps\,  \varphi(\htk,\hxk,\hsk,\hyk) - \mu(\hsk-s_k) + \hq\\
&0 = \nx \hPhik(\htk,\hxk,\hsk,\hyk) = \nx V_k(\htk,\hxk) - \nx  \varphi(\htk,\hxk,\hsk,\hyk) - \mu(\hxk-x_k) + p\\
&0 = \ny \hPhik(\htk,\hxk,\hsk,\hyk) = -\ny V^*(\hsk,\hyk) - \ny  \varphi(\htk,\hxk,\hsk,\hyk) - \mu(\hyk-y_k) + \hp\\
&\begin{pmatrix} \nx^2V_k(\htk,\hxk) & 0 \\0 & -\ny^2V^*(\hsk,\hyk) \end{pmatrix} \le 
\nabla_{x,y}^2 \varphi
+ \mu I_{2n} = \dfrac{1}{\delta}\begin{pmatrix} I_n & -I_n \\ -I_n  & I_n \end{pmatrix} + \mu I_{2n}
\end{aligned} \right.
\end{equation}
as first and second order necessary conditions. Note that $(\htk,\hxk,\hsk,\hyk)$ depend on $\ve$, $\delta$, $\lambda$, $q$, $p$, $\hq$, $\hp$, $\mu$, and $k$. Also, recall that $\lambda$, and $\mu$ are fixed.

For given $\ve$, $\delta$ and $k$, we can view $(\htk,\hxk,\hsk,\hyk)$ as an implicit function of $(q,p,\hq,\hp)$, given by the equation $\nabla \hPhik(\htk,\hxk,\hsk,\hyk) = 0$, i.e.,
\begin{equation}\label{eq:implicit_eqn}
(q,p,\hq,\hp) = -\nabla \Phi_k(\htk,\hxk,\hsk,\hyk) + \mu(\htk-t_k, \hxk-x_k, \hsk-s_k, \hyk-y_k).
\end{equation}
Here, the gradient is taken w.r.t. $(t,x,s,y)$.
We claim that we can consider it inversely and view $(q,p,\hq,\hp)$ as an implicit function of $(\htk,\hxk,\hsk,\hyk)$ locally, still given by \eqref{eq:implicit_eqn}. 
The Jacobian of this implicit function is given by
\begin{equation}\label{eq:Ak}
A_k := \pd{(q,p,\hq,\hp)}{(\htk,\hxk,\hsk,\hyk)} = \mu I_{2n+2} - \nabla^2 \Phi_k(\htk,\hxk,\hsk,\hyk).
\end{equation}
We will show the claim by proving that $A_k$ is nonsingular locally.

 Let us restrict
\begin{equation}\label{eq:range_hat}
\abs{\htk-t_k}, \abs{\hxk-x_k}, \abs{\hsk-s_k}, \abs{\hyk-y_k} < r_2
\end{equation}
where
\begin{equation}\label{eq:range_r}
0 < r_2 < \mu / \sqbra{8\parentheses{L + 3(M+1)^4 / \lambda^3}}.
\end{equation}
Later, this $r_2$ will change according to $\ve$ and $\delta$, but will be independent with $k$. We give an estimate next.
\begin{equation}\label{eq:diff_d2Phi}
\begin{aligned}
& \quad \abs{ \nabla^2 \Phi_k(t_k, x_k, s_k, y_k) - \nabla^2 \Phi_k(\htk,\hxk,\hsk,\hyk) } \\
& \le \abs{\nabla^2 V_k(\htk, \hxk) - \nabla^2 V_k(t_k, x_k)} + \abs{\nabla^2 V^*(\hsk, \hyk) - \nabla^2 V^*(s_k, y_k)} \\
& \quad + 2\lambda \parentheses{ \abs{\htk^{-3} - t_k^{-3}} + \abs{\hsk^{-3} - s_k^{-3}} } \\
& \le L \abs{(\htk, \hxk) - (t_k, x_k)} + L \abs{(\hsk, \hyk) - (s_k, y_k)} \\
& \quad + 2\lambda \parentheses{ 3\abs{\htk - t_k} \parentheses{\min\{ \htk, t_k \}}^{-4} + 3\abs{\hsk - s_k} \parentheses{\min\{ \hsk, s_k \}}^{-4}}\\
& \le 4Lr_2 + 12\lambda r_2 \parentheses{ \lambda / (M+1) }^{-4} \le \dfrac{1}{2} \mu.
\end{aligned}
\end{equation}
Here, we used the the definition of $\varphi$ \eqref{eq:varphi} and $\Phi_k$ \eqref{eq:Phi_k} in the first inequality.
The third inequality is due to the range of $\htk$, $\hsk$, $t_k$, $s_k$ given by \eqref{eq:Q_set}. The fourth inequality comes from the range for $r_2$ in \eqref{eq:range_r}. In the second inequality, we used the Lipschitz condition for the derivatives of the value functions and a mean value theorem. Note that the $\nabla$ in \eqref{eq:diff_d2Phi} operates on all the inputs, so we also used the Lipschitz condition of $\pt^2 V_k(t,x)$ (and $\ps^2 V^*(s,y)$). If we take derivative of the HJ equation \eqref{eq:HJ} w.r.t. $t$, we get
\begin{equation}\label{eq:V_tt}
\begin{aligned}
\pt^2 V_k(t,x) & = - \pt \Tr\parentheses{D(x,u(t,x)) \nx^2 V_u(t,x)} \\
& \quad - \pt \inner{b(x,u(t,x))}{\nx V_u(t,x)} - \pt r(x,u(t,x)).
\end{aligned}
\end{equation}
Expanding the right hand side of \eqref{eq:V_tt} with chain rule and product rule, we find that each term is bounded and Lipschitz in $t$ and $x$, so $\pt^2 V_k(t,x)$ (and $\ps^2 V^*(s,y)$) is Lipschitz. We also remark that this part \eqref{eq:diff_d2Phi} makes the analysis not easy to generalize to the viscosity solution of the HJB equation, which does not have sufficient regularity in general. Therefore,
\begin{equation*}
\begin{aligned}
& \quad A_k = \mu I_{2n+2} - \nabla^2 \Phi_k(\htk, \hxk, \hsk, \hyk) \\
& = \mu I_{2n+2} - \nabla^2 \Phi_k(t_k, x_k, s_k, y_k) + \parentheses{\nabla^2 \Phi_k(t_k, x_k, s_k, y_k) - \nabla^2 \Phi_k(\htk, \hxk, \hsk, \hyk)} \\
& \ge \mu I_{2n+2} - \abs{\nabla^2 \Phi_k(t_k, x_k, s_k, y_k) - \nabla^2 \Phi_k(\htk,\hxk,\hsk,\hyk)} \cdot I_{2n+2} \ge \dfrac{1}{2} \mu I_{2n+2}.
\end{aligned}
\end{equation*}
Here, the inequality $\ge$ between two symmetric matrix means that their difference is positive semi-definite. In the first inequality, we use the fact that $\nabla^2 \Phi_k(t_k, x_k, s_k, y_k) \le 0$, coming from the optimality of $(t_k, x_k, s_k, y_k)$. In the second equality, we use the estimate \eqref{eq:diff_d2Phi}. 
Therefore, the Jacobian $A_k$ in \eqref{eq:Ak} always nonsingular when \eqref{eq:range_r} holds and we confirm the claim after \eqref{eq:implicit_eqn}.

Next, we also want to derive an upper bound for this Jacobian. A direct calculation from \eqref{eq:Ak} gives us
\begin{equation*}
\begin{aligned}
\norm{A_k}_2 \le & \mu + \norm{\nx^2V_k(\htk,\hxk)}_2 + \norm{\ny^2V^*(\hsk,\hyk)}_2\\
& + \dfrac{1}{\ve} \norm{ \begin{pmatrix} 1 & -1 \\-1 & 1 \end{pmatrix}}_2 + \dfrac{1}{\delta} \norm{\begin{pmatrix} I_n & -I_n \\-I_n & I_n \end{pmatrix}}_2 + \dfrac{4\lambda}{(\lambda/M)^2}\\
\le & \mu + 2K + \dfrac{2}{\ve} + \dfrac{2}{\delta} + \dfrac{4M^2}{\lambda} \le C \parentheses{\dfrac{1}{\ve} + \dfrac{1}{\delta}}.
\end{aligned}
\end{equation*}
Note that the notation $\norm{\cdot}_2$ is the matrix norm (instead of the Frobenius norm). The first inequality above is by definition of $\varphi$ \eqref{eq:varphi} and $\Phi_k$ \eqref{eq:Phi_k}. The second is by boundedness of the derivatives of the value functions. The third is because $\lambda$ is fixed while $\ve$ and $\delta$ are small and are going to $0$s later. Therefore, we obtain an estimate 
\begin{equation}\label{eq:regularity_p}
\abs{(q,p,\hq,\hp)} \le C \parentheses{\dfrac{1}{\ve} + \dfrac{1}{\delta}} \abs{(\htk-t_k, \hxk-x_k, \hsk-s_k, \hyk-y_k)}.
\end{equation}
Therefore, we also require that 
\begin{equation}\label{eq:range_r2}
r_2 \le \parentheses{\dfrac{1}{\ve} + \dfrac{1}{\delta}}^{-1} \dfrac{r_1}{2C}
\end{equation}
where the $C$ in \eqref{eq:range_r2} is the same as the $C$ in \eqref{eq:regularity_p}, in order to guarantee 
$\abs{(q,p,\hq,\hp)} \le r_1$. Now we can see that $r_2$ depends on $\ve$ and $\delta$, but it is independent of $k$.

Next, we consider the quantity
\begin{equation*}
B_k := \ps V^*(\hsk,\hyk) - \pt V_k(\htk,\hxk),
\end{equation*}
which depends on $\ve,\,\delta,\,\lambda,\,q,\,p,\,\hq,\,\hp,\,\mu,$ and $k$. On the one hand, by the optimality condition \eqref{eq:optimality_hPhik},
\begin{equation}\label{eq:one_hand}
\begin{aligned}
B_k &= -\ps \varphi(\htk,\hxk,\hsk,\hyk) -\pt \varphi(\htk,\hxk,\hsk,\hyk) - \mu \sqbra{(\hsk-s_k) + (\htk-t_k)} + q+\hq \\
& = \lambda/\htk^{~2} + \lambda/\hsk^{~2} - \mu \sqbra{(\hsk-s_k) + (\htk-t_k)} + q+\hq \\
& \ge 2\lambda/T^2 - \mu \parentheses{\abs{\hsk-s_k} + \abs{\htk-t_k}} + q+\hq,
\end{aligned}
\end{equation}
where the terms with $\ve$ in $\ps \varphi$ and $\pt \varphi$ cancel each other.

On the other hand, using the HJ equations that $V^*$ and $V_k$ satisfy, we have
\begin{equation*}
\begin{aligned}
B_k &= G(\hsk,\hyk,u^*(\hsk,\hyk),-\ny V^*,-\ny^2 V^*) - G(\htk,\hxk,u_k(\htk,\hxk),-\nx V_k,-\nx^2 V_k) \\
& = \sup_{u} G(\hsk,\hyk,u,-\ny V^*,-\ny^2 V^*) - G(\htk,\hxk,u_k(\htk,\hxk),-\nx V_k,-\nx^2 V_k) \\
& \le \sup_{u} G(\hsk,\hyk,u,-\ny V^*,-\ny^2 V^*) - \sup_{u} G(\htk,\hxk,u,-\nx V_k,-\nx^2 V_k) \\
& \quad + L \abs{u_k(\htk,\hxk) - u_k^{\diamond}(\htk,\hxk)}\\
& \le \sup_{u} \sqbra{G(\hsk,\hyk,u,-\ny V^*,-\ny^2 V^*) - G(\htk,\hxk,u,-\nx V_k,-\nx^2 V_k)}\\
& \quad + L \abs{u_k(\htk,\hxk) - u_k^{\diamond}(\htk,\hxk)},
\end{aligned}
\end{equation*}
where we have consecutively used: HJ equations for $V^*$ and $V_k$; the optimality condition \eqref{eq:max1} for $u^*$; the definition of $u_k^{\diamond}$ \eqref{eq:udiamond} and the Lipschitz condition of $G$ in $u$; a simple inequality. Therefore, by the definition of $G$ \eqref{eq:generalized_Hamiltonian}, 
\begin{equation}\label{eq:other_hand}
\begin{aligned}
B_k & \le \sup_{u} \left\{ \frac12 \Tr \sqbra{ \nx^2V_k(\htk,\hxk)\sigma\sigma\tp(\hxk,u) - \ny^2 V^*(\hsk,\hyk) \sigma\sigma\tp(\hyk,u)} \right.\\
& \quad +\sqbra{ \inner{\nx V_k(\htk,\hxk)}{b(\hxk,u)} - \inner{\ny V^*(\hsk,\hyk)}{b(\hyk,u)} } \\
& \quad + r(\hxk,u) - r(\hyk,u) \bigg\} + L \abs{u_k(\htk,\hxk) - u_k^{\diamond}(\htk,\hxk)}\\
& =: \sup_{u} \curlybra{(\rom{1}) + (\rom{2}) + (\rom{3})} + L \abs{u_k(\htk,\hxk) - u_k^{\diamond}(\htk,\hxk)}.
\end{aligned}
\end{equation}

Next, we bound the three terms in \eqref{eq:other_hand}. Using the estimates \eqref{eq:bound1} and \eqref{eq:bound2}, we can easily show that
\begin{equation}\label{eq:bound3}
\abs{\hxk-\hyk} \le \abs{\hxk-x_k} + \abs{\hyk-y_k} + \abs{x_k - y_k} \le 2r_2 + 2L(\ve + \delta)
\end{equation}
and
\begin{equation}\label{eq:bound4}
\dfrac{1}{\delta}\abs{\hxk-\hyk}^2 \le \dfrac{1}{\delta} \parentheses{8r_2^2 + 2\abs{x_k - y_k}^2} \le \dfrac{8r_2^2}{\delta} + 8L^2(\ve + \delta)
\end{equation}

For $(\rom{3})$, we have
\begin{equation}\label{eq:term3}
(\rom{3}) = r(\hxk,u) - r(\hyk,u) \le L \abs{\hxk -\hyk} \le 2Lr_2 + 2L^2(\ve+\delta),
\end{equation}
where we used Lipschitz condition of $r$ in Assumption \ref{assump:basic} and \eqref{eq:bound3}.

For $(\rom{2})$, we have
\begin{equation}\label{eq:term2}
\begin{aligned}
(\rom{2}) &= \inner{\nx \varphi(\htk,\hxk,\hsk,\hyk) + \mu(\hxk-x_k) -p}{b(\hxk,u)}\\
& \quad + \inner{\ny \varphi(\htk,\hxk,\hsk,\hyk) + \mu(\hyk-y_k) -\hp}{b(\hyk,u)}\\
& = \inner{\dfrac{1}{\delta} (\hxk-\hyk) + \mu(\hxk-x_k) -p}{b(\hxk,u)} \\
& \quad + \inner{ \dfrac{1}{\delta} (\hyk-\hxk) + \mu(\hyk-y_k) -\hp}{b(\hyk,u)}\\
& \le \dfrac{L}{\delta} \abs{\hxk-\hyk}^2 + \mu K(\abs{\hxk-x_k} + \abs{\hyk-y_k}) + K(\abs{p}+\abs{\hp})\\
& \le 8r_2^2L/\delta + 8L^3(\ve+\delta) +  + \mu K(\abs{\hxk-x_k} + \abs{\hyk-y_k}) + K(\abs{p}+\abs{\hp}),
\end{aligned}
\end{equation}
where we have consecutively used: the optimality condition \eqref{eq:optimality_hPhik}; the definition of $\varphi$ in \eqref{eq:varphi}; boundness and Lipschitz condition of $b$ in Assumption \ref{assump:basic}; the bound \eqref{eq:bound4}.

For $(\rom{1})$, we have
\begin{equation}\label{eq:term1}
\begin{aligned}
(\rom{1}) &= \dfrac12 \Tr \sqbra{ \begin{pmatrix}  \sigma(\hxk,u) \\ \sigma(\hyk,u) \end{pmatrix}\tp \begin{pmatrix} \nx^2V_k(\htk,\hxk) &0 \\ 0& -\ny^2V^*(\hsk,\hyk) \end{pmatrix} \begin{pmatrix}  \sigma(\hxk,u) \\ \sigma(\hyk,u) \end{pmatrix} }\\
&\le \dfrac12 \Tr \sqbra{ \begin{pmatrix}  \sigma(\hxk,u) \\ \sigma(\hyk,u) \end{pmatrix}\tp \parentheses{\dfrac{1}{\delta}\begin{pmatrix} I_n & -I_n \\ -I_n & I_n \end{pmatrix} + \mu I_{2n}} \begin{pmatrix}  \sigma(\hxk,u) \\ \sigma(\hyk,u) \end{pmatrix} }\\
&= \dfrac{1}{2\delta} \abs{\sigma(\hxk,u) - \sigma(\hyk,u)}^2 + \dfrac{\mu}{2} \parentheses{\abs{\sigma(\hxk,u)}^2 + \abs{\sigma(\hyk,u)}^2}\\
& \le \dfrac{L^2}{2\delta} \abs{\hxk-\hyk}^2 + \mu K^2 \le 4r_2^2L^2/\delta + 4L^4(\ve + \delta) + \mu K^2,
\end{aligned}
\end{equation}
where we have consecutively used: a simple transform in linear algebra; the second order optimality condition in \eqref{eq:optimality_hPhik}; a simple calculation; boundness and Lipschitz condition of $\sigma$ in Assumption \ref{assump:basic}; the bound \eqref{eq:bound4}.

Combining \eqref{eq:one_hand}, \eqref{eq:other_hand}, \eqref{eq:term3}, \eqref{eq:term2}, and \eqref{eq:term1}, we obtain
\begin{equation*}
\begin{aligned}
& \quad 2\lambda/T^2 - \mu \parentheses{\abs{\hsk-s_k} + \abs{\htk-t_k}} + q+\hq \\
& \le 4r_2^2L^2/\delta + 4L^4(\ve + \delta) + \mu K^2 +8r_2^2L/\delta + 8L^3(\ve+\delta) + \mu K\parentheses{\abs{\hxk-x_k} + \abs{\hyk-y_k}} \\
& \quad + K(\abs{p}+\abs{\hp}) + 2Lr_2 + 2L^2(\ve+\delta) + L \abs{u_k(\htk,\hxk) - u_k^{\diamond}(\htk,\hxk)}.
\end{aligned}
\end{equation*}
which simplifies to
\begin{equation}\label{eq:ineq2}
\begin{aligned}
2\lambda/T^2  \le 2\mu(K+1)r_2 + \mu K^2 + (4L^4+8L^3+2L^2)(\ve+\delta) + (4r_2^2L^2+8r_2^2L)/\delta \\
+ (2K+2)C\parentheses{\dfrac{1}{\ve} + \dfrac{1}{\delta}}r_2 + L \abs{u_k(\htk,\hxk) - u_k^{\diamond}(\htk,\hxk)}
\end{aligned}
\end{equation}
with the help of the bounds \eqref{eq:regularity_p} and \eqref{eq:range_hat}.
Next, we pick a box in $\RR^{2n+2}$ that centered at $(t_k, x_k, s_k, y_k)$ and have side length $r_3 = 2r_2/\sqrt{n}$. Then $\abs{\hxk - x_k}, \abs{\hyk - y_k} \le \sqrt{n}r_3/2 = r_2$, so that all the estimates before hold in this box. If we integrate \eqref{eq:ineq2} over the box w.r.t. $(\htk,\hxk,\hsk,\hyk)$ and divided it by $r_3^{2n+2}$, we obtain
\begin{equation}\label{eq:ineq3}
\begin{aligned}
2\lambda/T^2  &\le 2\mu(K+1)r_2 + \mu K^2 + (4L^4+8L^3+2L^2)(\ve+\delta) + (4r_2^2L^2+8r_2^2L)/\delta\\
& \quad + (2K+2)C\parentheses{\dfrac{1}{\ve} + \dfrac{1}{\delta}}r_2 +  (2r_2/\sqrt{n})^{-2n-2} L \norm{u_k- u_k^{\diamond}}_{L^1}.
\end{aligned}
\end{equation}
We recall that $\lambda$ are fixed at first. We also recall that the definition of $\mu$ in \eqref{eq:mu_def} ensures that
\begin{equation}\label{eq:contra1}
2\mu(K+1)r_2 + \mu K^2 \le \dfrac{\lambda}{2T^2}.
\end{equation}
Therefore, if we firstly set $\ve$ and $\delta$ to be small such that 
\begin{equation}\label{eq:contra2}
(4L^4+8L^3+2L^2)(\ve + \delta) < \dfrac{\lambda}{2T^2}.
\end{equation}
Then we set $r_2$ to be small such that
\begin{equation}\label{eq:contra3}
(4r_2^2L^2+8r_2^2L)/\delta + (2K+2)C\parentheses{\dfrac{1}{\ve} + \dfrac{1}{\delta}}r_2 < \dfrac{\lambda}{2T^2},
\end{equation}
where the $C$ in \eqref{eq:contra3} is the same as the $C$ in \eqref{eq:ineq3}. Next, note that $\norm{u_k- u_k^{\diamond}}_{L^1} \le \sqrt{T} \norm{u_k- u_k^{\diamond}}_{L^2}$. So, by \eqref{eq:hard_case}, we can set $k$ to be large enough such that 
\begin{equation}\label{eq:contra4}
(2r_2/\sqrt{n})^{-2n-2} L \norm{u_k- u_k^{\diamond}}_{L^1} < \dfrac{\lambda}{2T^2},
\end{equation}
Finally, substituting \eqref{eq:contra1}, \eqref{eq:contra2}, \eqref{eq:contra3}, and \eqref{eq:contra4} into \eqref{eq:ineq3}, we obtain an contradiction, so Lemma \ref{lem:long} is proved.
\end{proof}

\end{document}